\documentclass[10pt]{article}

\usepackage{amsfonts}
\usepackage{amsmath}
\usepackage{amssymb}
\usepackage{amsthm}
\usepackage{bm}
\usepackage{enumerate}
\usepackage{color}

\newcommand{\Z}{\mathbb{Z}}

\newcommand{\C}{\mathbb{C}}
\newcommand{\s}{\sigma}
\newcommand{\D}{\partial}
\newcommand{\op}{\oplus}
\newcommand{\bu}{\bm{u}}
\newcommand{\bv}{\bm{v}}
\newcommand{\bw}{\bm{w}}

\newcommand{\T}{\mathcal{T}}
\renewcommand{\atop}[2]{\genfrac{}{}{0pt}{}{#1}{#2}}

\newtheorem{theorem}{Theorem}[section] 
\newtheorem{lemma}[theorem]{Lemma}     

\newtheorem{definition}[theorem]{Definition}

\begin{document}

\title{Abelian functions associated with genus three algebraic curves}

\author{J.C.~Eilbeck, M.~England and Y.~\^Onishi}

\date{July 2011}

\maketitle

\begin{abstract}
We develop the theory of Abelian functions associated with algebraic curves.  The growth in computer power and an advancement of efficient symbolic computation techniques has allowed for recent progress in this area.  In this paper we focus on the genus three cases, comparing the two canonical classes of hyperelliptic and trigonal curves.  We present new addition formulae, derive bases for the spaces of Abelian functions and discuss the differential equations such functions satisfy.
\end{abstract}

\section{Introduction} \label{SEC_Intro}

In this paper we present addition formulae and differential equations
satisfied by the Abelian functions with poles along the standard theta
divisor, associated with the hyperelliptic and trigonal curves of
genus three.  Alongside the presentation of the new results, this
paper offers a comparison of the formulae for the two canonical types
of genus three curves.

The most important new results are the addition formulae presented in
Theorems \ref{thm:34_3t3v} and \ref{thm:34_4t2v}.  The first of these
was obtained by following the ideas in \cite{eemop07} while the second
requires the explicit derivation of a basis for the vector space of
Abelian functions having poles of order at most four along the
standard theta divisor.  They both follow from the recent work in
\cite{emo11} which introduced new addition formulae for the
Weierstrass $\wp$ and $\sigma$-functions along with generalisations to
genus two.  We also present and discuss the sets of differential
equations satisfied by the functions.  These results can be seen as a
continuation of the work started in \cite{eemop07} and \cite{bel97}
using new efficient computational techniques first introduced in
\cite{MEe09}.

Curves of genus three can be categorised as either hyperelliptic or
trigonal.  We study two canonical examples, the $(2,7)$ and the
$(3,4)$-curves which are hyperelliptic and trigonal respectively.  We
define these terms and the curves formally in the next section.  We
develop the results of \cite{eemop07} for the $(3,4)$-curve,
completing the key sets of differential equations and deriving new
addition formulae.  We compare these results with the corresponding
genus three hyperelliptic results for the $(2,7)$-curve, furthering
the results in \cite{bel97} for this case.

Many of the results we present can be viewed as generalisations of
classic results for elliptic functions.  We will conclude the
introduction below by reminding the reader of the relevant results
from the theory of Weierstrass functions.  Then in Section
\ref{SEC_Construct} we give the definitions of the Abelian functions,
discussing the general properties they satisfy.  We proceed in Section
\ref{SEC_Bases} to consider the problem of determining bases for the
vector spaces of such functions, presenting explicit constructions of
these for the genus three curves.  In Section \ref{SEC_DE} we derive
sets of differential equations satisfied by the functions and compare
the two canonical genus three cases.  Finally is Section \ref{SEC_AF}
we present the new addition formulae.

The classical results for Weierstrass elliptic function form a template for our theory.  Let $\wp(u)$ be the \textit{Weierstrass $\wp$-function}, which as an elliptic function has two complex periods $\omega_1,\omega_2$:
\begin{equation} \label{eq:Intro_period}
\wp(u + \omega_1) = \wp(u + \omega_2) = \wp(u), \qquad \text{for all } \, u \in \C.
\end{equation}
The $\wp$-function has the simplest possible pole structure for an elliptic function and satisfies a number of interesting properties.  For example, it can be used to parametrise an elliptic curve,
\begin{equation} \label{eq:Intro_ec}
y^2 = 4x^3 - g_2x - g_3,
\end{equation}
where $g_2$ and $g_3$ are constants.  It also satisfies the following well-known differential equations,
\begin{align}
\big(\wp'(u)\big)^2 &= 4\wp(u)^3 - g_2\wp(u) - g_3,            \label{eq:Intro_elliptic_diff1}     \\
           \wp''(u) &= 6\wp(u)^2 - \textstyle \frac{1}{2}g_2.  \label{eq:Intro_elliptic_diff2}
\end{align}
Weierstrass introduced an auxiliary function, $\sigma(u)$, in his theory which satisfies
\begin{align}
\wp(u) &= - \frac{d^2}{d u^2} \log \big[ \sigma(u) \big].  \label{eq:Intro_elliptic_ps}
\end{align}
The $\sigma$-function plays a crucial role in the generalisation and in applications of the theory.  It satisfies the following two term addition formula,
\begin{eqnarray}
- \frac{\sigma(u+v)\sigma(u-v)}{\sigma(u)^2\sigma(v)^2} = \wp(u) - \wp(v). \label{eq:Intro_elliptic_add}
\end{eqnarray}
In this document we present generalisations of equations
(\ref{eq:Intro_period})$-$(\ref{eq:Intro_elliptic_add}) for higher genus functions.

Elliptic functions have been the subject of much study since their
discovery and have been extensively used to enumerate solutions of
non-linear wave equations.  Recent times have seen a revival of
interest in the theory of Abelian functions, which have multiple
independent periods, and so generalise the elliptic functions.  The
periodicity property is usually defined in association with an
underlying algebraic curve.  These functions have been shown to solve
differential equations arising in mathematical physics
and have been used in a variety of applications.

\section{Constructing Abelian functions} \label{SEC_Construct}

In this paper we study in detail the case of functions associated with genus three curves.  However, much of the theory is applicable to a wider space of curves and so we will include some general definitions.

\begin{definition} \label{def:HG_general_curves}
For two coprime integers $(n,s)$ with $s>n$ we define an \textbf{$\bm{(n,s)}$-curve} as a non-singular algebraic curve defined by $f(x,y)=0$, where
\begin{equation} \label{eq:general_curve}
f(x,y) = y^n + p_1(x)y^{n-1} + p_{2}(x)y^{n-2} + \cdots + p_{n-1}(x)y - p_{n}(x).
\end{equation}
Here $x$, $y$ are complex variables and $p_j(x)$ are polynomials in $x$ of degree (at most)
$\lfloor{ js } / n \rfloor$.
We define a simple subclass of the curves by setting $p_j(x)=0$ for $0\leq j \leq n-1$.  Such curves are then defined by
\begin{equation} \label{eq:ct_curve}
f(x,y) = y^n - (x^s +\lambda_{s-1}x^{s-1}+\dots+\lambda_{1}x+\lambda_0)
\end{equation}
and are called \textbf{cyclic $\bm{(n,s)}$-curves}.  We follow tradition and denote the curve constants by $\lambda_j$ for the cyclic $(n,s)$-curves and by $\mu_j$ for the general $(n,s)$-curves.
Note that in the literature the word \lq\lq cyclic" is sometimes replaced by \emph{\lq\lq strictly"} or
\emph{\lq\lq purely (\(n\)-gonal)"}.  Also, some authors use a different normalisation with the coefficient of $x^s$ chosen to be $4$ instead of $1$.  It is simple to move between the different normalisations.
\end{definition}
We denote the surface defined by such a curve as $C$.  The genus of $C$ is given by
\[
g=\tfrac{1}{2}(n-1)(s-1)
\]
and the associated functions to be defined shortly will be multivariate with $g$ variables, $\bm{u} = (u_1, \dots, u_g)$.  For example, the elliptic curve in equation (\ref{eq:Intro_ec}) is a (2,3)-curve and the Weierstrass $\sigma$ and $\wp$-functions depend upon a single variable $u$.

Given this equation for the genus we see that it can only be three when $(n,s)=(2,7)$ or $(3,4)$.  Hence the two curves we will investigate in this paper are as follows:
\begin{align}
&y^2+(\mu_1x^3 + \mu_3x^2+\mu_5x+\mu_7)y = x^7+\mu_2x^6+\mu_4x^5+\mu_6x^4+\mu_8x^3+\mu_{10}x^2+\mu_{12}x+\mu_{14},
\label{eq:27}\\
&y^3+(\mu_1x+\mu_4)y^2+(\mu_2x^2+\mu_5x+\mu_8)y = x^4+\mu_3x^3+\mu_6x^2+\mu_9x+\mu_{12}.
\label{eq:34}
\end{align}
Their cyclic restrictions are given by:
\begin{align}
y^2 &= x^7 + \lambda_6x^6 + \lambda_5x^5 + \lambda_4x^4 + \lambda_3x^3 + \lambda_2x^2 + \lambda_1x + \lambda_0,
\label{eq:c27}\\
y^3 &= x^4 + \lambda_3x^3 + \lambda_2x^2 + \lambda_1x + \lambda_0.
\label{eq:c34}
\end{align}
We define a weight for the theory, denoted by $\mathrm{wt}$ and called the {\it Sato weight}.  We start with
\begin{equation}  \label{eq:Sato_weight}
\mathrm{wt}(x)=-n, \ \ \mathrm{wt}(y)=-s
\end{equation}
and then choose the weights of the curve parameters to be such that the
curve equation is homogeneous.  We see that for cyclic curves this
imposes $\mathrm{wt}( \lambda_j )=-n(s-j)$ while for the non-cyclic
curves (\ref{eq:27}) and (\ref{eq:34}) the constants $\mu_j$ have been
labeled with the absolute value of their weights.  The weights of all
other objects may then be derived uniquely and will be commented on
throughout the paper.  Note that all the objects in this paper have a
definite Sato weight and all the equalities are homogeneous with
respect to this weight.  A more detailed discussion of the weight properties may be found for example in \cite{MEe09}.

The $(n,s)$-curves with $n=2$ are generally defined to be \emph{hyperelliptic curves}.  Klein developed an approach to generalise the Weierstrass $\wp$-function to Abelian functions associated with hyperelliptic curves as described in Baker's classic texts \cite{ba97} and \cite{ba07}.   This approach has motivated the general definitions in \cite{bel97} and \cite{eel00} of what we now call \emph{Kleinian $\wp$-functions}.  It is the properties of these and the generalised $\sigma$-function that are our objects of study.

In the last few years a good deal of progress has been made on the theory of Abelian functions associated to those $(n,s)$ curve with $n=3$, which we label \textit{trigonal curves}.  In \cite{bel00}, the authors of \cite{bel97} furthered their methods to the trigonal cases, obtaining realisations of the Jacobian variety and some key differential equations between the functions.  More recently the two canonical trigonal cases of the (3,4) and (3,5)-curves have been examined in \cite{eemop07} and \cite{bego08} respectively.   Both papers explicitly construct the differentials on the curve, solve the Jacobi inversion problem and obtain differential equations between the $\wp$-functions.  Some of the properties of higher genus trigonal curves have been explored in \cite{MEhgt10}.  The class of $(n,s)$-curves with $n=4$ are defined as \textit{tetragonal curves} and have been recently considered for the first time in \cite{MEe09} and \cite{MEg09}.  The lowest genus tetragonal curve has genus six.

We now discuss how to construct these functions, starting by choosing bases for the space of differential forms of the first kind.  A standard basis may be constructed from the Weierstrass gap sequence associated to the curve (see for example \cite{bg06}).  For any $(2,7)$-curve the basis is
\begin{equation} \label{eq:du27}
  \bm{du}=\Big(\frac{dx}{f_y(x,y)}, \frac{xdx}{f_y(x,y)}, \frac{x^2dx}{f_y(x,y)}\Big),
\end{equation}
while for any \((3,4)\)-curve it is
\begin{equation} \label{eq:du34}
  \bm{du}=\Big(\frac{dx}{f_y(x,y)}, \frac{xdx}{f_y(x,y)}, \frac{ydx}{f_y(x,y)}\Big).
\end{equation}
We choose a symplectic basis in \(H_1(C,\mathbb{Z})\) of cycles (closed paths) upon the compact Riemann surface defined by $C$.  We denote these by
$\{\alpha_1,\ \alpha_2,\ \alpha_3,\ \beta_1,\ \beta_2,\ \beta_3\}$
and ensure they have intersection numbers
\begin{eqnarray*}
\alpha_i \cdot \alpha_j = 0, \qquad \beta_i \cdot \beta_j = 0, \qquad
\alpha_i \cdot \beta_j 
= \begin{cases}
1 \ & \text{if}\ \ i = j, \\
0 \ & \text{if}\ \ i \neq j.
\end{cases}
\end{eqnarray*}
We introduce $\bm{dr}$ as a basis of meromorphic differentials which have their only pole at $\infty$.  We will not explicitly use these in this paper and so for an explicit construction we refer the reader to Section 2.2.2 in \cite{bel97} for the $(2,7)$-case and to equation  (2.22) in \cite{eemop07} for the $(3,4)$-case.
These bases are derived alongside a fundamental differential of the second kind which plays an important role in the theory of $\wp$-functions.  Again, we will not explicitly use the differential in this paper and so refer the reader to
\cite{bg06} for an example of its construction and role.

We can now define the standard period matrices associated to the curve as follows.
\begin{eqnarray*}
\begin{array}{cc}
      \omega'  = \left( \oint_{\alpha_k} du_\ell \right)_{k,\ell = 1,2,3} &
\qquad\omega'' = \left( \oint_{ \beta_k} du_\ell \right)_{k,\ell = 1,2,3}  \\
        \eta'  = \left( \oint_{\alpha_k} dr_{\ell} \right)_{k,\ell = 1,2,3} &
\qquad  \eta'' = \left( \oint_{ \beta_k} dr_{\ell} \right)_{k,\ell = 1,2,3}
\end{array}.
\end{eqnarray*}
We define the period lattice $\Lambda$ formed from $\omega'$, $\omega''$ by
\[
\Lambda = \big\{ \bm{m}\omega' + \bm{n}\omega'', \quad \bm{m},\bm{n} \in \Z^3 \big\}.
\]
The functions we treat in this paper are defined upon \(\mathbb{C}^3\) with coordinates usually given as
\begin{equation}
\bm{u}=(u_1,\ u_2,\ u_3).
\end{equation}
Note that any point $\bu \in \C^3$ can be given by
\[
\bu = \sum_{i=1}^3 \int_{\infty}^{P_i} \bm{du},
\]
where the $P_i$ are points upon $C$.  The period lattice $\Lambda$ is a lattice in the space $\C^3$.  Then the Jacobian variety of $C$ is presented by $\mathbb{C}^3/\Lambda$, and is denoted by $J$.
We define \(\kappa\) as the modulo \(\Lambda\) map,
\begin{equation}
\kappa \ : \ \mathbb{C}^3 \to J.
\end{equation}
For $k=1$, $2$, $\dots$ define $\mathfrak{A}_k$,
the \emph{Abel map} from the $k$-th symmetric product \(\mathrm{Sym}^k(C)\) to $J$.
\begin{align}
\mathfrak{A}_k: \text{Sym}^k(C) &\to     J \nonumber \\
(P_1,\dots,P_k)   &\mapsto
\left( \int_{\infty}^{P_1} \bm{du} + \dots + \int_{\infty}^{P_k} \bm{du} \right) \pmod{\Lambda},
\label{eq:Abel}
\end{align}
where the $P_i$ are again points upon $C$.
Denote the image of the $k$-th Abel map by $W^{[k]}$ and define the \emph{$k$-th standard theta subset}
(often referred to as the $k$-th strata) by
\begin{eqnarray*}
\Theta^{[k]} = W^{[k]} \cup [-1]W^{[k]},
\end{eqnarray*}
where \([-1]\) means that
\begin{eqnarray*}
[-1](u_1, u_2 ,u_3) = (-u_1, -u_2 ,-u_3).
\end{eqnarray*}
When $k=1$ the Abel map gives an embedding of the curve $C$ upon which we define $\xi= x^{-\frac{1}{n}}$ as
the local parameter at the origin, $\mathfrak{A}_1(\infty)$.
We can then express $y$ and the basis of differentials using $\xi$ and integrate to give series expansions for $\bu$.
We can check the weights of $\bu$ from these and see that they are prescribed by the Weierstrass gap sequence.
In particular, $\bu=(u_1,u_2,u_3)$ has weights $(5,3,1)$ in the $(2,7)$-case and weights $(5,2,1$) in the $(3,4)$-case.

We now consider functions that are periodic with respect to the lattice $\Lambda$.

\begin{definition} \label{def:HG_Abelian}
Let $\mathfrak{M}(\bu)$ be a meromorphic function of $\bu \in \C^g$.
Then $\mathfrak{M}$ is a \textbf{standard Abelian function associated with $\bm{C}$} if it has poles only along \(\kappa^{-1}(\Theta^{[g-1]})\), and satisfies
\begin{equation} \label{eq:HG_Abelian}
\mathfrak{M}(\bu + \bm{\ell}) = \mathfrak{M}(\bu) \qquad \text{for all } \bm{\ell}\in \Lambda.
\end{equation}
\end{definition}
\noindent Note the comparison with equation (\ref{eq:Intro_period}) and that the period matrices play the role of the scalar periods in the elliptic case.  We will define generalisations of the Weierstrass $\wp$-function which satisfy
equation (\ref{eq:HG_Abelian}), defined using the quasi-periodic function defined below.

First, let \(\bm{\delta} =
\bm{\delta'}\omega'+\bm{\delta''}\omega''\) be the Riemann constant
with base point \(\infty\).  Then
\(\Big[{\atop{\bm{\delta'}}{\bm{\delta''}}}\Big]\) is the theta characteristic
presenting the Riemann constant for the curve C with respect to the
base point $\infty$ and generators $\{\alpha_j,\ \beta_j\}$ of
$H_1(C,\mathbb{Z})$.  For any $((2,7)$-curve, we have a classical
choice of \(\{\alpha_j,\beta_j\}\) and (see \cite{mu83} or
\cite{christoffel1901}) we have
\begin{equation}
\bm{\delta}'=[\tfrac32\ 1\ \tfrac12]^T, \quad  \bm{\delta}''=[\tfrac12\ \tfrac12\ \tfrac12]^T.
\end{equation}
For any \((3,4)\)-curve, Shiga computed in \cite{shiga88} that with his choice of  \(\{\alpha_j,\beta_j\}\) we have
\begin{equation}
\bm{\delta}'=[0\ \tfrac12\ 0]^T, \quad  \bm{\delta}''=[0\ \tfrac12\ 0]^T.
\end{equation}

\begin{definition} \label{def:HG_sigma}
The \textbf{Kleinian $\bm{\s}$-function} associated to a genus three $(n,s)$-curve may be defined
using a multivariate $\theta$-function with characteristic \(\bm{\delta}\) as
\begin{align*}
\s(\bu) &= c \exp \big( \textstyle \frac{1}{2} \bm{u} \eta' (\omega')^{-1} \bm{u}^T \big)
\cdot \theta[\bm{\delta}]\big((\omega')^{-1}\bm{u}^T \hspace*{0.05in} \big| \hspace*{0.05in} (\omega')^{-1} \omega''\big).  \\
&= c \exp \big( \textstyle \frac{1}{2} \bm{u} \eta' (\omega')^{-1} \bm{u}^T \big) 
\\ & \qquad \times \displaystyle \sum_{\bm{m} \in \Z^3} \exp \bigg[ 2\pi i \big\{                    \label{eq:HG_sigma}
 \textstyle \frac{1}{2} (\bm{m} + \bm{\delta'})^T (\omega')^{-1} \omega''(\bm{m} + \bm{\delta'})
+ (\bm{m} + \bm{\delta'})^T ((\omega')^{-1} \bm{u}^T + \bm{\delta''} )\big\} \bigg].  \nonumber
\end{align*}
The constant $c$ is dependent upon the curve parameters and the basis of cycles and is fixed later,
following Lemma {\rm \ref{lem:sigexp}}.
\end{definition}
We now summarise the key properties of the $\sigma$-function.  See \cite{bel97} or \cite{N10} for the construction of the $\sigma$-function to satisfy these properties.

For any point $\bu \in \mathbb{C}^3$ we denote by $\bu'$ and $\bu''$ the vectors in $\mathbb{R}^3$ such that
\begin{equation}
\bu=\bm{u}'\omega'+\bm{u}''\omega''.
\end{equation}
Therefore a point  $\bm{\ell}\in\Lambda$ is written as
\begin{equation} \label{eq:HG_ell}
\bm{\ell} = \bm{\ell'}\omega' + \bm{\ell''}\omega'' \in \Lambda, \qquad \bm{\ell'},\bm{\ell''} \in \Z^3.
\end{equation}
For $\bu, \bv \in \C^3$ and $\bm{\ell} \in \Lambda$, define $L(\bu,\bv)$ and $\chi(\bm{\ell})$ as follows:
\begin{align*}
L(\bu,\bv) &= \bu\big( \bm{v'}\eta' + \bm{v''}\eta'' \big)^T, \\
\chi(\bm{\ell}) &= \exp \big[ 2 \pi \text{i} \big\{ \bm{\ell'}(\delta'')^T - \bm{\ell''}(\delta')^T
+ \tfrac{1}{2}\bm{\ell'}(\bm{\ell''})^T \big\}\big].
\end{align*}

\begin{lemma}
The $\sigma$-function 
has the following properties.
\begin{itemize}

\item It is an entire function on $\C^3$.

\item It has zeros of order one along the set  $\kappa^{-1}(\Theta^{[2]})$.
Further, $\sigma(\bu) \neq 0$ outside this set.

\item For all $\bu \in \C^3, \bm{\ell} \in \Lambda$ the function $\sigma(\bu)$ has the following
quasi-periodicity property:
\begin{eqnarray} \label{eq:HG_quas}
\sigma(\bu + \bm{\ell}) = \chi(\bm{\ell})
\exp \left[ L \left( \bu + \frac{\bm{\ell}}{2}, \bm{\ell} \right) \right] \sigma(\bu).
\end{eqnarray}

\item It has definite parity given by
$
\sigma([-1]\bu) = \pm\sigma(\bu)
$
in the $(2,7)$ and $(3,4)$-cases respectively.
\end{itemize}

\end{lemma}
\begin{proof}
The function is clearly entire from the definition, while the quasi-periodicity and zeros of the function are classical results, (see \cite{ba97}), that are fundamental to the definition of the function.  They both follow from the properties of the multivariate $\theta$-function.  The parity property follows from Proposition 4(iv) in \cite{N10} which states that for a general $(n,s)$-curve,
\begin{equation} \label{eq:sigparity}
\sigma(-\bu) = (-1)^{\frac{1}{24}(n^2-1)(s^2-1)}\sigma(\bu).
\end{equation}
\end{proof}
We next define $\wp$-functions using an analogy of equation (\ref{eq:Intro_elliptic_ps}).
Since there is more than one variable we need to be clear which we differentiate with respect to.
We actually define several multivariate $\wp$-functions and introduce the following notation.

\begin{definition} \label{def:nip}
Define \textbf{$\bm{m}$-index Kleinian $\bm{\wp}$-functions}, (where $m\geq2$) as
\begin{eqnarray*}
\wp_{i_1,i_2,\dots,i_m}(\bu) = - \frac{\D}{\D u_{i_1}} \frac{\D}{\D u_{i_2}}\dots \frac{\D}{\D u_{i_m}} \log \big[ \sigma(\bu) \big],
\quad i_1 \leq \dots \leq i_m \in \{1,\dots,g\}.
\end{eqnarray*}
\end{definition}
The $m$-index $\wp$-functions are meromorphic with poles of order $m$ when $\sigma(\bu)=0$.  We can check that they satisfy equation (\ref{eq:HG_Abelian}) and hence they are Abelian.   The $m$-index $\wp$-functions have definite parity with respect to the change of variables $\bu \to [-1]\bu$ and are odd if $m$ is odd and even if $m$ is even.  Note that the ordering of the indices is irrelevant and so for simplicity we always order in ascending values.

When the $(n,s)$-curve is chosen to be the classic elliptic curve then the Kleinian $\sigma$-function coincides with the classic $\sigma$-function and the sole 2-index $\wp$-function coincides with the Weierstrass $\wp$-function.  The only difference is the notation with
\begin{align*}
\wp_{11}(\bu) \equiv \wp(u), \quad \wp_{111}(\bu) \equiv \wp'(u), \quad \wp_{1111}(\bu) \equiv \wp''(u).
\end{align*}
Clearly, as the genus of the curve increases so do the number of associated $\wp$-functions.  For the curves we study we set $g=3$ to leave six 2-index $\wp$-functions, ten 3-index $\wp$-functions etc.  By considering Definition \ref{def:nip} we see that the weight of the $\wp$-functions is the negative of the sum of the weights of the variables indicated by the indices.  So although the two curves have the same number of associated $\wp$-functions, they are of different weights arising from the different basis of holomorphic differentials.  We will find in the following sections that this leads to variations in the theory of the two sets of functions.

We now introduce a final result detailing how the functions can be expressed using series expansions.
\begin{lemma} \label{lem:sigexp}
The Taylor series expansion of $\sigma(\bu)$ about the origin may be written as follows\,{\rm :}
For each type of cyclic genus three curve,
there are constants $K_{2,7}$ and $K_{3,4}$ depending only on the $\lambda_j$ and $\{\alpha_j,\,\beta_j\}$ such that
\begin{equation} \label{eq:HG_SW}
\sigma(\bu)=
\begin{cases}\displaystyle
\displaystyle K_{2,7} \cdot \Big(SW_{2,7}(\bu) + \sum_{k=1}^{\infty} C_{6+2k}(\bu,\bm{\lambda})\Big) & \text{for a cyclic } (2,7)\text{-curve},\\
\displaystyle K_{3,4} \cdot \Big(SW_{3,4}(\bu) + \sum_{k=1}^{\infty} C_{5+3k}(\bu,\bm{\lambda})\Big) & \text{for a cyclic } (3,4)\text{-curve}.\\
\end{cases}
\end{equation}
Here $SW_{n,s}$ is the Schur-Weierstrass polynomial generated by $(n,s)$ and each $C_m$ is a polynomial composed of products of monomials in $\bm{u}$ of weight $m$ multiplied by monomials in the curve parameters of weight $(6-m)$ and $(5-m)$ in the $(2,7)$ and $(3,4)$-cases respectively.
\end{lemma}
\noindent A similar result may be stated for the non-cyclic curves using $\mu_j$ instead of $\lambda_j$, but the calculations involved in deriving such an expansion are computationally much more difficult.
\begin{proof}
We refer the reader to \cite{N10} for a proof of the relationship between the $\sigma$-function and the Schur-Weierstrass polynomials and note that this was first discussed in \cite{bel99}.  We see that the remainder of the expansion must depend on the curve parameters and split it up into the $C_k$ using the weight properties. Each $C_k$ must be finite since the number of possible terms with the prescribed weight properties is finite. \\
\end{proof}

Expansions of this type were first introduced in \cite{bg06} in relation to the study of the Benney equations.
Since then they have been an integral tool in the investigation of Abelian functions.  Recently, computational techniques based on the weight properties have been used to derive much larger expansions and we refer the reader to \cite{MEe09} and \cite{MEhgt10} for details of how to construct and use these expansions.

We fix $c$ in Definition \ref{def:HG_sigma} to be the value that makes the constant one in the above lemma.  Some other authors working in this area may use a different constant and in general these choices are not equivalent.  However, the constant can be seen to cancel in the definition of the $\wp$-functions, leaving most results independent of $c$.  Note that this choice of $c$ ensures that the Kleinian $\sigma$-function matches the Weierstrass $\sigma$-function when the $(n,s)$-curve is chosen to be the classic elliptic curve.

The connection with the Schur-Weierstrass polynomials also allows us to determine the weight of the $\sigma$-function as $(1/24)(n^2-1)(s^2-1)$.  In the $(2,7)$-case this gives $\sigma(\bu)$ weight $6$ while in the $(3,4)$-case it has weight $5$.  The respective Schur-Weierstrass polynomials are,
\begin{align}
SW_{2,7} &= \textstyle \frac{1}{45}u_{3}^6 - \frac{1}{3}u_{3}^3u_{2} - u_{2}^2 + u_{3}u_{1}, \label{eq:27SW} \\
SW_{3,4} &= \textstyle \frac{1}{20}u_{3}^5 - u_{3}u_{2}^2 + u_{1}.                           \label{eq:34SW}
\end{align}
The $\sigma$-function expansion has been calculated up to and including the polynomial $C_{38}$ in the $(2,7)$-case and $C_{35}$ in the $(3,4)$-case.  These expansions are available from the authors and should also be available from the journal in the \emph{Supplementary Materials} to the paper.

\section{Bases for the vector spaces of Abelian functions} \label{SEC_Bases}

\subsection{General theory} \label{SEC_Bases_gen}

We classify the Abelian functions according to their pole structure.  We denote by
$\Gamma \big( J, \mathcal{O}(m \Theta^{[k]} ) \big)$ the vector space of Abelian functions defined upon $J$ which have poles of order at most $m$, occurring only on the $k$th strata, $\Theta^{[k]}$.  The case where $k=g-1$ is of interest since all the Abelian functions we deal with have poles occurring here, on the $\theta$-divisor.  A key problem in the theory of Abelian functions is the generation of bases for these vector spaces.

Note that the dimension of the space $\Gamma \big( J, \mathcal{O}(m \Theta^{[g-1]} ) \big)$ is $m^g$ by the Riemann-Roch theorem for Abelian varieties. (See for example \cite{la82} page 99.)  Recall that the $m$-index $\wp$-functions all have poles of order $m$.  We see that the number of $m$-index $\wp$-functions associated with a genus $g$ curve is
\[
( g + m - 1)!/[m!(g - 1)!],
\]
which will not grow as fast as the dimension of the space.  We hence need to identify a wider class of Abelian functions than the $\wp$-functions in order to construct such bases.

Recall that an entire Abelian function must be a constant and that there is no Abelian function with a single pole of order one.  Hence those Abelian functions with poles of order two are the simplest and so are often referred to as \textit{fundamental Abelian functions}.  The basis problem has been solved in general for such functions, through the inclusion of the following extra class of Abelian functions.

\begin{definition} \label{def:Qdef}
Define the operator $\mathcal{D}_i$ as below.  This is now known as \textit{Hirota's bilinear operator}, although it was used much earlier by Baker in \cite{ba07}.
\[
\mathcal{D}_i = \frac{\D}{\D u_i} - \frac{\D}{\D v_i}.
\]
Then an alternative, equivalent definition of the 2-index $\wp$-functions is given by
\begin{equation} \label{eq:2ip_Hirota}
\wp_{ij}(\bm{u}) = - \dfrac{1}{2\sigma(\bm{u})^2} \mathcal{D}_i\mathcal{D}_j \sigma(\bm{u}) \sigma(\bm{v}) \hspace*{0.1in} \Big|_{\bv=\bu}
\qquad i \leq j \in \{1,\dots,g\}.
\end{equation}
We extend this approach to define the \textbf{$\bm{m}$-index $\bm{Q}$-functions}, for $m$ even, by
\begin{equation} \label{eq:Qdef}
Q_{i_1, i_2,\dots,i_m}(\bm{u}) =  \frac{(-1)}{2\sigma(\bm{u})^2} \mathcal{D}_{i_1}\mathcal{D}_{i_2}\dots \mathcal{D}_{i_m} \sigma(\bm{u}) \sigma(\bm{v}) \hspace*{0.08in} \Big|_{\bv=\bu} \quad
\end{equation}
where $i_1 \leq \dots  \leq i_m \in \{1,\dots,g\}$.
\end{definition}

The $m$-index $Q$-functions are Abelian functions with poles of order two when $\sigma(\bu)=0$.  Note that if you were to apply the definition with $m$ odd then the functions would be identically zero.  A 4-index $Q$-function was originally used by Baker with the generic 4-index $Q$-functions introduced when research first started on the trigonal curves.  (In the literature they are just defined as \textit{$Q$-functions}).  The definition for $m$-index $Q$-functions above was developed in \cite{MEe09} as increasing classes are required to deal with cases of higher genus, (see also \cite{MEhgt10}).  In this paper we only need to use 4-index $Q$-functions, which in \cite{eemop07} were shown to satisfy,
\begin{eqnarray} \label{eq:4iQ}
Q_{ijk\ell} = \wp_{ijk\ell} - 2 \wp_{ij}\wp_{k\ell}-2\wp_{ik}\wp_{j\ell} -2\wp_{i\ell}\wp_{jk}.
\end{eqnarray}
Similar expressions have been found for the higher index $Q$-functions, (see \cite{MEe09}).

When considering the vector spaces for Abelian functions with poles of
higher order, a natural place to look for extra functions is in the
derivatives of the functions with lower order poles.  Note that while
the derivatives of $m$-index $\wp$-functions are $(m+1)$-index
$\wp$-functions, the same is not true for the $Q$-functions.  For
brevity we adopt the notation
\[
\partial_{i} Q_{ijkl}(\bu) = \frac{\partial}{\partial u_i} Q_{ijkl}(\bu)
\]
and similarly for other functions.  As discussed in \cite{N11}, the derivatives of existing basis functions will not be sufficient to find successive bases in the case where the theta divisor has singular points.  In \cite{eemop07} the authors introduced the following new class of functions to overcome this problem for the three pole basis in the $(3,4)$-case.

\begin{definition} \label{def_DM}
Consider the matrix,
\[
\big[ \wp_{ij} \big]_{3 \times 3} =
\left[ \begin{array}{ccc}
\wp_{11} & \wp_{12} & \wp_{13} \\
\wp_{21} & \wp_{22} & \wp_{23} \\
\wp_{31} & \wp_{32} & \wp_{33} \\
\end{array} \right]
=
\left[ \begin{array}{ccc}
\wp_{11} & \wp_{12} & \wp_{13} \\
\wp_{12} & \wp_{22} & \wp_{23} \\
\wp_{13} & \wp_{23} & \wp_{33} \\
\end{array} \right].
\]
Then the function $\wp^{[ij]}$ is defined to be the $(i,j)$ minor of
$[ \wp_{ij} ]_{3 \times 3}$.
\end{definition}
\noindent For example,
\[
\wp^{[12]} =
\left| \begin{array}{cc}
\wp_{12} &  \wp_{23} \\
\wp_{13} &  \wp_{33} \\
\end{array} \right| = \wp_{12}\wp_{33} - \wp_{23}\wp_{13}.
\]
So each of these functions will be a sum of two pairs of products of 2-index $\wp$-functions.  Although each product will have poles of order four, we can check by substituting with Definition \ref{def:nip} that these cancel so that the function has poles of order three.  Although this class solved the three pole basis problem in the $(3,4)$-case, it was not sufficient to complete the corresponding basis in the $(2,7)$-case.  We explicitly address this and similar problems for the four pole vector spaces below.  A general method to generate new Abelian functions is currently being developed with some recent new results in \cite{ME11}.

\subsection{The general $(2,7)$-curve}

In this section we treat the most general $(2,7)$-curve, namely the curve defined by equation (\ref{eq:27}).
A well known basis for $\Gamma \big( J, \mathcal{O}(2 \Theta^{[2]} ) \big)$ in the $(2,7)$-case is
\begin{equation} \label{eq:27_2pole}
\C1 \,\op\, \C\wp_{11} \,\op\, \C \wp_{12} \,\op\, \C\wp_{13}
\,\op\, \C\wp_{22} \,\op\, \C\wp_{23} \,\op\, \C\wp_{33} \,\op\, \C \Delta,
\end{equation}
where
\begin{equation} \label{eq:27Delta}
\Delta=\wp_{11}\wp_{33}-\wp_{12}\wp_{23}-\wp_{13}^2+\wp_{13}\wp_{22}.
\end{equation}
The function $\Delta$ has poles of order three in general, (we can this
check using Definition {\ref{def:nip}}).  However, these cancel to order two
in the $(2,7)$-case only, (which we can check using the $\sigma$-expansion).
The $\Delta$-function was introduced by Baker in \cite{ba03} although he did not use it explicitly to form such bases.
More recently the $\Delta$-function has been used in \cite{CA2008} in
a covariant analogue to the theory of hyperelliptic $\wp$-functions.
It is possible to rewrite the theory to match the general approach
suggested above, by replacing $\Delta$ with a $Q$-function of weight
$-12$.  However, the use of $\Delta$ is advantageous since it allows
the theory to be completely realised in terms of 2 and 3-index
$\wp$-functions, with all higher index $\wp$-functions given
recursively in these.

The construction of $\Gamma \big( J, \mathcal{O}(3 \Theta^{[2]} ) \big)$ has recently been studied by Nakayashiki in \cite{N11}.  The author was able to enumerate in detail all the terms with the exception of the final one, which was labeled $F_3$, (see \cite{N11} page 27 onwards).   It was shown that
\[
F_3 = \frac{1 + \text{(lower degree terms)}}{\sigma^3}
\]
and that it satisfies certain differential equations involving power
series of the $u_i$.  In Theorem \ref{thm:27_3pole} below we give a
new explicit form for this term, which we rename as $T$, in terms of
$\wp$-functions.

\begin{theorem} \label{thm:27_3pole}
The basis for $\Gamma \big( J, \mathcal{O}(3 \Theta^{[2]} ) \big)$ is given by
\begin{align} \label{eq:27_3pole}
\begin{array}{ccccccccccccccc}
(\ref{eq:27_2pole})  &\op& \C\wp_{111} &\op& \C\wp_{112} &\op& \C\wp_{113} &\op& \C\wp_{122} &\op& \C\wp_{123}         &\op& \C\wp_{133}  \\
                     &\op& \C\wp_{222} &\op& \C\wp_{223} &\op& \C\wp_{233} &\op& \C\wp_{333} &\op& \C \partial_1\Delta &\op& \C \partial_2\Delta  \\
                     &\op& \C \partial_3\Delta           &\op& \C\wp^{[11]}&\op& \C\wp^{[12]}&\op& \C\wp^{[13]}        &\op& \C\wp^{[23]}
                     &\op& \C\wp^{[33]} &\op& \C T,\\
\end{array}
\end{align}
where
\begin{equation} \label{eq:T}
T = \wp_{222}^2 - 4\wp_{22}^3 - Q_{2222}\wp_{22} \,= \,2\wp_{22}^3 + \wp_{222}^2 - \wp_{22}\wp_{2222}.
\end{equation}
\end{theorem}
\begin{proof}

The dimension of the space is $3^g=3^3=27$ by the Riemann-Roch theorem for Abelian varieties.  All the selected elements belong to the space and we can easily check their linear independence explicitly in Maple using the $\sigma$-expansion.

To actually construct the basis we started by including the 8
functions from basis (\ref{eq:27_2pole}) for the functions with poles
of order at most two.  We then know that the remaining entries must
have poles of order three.  We start by looking for entries from the
set of derivatives of the basis (\ref{eq:27_2pole}).  We test at
decreasing weight levels and look to see whether these functions can
be written as a linear combination upon substitution of the series
expansions.  (A more detailed discussion of the algorithm used for
this is available in \cite{MEe09}.)

Note that while this theorem holds for the general $(2,7)$-curve in equation (\ref{eq:27}), we need only use the series expansions associated to the cyclic $(2,7)$-curve in equation (\ref{eq:c27}).  This is because if an element cannot be expressed using the basis with the restriction on the parameters, then neither will it be expressible with the wider set of parameters.  Further, we only need to use sufficient expansion to give non-zero evaluations of the functions considered in order to check whether they are linearly independent.

After examining all these functions we find that 21 basis elements have been identified.
We identify a further 5 by considering the class given in Definition \ref{def_DM}.  Note that unlike the $(3,4)$-case, here
\[
\wp^{[13]} = \wp_{12}\wp_{23} - \wp_{22}\wp_{13} \qquad \text{and} \qquad \wp^{[22]} = \wp_{11}\wp_{33} - \wp_{13}^2
\]
have the same weight, $(-12)$.  Further, they are linearly dependent and so only one can be included in the basis, (either may be chosen).

To find the final function the following new class of functions is considered.
\begin{align}
&\mathcal{T}_{ijklmn} = \textstyle \wp_{ijk}\wp_{lmn}
- \frac{2}{3}\wp_{ij}\wp_{kl}\wp_{mn} - \frac{2}{3}\wp_{ij}\wp_{km}\wp_{ln}
- \frac{2}{3}\wp_{ij}\wp_{kn}\wp_{lm} - \frac{2}{3}\wp_{ik}\wp_{jl}\wp_{mn}   \label{eq:Tgen} \\
&\quad \textstyle - \frac{2}{3}\wp_{ik}\wp_{jm}\wp_{ln}
- \frac{2}{3}\wp_{ik}\wp_{jn}\wp_{lm}
- \frac{2}{3}\wp_{il}\wp_{jk}\wp_{mn} + \frac{1}{3}\wp_{il}\wp_{jm}\wp_{kn}
+ \frac{1}{3}\wp_{il}\wp_{jn}\wp_{km}   \nonumber \\
&\quad \textstyle  - \frac{2}{3}\wp_{im}\wp_{jk}\wp_{ln}
+ \frac{1}{3}\wp_{im}\wp_{jl}\wp_{kn}
+ \frac{1}{3}\wp_{im}\wp_{jn}\wp_{kl}
- \frac{2}{3}\wp_{in}\wp_{jk}\wp_{lm} + \frac{1}{3}\wp_{in}\wp_{jl}\wp_{km}
+ \frac{1}{3}\wp_{in}\wp_{jm}\wp_{kl} \nonumber \\
&\quad \textstyle - \frac{2}{3}Q_{ijkl}\wp_{mn} - \frac{2}{3}Q_{ijkm}\wp_{ln}
- \frac{2}{3}Q_{ijkn}\wp_{lm} + \frac{1}{3}Q_{ijlm}\wp_{kn}
+ \frac{1}{3}Q_{ijln}\wp_{km}
+ \frac{1}{3}Q_{ijmn}\wp_{kl}   \nonumber \\
&\quad \textstyle + \frac{1}{3}Q_{iklm}\wp_{jn}
+ \frac{1}{3}Q_{ikln}\wp_{jm}
+ \frac{1}{3}Q_{ikmn}\wp_{jl}
- \frac{2}{3}Q_{ilmn}\wp_{jk}
+ \frac{1}{3}Q_{jklm}\wp_{in} + \frac{1}{3}Q_{jkln}\wp_{im} \nonumber \\
&\quad \textstyle + \frac{1}{3}Q_{jkmn}\wp_{il} - \frac{2}{3}Q_{jlmn}\wp_{ik}
- \frac{2}{3}Q_{klmn}\wp_{ij}.  \nonumber
\end{align}
Substituting with Definition \ref{def:nip} shows these functions to
have poles of order at most three.  We stress that this is the case
for the $\mathcal{T}$-functions associated to {\em any} curve, and not
just the genus three case we consider here. The functions were derived
through an attempt to match in general, the poles of a quadratic term
in the 3-index $\wp$-functions with a polynomial in 2-index
$\wp$-functions.  Although this is not possible, we can match by
including Q-functions as well.  A general polynomial of this type was
constructed and the coefficients then determined as above to ensure
the poles of order greater than three vanish.  A method to generate
new basis functions is currently being developed with some recent results published in \cite{ME11}.

The new class of functions is examined at decreasing weight levels.  We  find that all can expressed as a linear combination of existing basis entries except for those at weight $-18$.  We can choose any $\mathcal{T}$-function at weight $-18$ to complete the basis and use the simplest of these,
$\mathcal{T}_{222222} = T$ as given in equation (\ref{eq:T}). \\
\end{proof}

\begin{theorem} \label{thm:27_4pole}
The basis for $\Gamma \big( J, \mathcal{O}(4 \Theta^{[2]} ) \big)$ is given by
\begin{align*} 
\begin{array}{ccccccccccccc}
(\ref{eq:27_3pole})
&\op& \C\wp_{1111} &\op& \C\wp_{1112} &\op& \C\wp_{1113} &\op& \C\wp_{1122} &\op& \C\wp_{1123} \\
&\op& \C\wp_{1133} &\op& \C\wp_{1222} &\op& \C\wp_{1223} &\op& \C\wp_{1233} &\op& \C\wp_{1333} \\
&\op& \C\wp_{2222} &\op& \C\wp_{2223} &\op& \C\wp_{2233} &\op& \C\wp_{2333} &\op& \C\wp_{3333} \\
&\op& \C \partial_{13}\Delta          &\op& \C \partial_{23}\Delta          &\op& \C \partial_{33}\Delta
&\op& \C \partial_{12}\Delta          &\op& \C \partial_{22}\Delta          \\
&\op& \C \partial_{11}\Delta          &\op& \C \partial_{1}\wp^{[11]}       &\op& \C \partial_{2}\wp^{[11]}
&\op& \C \partial_{3}\wp^{[11]}       &\op& \C \partial_{1}\wp^{[12]}       \\
&\op& \C \partial_{3}\wp^{[12]}       &\op& \C \partial_{1}\wp^{[13]}       &\op& \C \partial_{2}\wp^{[13]}
&\op& \C \partial_{3}\wp^{[13]}       &\op& \C \partial_{1}\wp^{[23]}       \\
&\op& \C \partial_{2}\wp^{[23]}       &\op& \C \partial_{1}\wp^{[33]}       &\op& \C \partial_{2}\wp^{[33]}
&\op& \C \partial_{1}T                &\op& \C \partial_{2}T                \\
&\op& \C \partial_{3}T                &\op& \C G,
\end{array}
\end{align*}
where $G = \wp_{22222222} - 140\wp_{2222}^2$.
\end{theorem}
\begin{proof}
We follow the proof of Theorem \ref{thm:27_3pole}.  This time the dimension is $4^g=4^3=64$ and we find that we can identify 63 functions using the basis (\ref{eq:27_3pole}) and its derivatives.

To find the final basis function the following new class of Abelian functions was examined.
\begin{align*}
&\mathcal{G}_{ijklmnop} = \wp_{ijklmnop} - 4\big(
\wp_{ijno}\wp_{klmp} + \wp_{ijnp}\wp_{klmo} + \wp_{ijlp}\wp_{kmno}
+ \wp_{ijmn}\wp_{klop} + \wp_{ijmo}\wp_{klnp} \\
&\quad
+ \wp_{ijmp}\wp_{klno} + \wp_{ijkm}\wp_{lnop} + \wp_{ijkn}\wp_{lmop}
+ \wp_{ijko}\wp_{lmnp} + \wp_{ijkp}\wp_{lmno} + \wp_{ijlm}\wp_{knop} \\
&\quad + \wp_{ijln}\wp_{kmop} + \wp_{ijlo}\wp_{kmnp} + \wp_{ijkl}\wp_{mnop}
+ \wp_{ijop}\wp_{klmn} + \wp_{iklm}\wp_{jnop} + \wp_{ikln}\wp_{jmop} \\
&\quad + \wp_{iklo}\wp_{jmnp} + \wp_{iklp}\wp_{jmno} + \wp_{ikmn}\wp_{jlop}
+ \wp_{ikmo}\wp_{jlnp} + \wp_{ikmp}\wp_{jlno} + \wp_{ikno}\wp_{jlmp} \\
&\quad + \wp_{iknp}\wp_{jlmo} + \wp_{ikop}\wp_{jlmn} + \wp_{ilmn}\wp_{jkop}
+ \wp_{ilmo}\wp_{jknp} + \wp_{ilmp}\wp_{jkno} + \wp_{ilno}\wp_{jkmp} \\
&\quad + \wp_{ilnp}\wp_{jkmo}  + \wp_{ilop}\wp_{jkmn} + \wp_{imno}\wp_{jklp}
+ \wp_{imnp}\wp_{jklo} + \wp_{imop}\wp_{jkln} + \wp_{inop}\wp_{jklm} \big).
\end{align*}
The $\mathcal{G}$-functions associated with any curve have poles of order at most four (using Definition \ref{def:nip}) and were derived by matching the higher order poles in an 8-index $\wp$-function using arbitrary sum of quadratic terms in the 4-index $\wp$-functions.  Examining at decreasing weight levels we see that we need to include a function at weight $-24$ and choose $\mathcal{G}_{22222222} = G$ as given in the theorem to complete the basis.  \\
\end{proof}

We note that these bases are not unique.  As discussed at the start of the section, we could replace the $\Delta$-function with a Q-function if desired.  Similarly, in the theorems above we could have used any $\mathcal{T}$-function at weight $-18$ and any $\mathcal{G}$-function at weight $-24$ to complete the 3 and 4-pole bases respectively.  However, we can conclude that the weight structure of each basis is unique.
\begin{lemma}
Consider a basis for $\Gamma \big( J, \mathcal{O}(k \Theta^{[g-1]} ) \big)$, associated to any $(n,s)$-curve.  Scale the functions so that they do not depend on any curve constants.  Then the set of weights of the functions is unique.
\end{lemma}
For example, the basis (\ref{eq:27_2pole}) has weights $[-12, -10, -8, -6, -6, -4, -2, 0]$.  While the functions for this basis may be changed, this set of weights can not.
\begin{proof}
Suppose that we could replace a basis entry with another of different weight.  The replacement function would have to be linearly independent of the other basis entries, 
and hence a constant multiple of the original function.
This would lead to equations that are inhomogeneous in the Sato weights. \\ 
\end{proof}

We can also determine a minimal bound of these weights, as detailed in the following lemma.
\begin{lemma} \label{lem:minwt}
The basis for the  space $\Gamma \big( J, \mathcal{O}(k \Theta^{[g-1]} ) \big)$ associated to
an $(n,s)$-curve will contain functions with weights no lower than $-k\mathrm{wt}(\sigma)$.
\end{lemma}
\begin{proof}
Such basis functions may be written as a rational function in $\sigma(\bu)$ with overall denominator $\sigma(\bu)^k$.  The simplest possible numerator is a constant, which gives a function with weight $-k\text{wt}(\sigma)$.  Any other numerator would depend on $\bu$ and so have a higher weight.    \\
\end{proof}
Lemma \ref{lem:minwt} stops us from testing those functions at a lower weight thus drastically reducing the amount of computation required to find such bases.  These lower weight functions must be expressible as a linear combination of the basis functions in which every term depends on the curve parameters.

We conjecture that every such basis may be evaluated using polynomials of $\wp$-functions.  This is known to be true for the curves considered here (see \cite{bel97} and \cite{bel00}) and should follow similarly for all $(n,s)$-curves.  All such polynomials have negative weight except the constant function which will take the maximal weight of zero.  Together with Lemma \ref{lem:minwt} this would restrict basis entries to the weight range $-k\text{wt}(\sigma), \dots, 0.$

We note that in the $(2,7)$-case the weight of $\sigma(\bu)$ is $+6$ and the functions $\Delta, T$ and $G$ have weight $12, 18$ and 24 respectively.  They are hence the minimal weight entries in their respective bases.
By considering the rational limit of these functions, where $\sigma(\bu)=SW_{2,7}(\bu)$, we can check that these functions each have leading term of a constant over $\sigma(\bu)^k$ with $k=2,3,4$ respectively.

It is clear that functions of the form $(\text{const}+\mathcal{O}(u_i))/\sigma^k$ play a special role in completing the construction of
$\Gamma \big( J, \mathcal{O}(k \Theta^{[g-1]} ) \big)$.  Such functions are straightforward to construct using the rational limit and the method of undetermined coefficients.  For example, the following function evaluates to a constant over $\sigma^5$ and we postulate that such a function will be essential for the 5-pole basis.
\[
12 \wp_{22}^5 - 8\wp_{22}^3\wp_{2222} + 6\wp\wp_{22}^2\wp_{222}^2 + \wp_{22}\wp_{2222}^2 - \wp_{2222}\wp_{222}^2.
\]

\subsection{The general $(3,4)$-curve}

Here we treat the most general \((3,4)\)-curve, namely the curve defined by equation (\ref{eq:34}).
Lemma 8.1 in \cite{eemop07} identified the basis for $\Gamma \big( J, \mathcal{O}(2 \Theta^{[2]} ) \big)$ in the $(3,4)$-case as
\begin{equation} \label{eq:34_2pole}
\C1 \,\op\, \C\wp_{11} \,\op\, \C\wp_{12} \,\op\, \C\wp_{13}
    \,\op\, \C\wp_{22} \,\op\, \C\wp_{23} \,\op\, \C\wp_{33} \,\op\, \C Q_{1333},
\end{equation}
and the basis for $\Gamma \big( J, \mathcal{O}(3 \Theta^{[2]} ) \big)$ as
\begin{align} \label{eq:34_3pole}
\begin{array}{ccccccccccccccccccc}
(\ref{eq:34_2pole})  &\op& \C\wp_{111} &\op& \C\wp_{112} &\op& \C\wp_{113} &\op& \C\wp_{122} &\op& \C\wp_{123} \\
                  &\op& \C\wp_{133} &\op& \C\wp_{222} &\op& \C\wp_{223} &\op& \C\wp_{233} &\op& \C\wp_{333} \\
                  &\op& \C \partial_1 Q_{1333}    &\op& \C \partial_2 Q_{1333}    &\op& \C \partial_3 Q_{1333}
                  &\op& \C\wp^{[11]}&\op& \C\wp^{[12]}\\
                  &\op& \C\wp^{[13]}&\op& \C\wp^{[22]}&\op& \C\wp^{[23]}  &\op& \C\wp^{[33]}.
\end{array}
\end{align}
Note that here the functions $\wp^{[13]}$ and $\wp^{[22]}$ have weight $-10$ and $-12$ respectively.  They were hence independent meaning this class was sufficient to complete the basis and no $\mathcal{T}$-functions were required.  However, we find that upon proceeding to the four pole basis we will have to define an extra class.
\begin{theorem} \label{thm:34_4pole}
The basis for $\Gamma \big( J, \mathcal{O}(4 \Theta^{[2]} ) \big)$ is given by
\begin{align} \label{eq:34_4pole}
\begin{array}{ccccccccccc}
(\ref{eq:34_3pole}) &\op& \C\wp_{1111} &\op& \C\wp_{1112} &\op& \C\wp_{1113} &\op& \C\wp_{1122} \\
                 &\op& \C\wp_{1123}    &\op& \C\wp_{1133} &\op& \C\wp_{1222} &\op& \C\wp_{1223} \\
                 &\op& \C\wp_{1233}    &\op& \C\wp_{1333} &\op& \C\wp_{2222} &\op& \C\wp_{2223} \\
                 &\op& \C\wp_{2233}    &\op& \C\wp_{2333} &\op& \C\wp_{3333} &\op& \C\partial_1\partial_1 Q_{1333} \\
                 &\op& \C\partial_1 \partial_2 Q_{1333}   &\op& \C\partial_1 \partial_3 Q_{1333}
                 &\op& \C\partial_2 \partial_2 Q_{1333}   &\op& \C\partial_2 \partial_3 Q_{1333} \\
                 &\op& \C\partial_3 \partial_3 Q_{1333}   &\op& \C\partial_1 \wp^{[11]}
                 &\op& \C\partial_1 \wp^{[12]}            &\op& \C\partial_1 \wp^{[13]} \\
                 &\op& \C\partial_1 \wp^{[22]}            &\op& \C\partial_1 \wp^{[23]}
                 &\op& \C\partial_1 \wp^{[33]}            &\op& \C\partial_2 \wp^{[11]} \\
                 &\op& \C\partial_2 \wp^{[12]}            &\op& \C\partial_2 \wp^{[13]}
                 &\op& \C\partial_2 \wp^{[22]}            &\op& \C\partial_2 \wp^{[23]} \\
                 &\op& \C\partial_2 \wp^{[33]}            &\op& \C\partial_3 \wp^{[11]}
                 &\op& \C\partial_3 \wp^{[12]}            &\op& \C\partial_3 \wp^{[22]} \\ &\op& \C F,
\end{array}
\end{align}
where
\begin{equation} \label{eq:F}
F = \wp_{11}\wp_{22}\wp_{33}-\wp_{11}\wp_{23}^2-\wp_{12}^2\wp_{33}+2\wp_{12}\wp_{13}\wp_{23}-\wp_{13}^2\wp_{22}.
\end{equation}
\end{theorem}
\begin{proof}
We follow the proof of Theorem \ref{thm:27_3pole}.  This time the dimension is $4^g=4^3=64$ and we find that we can identify 63 functions using the basis (\ref{eq:34_3pole}) and its derivatives.

To find the final basis function the following new class of Abelian functions was examined.
\begin{align*}
\mathcal{F}_{ijklmn} &=  \wp_{ij}\wp_{kl}\wp_{mn} - \wp_{ij}\wp_{kn}\wp_{lm} - \wp_{il}\wp_{jk}\wp_{mn} + \wp_{il}\wp_{jn}\wp_{km} \\
& \qquad + \wp_{im}\wp_{jl}\wp_{kn}
- \wp_{im}\wp_{jn}\wp_{kl} + \wp_{in}\wp_{jk}\wp_{lm} - \wp_{in}\wp_{jl}\wp_{km}.
\end{align*}
The $\mathcal{F}$-functions associated with any curve can be seen to have poles of order at most four by substituting with Definition \ref{def:nip} and were derived by matching the higher order poles in an arbitrary sum of cubic terms in the 2-index $\wp$-functions.  Examining at decreasing weight levels we see that we need to include a function at weight $-16$ and choose
$\mathcal{F}_{112233} = F$ as given in equation (\ref{eq:F}) to complete the basis. \\
\end{proof}
\noindent This new basis allows us to derive the new addition formula in Theorem \ref{thm:34_3t3v}.

The problem of determining such bases is of importance and is a
barrier to further development of the theory.  The construction of the
$\mathcal{T}, \mathcal{F}$ and $\mathcal{G}$-functions in this section
represent great progress towards a general solution of the problem.
We have used these functions, along with another new class to solve
the 2 and 3-pole basis problems for the (2,9)-curve (genus four) and
present these in Appendix \ref{APP_29}.

\section{Differential Equations} \label{SEC_DE}

The $\wp$-functions associated to $(n,s)$-curves satisfy a variety of
differential equations which we review in this section.  Some of
these differential equations can be found occurring naturally in areas
of mathematical physics, so the $\wp$-functions can be used to
give solutions to a variety of important problems.  In this section we
consider the three main classes of differential equations and compare
the explicit equations for the two genus three cases.  Note that
although the functions associated to the $(2,7)$ and $(3,4)$-curves are
notationally the same, they behave differently and satisfy different
equations.  This is most apparent from the different weights assigned
to the functions, summarised for the fundamental functions in the table below.

\begin{center}
\begin{tabular}{|c|c|c|c|c|c|c|c|c|c|c|}\hline
                              & $\wp_{11}$ & $\wp_{12}$ & $\wp_{13}$ & $\wp_{22}$ & $\wp_{23}$ & $\wp_{33}$
                              & $\Delta$   & $Q_{1333}$
                              \\ \hline
$\bm{(2,7)}$\textbf{-case} & $-10$  & $-8$  &  $-6$      & $-6$       & $-4$       & $-2$  &$-12$&      \\
$\bm{(3,4)}$\textbf{-case} & $-10$  & $-7$  &  $-6$      & $-4$       & $-3$       & $-2$  &     & $-8$ \\ \hline
\end{tabular}
\end{center}

The sets of differential equations in this paper are presented in decreasing weight order as indicated by the bold number in brackets.  They have all been made available in text files in the supplementary material.  Note that they refer to the functions associated with cyclic curves, but that relations for the general curves can be derived in a similar fashion at a greater computational cost.

We first consider the set of differential equations to express the 4-index $\wp$-functions.
These \textit{\textbf{4-index relations}} are the generalisation of equation (\ref{eq:Intro_elliptic_diff2}) from the elliptic case.  We aim to express each 4-index $\wp$-function as a degree two polynomial in the 2-index $\wp$-functions, in comparison with equation (\ref{eq:Intro_elliptic_diff2}).  Then, through differentiation and manipulation of this set, we could express all higher index $\wp$-functions as polynomials in the 2 and 3-index $\wp$-functions, analogous to the elliptic case.  A complete set of such relations can be obtained for the $(2,7)$-case as first presented by Buchstaber, Enolskii and Leykin in \cite{bel97}.
\begin{align}
\bm{(-4)} \quad
\wp_{3333} &= 4\wp_{23} +4\wp_{33}\lambda_{{6}}+2\lambda_{{5}}+6\wp_{33}^{2}   \label{eq:27_p3333} \\
\bm{(-6)} \quad
\wp_{2333} &= 6\wp_{13} -2\wp_{22} +4\wp_{23} \lambda_{{6}}+6\wp_{23} \wp_{33} \label{eq:27_p2333}\\
&\,\,\, \vdots \nonumber \\
\bm{(-18)} \quad
\wp_{1112} &= 6\wp_{11} \wp_{12} -4\lambda_{{5}}\lambda_{{0}}-2\wp_{22}\lambda_{{1}}+6\wp_{13} \lambda_{{1}}
-8\wp_{23} \lambda_{{0}}+4\wp_{12} \lambda_{{2}} \nonumber \\
\bm{(-20)} \quad
\wp_{1111} &= 6\wp_{11}^{2}+2\lambda_{{3}}\lambda_{{1}}-8\lambda_{{4}}\lambda_{{0}}+4\wp_{11}\lambda_{{2}}
+4\wp_{12} \lambda_{{1}}-12\wp_{22} \lambda_{{0}}+16\wp_{13}\lambda_{{0}} \nonumber
\end{align}
The complete set is displayed in Appendix \ref{APP_4index}.
There are various ways to derive such relations, with a recent survey given in \cite{eeg10}.  A constructive way is to consider the basis for $\Gamma \big( J, \mathcal{O}(2 \Theta^{[2]} ) \big)$ and the set of 4-index Q-functions.  Each Q-function has poles of order at most two and so belongs to the vector space $\Gamma \big( J, \mathcal{O}(2 \Theta^{[2]} ) \big)$.  Hence each Q-function can be expressed as a linear combination of basis entries.  (The explicit linear combination can be identified using the $\sigma$-expansion as discussed in \cite{MEe09}).  In the $(2,7)$-case we have basis (\ref{eq:27_2pole}) containing 2-index $\wp$-functions and $\Delta$.  Hence we can substitute for the Q-functions and $\Delta$ using equations (\ref{eq:4iQ}) and (\ref{eq:27Delta}) to leave the desired set of 4-index relations.

The corresponding set of equations for the $(3,4)$-case was derived in Lemma 6.1 of \cite{eemop07} and are also presented for comparison in Appendix \ref{APP_4index}.
\begin{align*}
\bm{(-4)} \quad
\wp_{3333}  &= -3\wp_{22} +6  \wp_{33}^{2} \\
\bm{(-5)} \quad
\wp_{2333}  &= 3\wp_{33}\lambda_{{3}} +6\wp_{23} \wp_{33} \\
&\,\,\, \vdots \\
\bm{(-17)} \quad
\wp_{1112}  &= -Q_{{1333}}\lambda_{{1}}+6\wp_{33}\lambda_{{3}}\lambda_{{0}}+6\wp_{11} \wp_{12} \\
\bm{(-20)} \quad
\wp_{1111}  &= -4Q_{{1333}}\lambda_{{0}}-3\wp_{33} {\lambda_{{1}}}^{2}
+12\wp_{33}\lambda_{{2}}\lambda_{{0}}+6  \wp_{11}^{2}
\end{align*}
Note that here the function $Q_{1333}$ is used in some of the expressions.  The reason for its inclusion can be explained by considering the constructive method discussed above.  Since $Q_{1333}$ was used in basis (\ref{eq:34_2pole}) for $\Gamma \big( J, \mathcal{O}(2 \Theta^{[2]} ) \big)$ it will appear in the expressions following from this basis.
The ability to construct 4-index relations using only 2-index $\wp$-functions as in the $(2,7)$-case is a feature that appears to be unique to the hyperelliptic cases.  We have explicitly checked that such relations cannot be achieved at certain weights in the $(3,4)$-case.  A more appropriate definition for \textit{\textbf{4-index relations}} seems to be a set that expresses all the 4-index $\wp$-functions using a degree two polynomial in the fundamental basis functions.
(Note that the Q-functions used will only need to appear linearly.)

It is natural to next consider a set of differential equations to generalise equation (\ref{eq:Intro_elliptic_diff1}). Such a set should give expressions for the product of two 3-index $\wp$-functions and so we refer to them as \textit{\textbf{quadratic 3-index relations}}.  The natural generalisation would express each product as a degree three polynomial in 2-index $\wp$-functions, but as in the previous case, this appears to only be possible for the hyperelliptic functions.  We have again explicitly checked that such relations do not exist at certain weights in the $(3,4)$-case and so propose the modified definition of \textit{\textbf{quadratic 3-index relations}} to be a set of differential equations that expresses all the products of 3-index $\wp$-functions using a degree three polynomial in the fundamental basis functions.

The quadratic 3-index relations for the $(2,7)$-case were considered in \cite{bel97}, but a complete set was not presented.  In \cite{CA2008} the corresponding relations for functions associated to covariant curves have been considered, but again a complete set is not directly obtainable from the results published there.

\begin{theorem} \label{thm:Quad_27set}
The quadratic 3-index relations associated to the cyclic $(2,7)$-curve start with those below, with the complete set as displayed in Appendix \ref{SEC_APP27}.
\begin{align}
&\bm{(-6)} &  \wp_{333}^2 &= \textstyle 4 \wp_{33}^{3}
+ 4\lambda_{{4}}
+ 4\lambda_{{5}}\wp_{33}
+ 4\lambda_{{6}} \wp_{33}^{2}
- 4\wp_{13}
+ 4\wp_{22}
+ 4\wp_{33} \wp_{23} \label{eq:27_p333sq}
\\
&\bm{(-8)} & \wp_{233}\wp_{333} &= \textstyle 2 \big( 2\wp_{23}\wp_{33}^{2}
+ \lambda_{{3}}
+ \lambda_{{5}}\wp_{23}
+ 2\lambda_{{6}}\wp_{33}\wp_{23}
+ \wp_{12}
+ 2\wp_{33}\wp_{13}
- \wp_{33}\wp_{22}
+ \wp_{23}^{2} \big) \nonumber
\\
&\bm{(-10)} & \wp_{233}^{2} &= \textstyle 8\wp_{23}\wp_{13}
- 4\wp_{23}\wp_{22}
+ 4\wp_{11}
+ 4\lambda_{{2}}
+ 4\wp_{23}^{2}\wp_{33}
+ 4\lambda_{{6}}\wp_{23}^{2} \nonumber
\\
&\bm{(-10)} & \wp_{223}\wp_{333} &= \textstyle - 4\wp_{33}\wp_{12}
- 2\wp_{23}\wp_{13} + 4\wp_{23}\wp_{22}
+ 4\lambda_{{6}}\wp_{33}\wp_{13}
+ 2\lambda_{{5}}\wp_{33}\wp_{23}
+ 4\lambda_{{4}}\wp_{23} \nonumber \\
& & & \quad \textstyle - 2\wp_{11} + 2\lambda_{{5}}\wp_{13}
- 2\lambda_{{3}}\wp_{33}
+ 2\wp_{22}\wp_{33}^{2}
+ 2\wp_{23}^{2}\wp_{33} \nonumber
\\
&\bm{(-10)} & \wp_{133}\wp_{333} &= \textstyle -2\wp_{33}\wp_{12}
+2\wp_{23}\wp_{13}
+4\lambda_{{6}}\wp_{33}\wp_{13}
-2\wp_{11}+2\lambda_{{5}}\wp_{13}
+4\wp_{13}\wp_{33}^{2} \nonumber
\end{align}
\end{theorem}
\begin{proof}
Once again, these relations can be derived through a variety of methods as discussed in \cite{eeg10}.  They can again be found constructively using the $\sigma$-expansion, although in this case there is no simple linear algebra result to dictate that such relations exist.  Instead we may just search for them using arbitrary polynomials of 2-index $\wp$-functions.  The computations involved can be heavy and so simplifications are made by ensuring the polynomials are homogeneous in weight.  Additionally, it is often possible to allocate the cubic terms to cancel higher order poles using Definition \ref{def:nip} at a low computational cost.   \\
\end{proof}

\begin{lemma} \label{lem:27_Matrix}
Theorem \ref{thm:Quad_27set} may be written concisely as a determinantal formula,
\begin{equation} \label{Q27_Matrix}
\big( \bm{l}^{T} A \bm{k} \big)\big( \bm{l'}^{T} A \bm{k'} \big) = -\frac{1}{4} \left| \begin{array}{ccc}
H     & \bm{l'} & \bm{k'} \\
\bm{l}^T & 0      & 0      \\
\bm{k}^T & 0      & 0
\end{array} \right|,
\end{equation}
using the matrices
\begin{align*}
A &= \left[
\begin{array}{ccccc}
0                     & -\wp_{333}          & \wp_{233}  & -\wp_{223}+\wp_{133} & \wp_{222}-2\wp_{123} \\
\wp_{333}             & 0                   & -\wp_{133} & \wp_{123}            & -\wp_{122}+\wp_{113} \\
-\wp_{233}            & \wp_{133}           & 0          & -\wp_{113}           & \wp_{112}            \\
\wp_{223}-\wp_{133}   & -\wp_{123}           & \wp_{113}  & 0                    & -\wp_{111}           \\
-\wp_{222}+2\wp_{123} & \wp_{122}-\wp_{113} & -\wp_{112} & \wp_{111}            & 0
\end{array}
\right], \\
H &= \left[
\begin{array}{ccccc}
4\lambda_0 & 2\lambda_1           & -2\wp_{11}                     & -2\wp_{12}            & -2\wp_{13}  \\
2\lambda_1 & 4\lambda_2+4\wp_{11} & 2\lambda_3+2\wp_{12}           & -2\wp_{22}+4\wp_{13}  & -2\wp_{23}  \\
-2\wp_{11} & 2\lambda_3+2\wp_{12} & 4\lambda_4+4\wp_{22}-4\wp_{13} & 2\lambda_5+2\wp_{23}  & -2\wp_{33}  \\
-2\wp_{12} & -2\wp_{22}+4\wp_{13} & 2\lambda_5+2\wp_{23}           & 4\lambda_6+4\wp_{33}  & 2           \\
-2\wp_{13} & -2\wp_{23}           & -2\wp_{33}                     & 2                     & 0
\end{array}
\right]
\end{align*}
and arbitrary vectors $\bm{l},\bm{k},\bm{l'},\bm{k'}$ of dimension 5.  Each of the 55 relations in Appendix \ref{SEC_APP27} may be obtained using appropriate choices of the vectors.
\end{lemma}
\begin{proof}
The formula (\ref{Q27_Matrix}) may be verified directly by expanding and substituting with the relations of Theorem \ref{thm:Quad_27set}.  To derive the individual relations from the determinantal formula one must make suitable choices of the vectors  $\bm{l},\bm{k},\bm{l'},\bm{k'}$ with only one entry in each vector non-zero.  For example, setting
\[
\bm{k} = \bm{k}^T = (1,0,0,0,0), \qquad
\bm{l} = \bm{l}^T = (0,1,0,0,0)
\]
into equation (\ref{Q27_Matrix}) leaves relation (\ref{eq:27_p333sq}) for $\wp_{333}^2$.  One restriction on this approach is that some equations must be derived before others.  For example, the equation for $\wp_{133}^2$ must be derived before the equation for $\wp_{133}\wp_{223}$ as the choice of vectors to give $\wp_{133}\wp_{223}$ on the left hand side will introduce $\wp_{133}^2$ as well.  However, when this is the case it is always possible to derive one of the other quadratic term independently and then substitute for it to find the other.  Hence, all the 55 individual relations may be derived from the formula. \\
\end{proof}
Lemma \ref{lem:27_Matrix} was inspired by the results of Athorne in \cite{CA2008} on covariant hyperelliptic curves and the corresponding $\wp$-functions.  These curves and functions belong to generic families permuted under an $sl_2$ action which can be easily mapped to the standard curves and functions considered here.  Recently, the covariant result corresponding to Lemma \ref{lem:27_Matrix} was developed in \cite{CA11}.

The corresponding quadratic relations for the $(3,4)$-case were first considered in \cite{eemop07}, but again a complete set has not been available until now.  These relations include the basis function $Q_{1333}$ which occurs only linearly, or multiplied by a single 2-index $\wp$-function.  A determinantal version of these equations, one similar to equation (\ref{Q27_Matrix}), has not been identified.  One may follow from development of the corresponding covariant theory to trigonal curves.

\newpage

\begin{theorem} \label{thm:Quad_34set}
The quadratic 3-index relations associated to the cyclic $(3,4)$-curve starts with those below with the complete set as given in Appendix \ref{APP_34quad}.
\begin{align*}
&\bm{(-6)} &  \wp_{333}^2 &= \textstyle 4\wp_{33}^3 + 4\wp_{13} + \wp_{23}^2 - 4\wp_{22}\wp_{33}
\\
&\bm{(-7)} &  \wp_{233}\wp_{333} &= \textstyle 4\wp_{23}\wp_{33}^2
- 2\wp_{12} - \wp_{22}\wp_{23} + 2\lambda_{3}\wp_{33}^2
\\
&\bm{(-8)} &  \wp_{233}^2 &= \textstyle 4\wp_{23}^2\wp_{33} + \wp_{22}^2
+ 4\lambda_{3}\wp_{23}\wp_{33} + 4\wp_{33}\lambda_{2} - \frac{4}{3}Q_{1333}
\\
&\bm{(-8)} &  \wp_{223}\wp_{333} &= \textstyle 2\wp_{23}^2\wp_{33}
+ 2\wp_{22}\wp_{33}^2 - 2\wp_{22}^2 + \lambda_{3}\wp_{23}\wp_{33}
+ 4\wp_{13}\wp_{33} + \frac{2}{3}Q_{1333}
\\
&\bm{(-9)} &  \wp_{223}\wp_{233} &= \textstyle 2\wp_{13}\lambda_{3} + 2\lambda_{1}
+ 4\wp_{13}\wp_{23} + 2\wp_{23}\wp_{22}\wp_{33} + 2\wp_{23}^3
\\& & &\quad \textstyle
+ \wp_{22}\wp_{33}\lambda_{3} + 2\wp_{23}^2\lambda_{3} + 2\wp_{23}\lambda_{2}
\\
&\bm{(-9)} &  \wp_{222}\wp_{333} &= \textstyle 6\wp_{23}\wp_{22}\wp_{33}
- 4\wp_{12}\wp_{33} - 2\wp_{23}^3 + 4\wp_{22}\wp_{33}\lambda_{3}
- 8\wp_{13}\wp_{23}
- \wp_{23}^2\lambda_{3} - 4\wp_{13}\lambda_{3}
\end{align*}
\end{theorem}
\begin{proof}
These relations can be derived using a variety of methods as in Theorem \ref{thm:Quad_27set}.\\
\end{proof}

The final set of differential equations that we consider here are a set bilinear in the 2 and 3-index $\wp$-functions.  Due to the parity properties of the $\wp$-functions we know that these \textit{bilinear relations} cannot contain any constant terms, or terms dependent only on the 2-index $\wp$-functions.  There is no analogue of these relations in the genus one case.  Three of the $(2,7)$ bilinear relations displayed below were derived in \cite{bel97} as part of a larger computation, but the complete set below is a new result.

\begin{theorem} \label{thm:BL_27set}
Every bilinear relation associated with the cyclic $(2,7)$-curve may be given as a linear combination of the 24 below.
\begin{align*}
\bm{(-7)} \quad 0&= \textstyle
\wp_{233} \wp_{33}
+\wp_{223}
-\wp_{333} \wp_{23}
-\wp_{133}
\\
\bm{(-9)} \quad 0&= \textstyle
\wp_{133} \wp_{33}
+\wp_{123}
-\wp_{333} \wp_{13}
\\
\bm{(-9)} \quad 0&= \textstyle 
-2\wp_{33} \wp_{223}
+\wp_{22}\wp_{333}
+\wp_{23} \wp_{233}
+\lambda_{{5}}\wp_{233}
-4\wp_{123}
+2\lambda_{{6}}\wp_{133}
+\wp_{222} -2\lambda_{{6}}\wp_{223}
\\
\bm{(-11)} \quad 0&= \textstyle
-\wp_{133} \wp_{23}
-\wp_{113}
+\wp_{233} \wp_{13}
\\
\bm{(-11)} \quad 0&= \textstyle 
\wp_{33} \wp_{222}
-2\wp_{22} \wp_{233}
+\wp_{23} \wp_{223}
+\lambda_{{5}}\wp_{223}
-\wp_{122}
+2\lambda_{{6}}\wp_{123}
+3\wp_{113}
-2\lambda_{{5}}\wp_{133} \\
&\quad \textstyle +\wp_{333} \lambda_{{3}}
-2\wp_{233} \lambda_{{4}}
\\
\bm{(-11)} \quad 0&= \textstyle 
-\wp_{133} \wp_{23}
-2\wp_{33}\wp_{123}
-3\wp_{113} +\wp_{122}
-2\lambda_{{6}}\wp_{123}
+\wp_{12} \wp_{333}
+2\wp_{233} \wp_{13}
+\lambda_{{5}}\wp_{133}
\\
\bm{(-13)} \quad 0&= \textstyle
\wp_{112} -\wp_{222} \wp_{23}
+\wp_{22} \wp_{223}
+\wp_{233}\lambda_{{3}}+2\wp_{133} \wp_{13}
-2\wp_{22} \wp_{133}
+2\wp_{12} \wp_{233} \\
&\quad \textstyle -2\wp_{11} \wp_{333}
-2\wp_{333} \lambda_{{2}}
-2\lambda_{{6}}\wp_{113}
\\
\bm{(-13)} \quad 0&= \textstyle  
-\wp_{123} \wp_{23}
-\wp_{133} \wp_{13}
+\wp_{13} \wp_{223}
+\wp_{33} \wp_{113}
\\
\bm{(-13)} \quad 0&= \textstyle 
-\wp_{22} \wp_{133}
-2\wp_{123} \wp_{23}
+\wp_{112} -2\lambda_{{6}}\wp_{113}
+\wp_{12} \wp_{233}
+2\wp_{13} \wp_{223}
\\
\bm{(-13)} \quad 0&= \textstyle 
\wp_{122} \wp_{33}
-2\wp_{22} \wp_{133}
+\lambda_{{5}}\wp_{123}
+\wp_{133} \wp_{13}
+\wp_{13} \wp_{223}
-\wp_{11} \wp_{333}
-2\wp_{133} \lambda_{{4}}
\\
\bm{(-15)} \quad 0&= \textstyle
-\lambda_{{5}}\wp_{113}
+2\wp_{133}\lambda_{{3}}
+2\wp_{12} \wp_{223}
+\wp_{22} \wp_{123}
-3\wp_{333}\lambda_{{1}}
+2\wp_{123} \wp_{13} \\
&\quad \textstyle -2\wp_{222} \wp_{13}
-\wp_{122} \wp_{23}
-\wp_{112} \wp_{33}
-\wp_{11} \wp_{233}
\\
\bm{(-15)} \quad 0&= \textstyle 
-\lambda_{{5}}\wp_{113}
+\wp_{133} \lambda_{{3}}
+\wp_{12} \wp_{223}
-\wp_{333} \lambda_{{1}}
+\wp_{111}
-\wp_{122}\wp_{23}
\\
\bm{(-15)} \quad 0&= \textstyle 
-\wp_{133} \lambda_{{3}}
+4\wp_{123} \wp_{13}
-2\wp_{12} \wp_{133}
-\wp_{11} \wp_{233}
-\wp_{23} \wp_{113}
+\wp_{122} \wp_{23}\\
&\quad \textstyle +\wp_{12} \wp_{223}
-\wp_{222} \wp_{13}
-\wp_{22} \wp_{123}
\\
\bm{(-15)} \quad 0&= \textstyle 
\wp_{133} \lambda_{{3}}
+\wp_{12} \wp_{133}
+\wp_{11} \wp_{233}
-\wp_{23} \wp_{113}
-\wp_{112} \wp_{33}
-\lambda_{{5}}\wp_{113}
\\
\bm{(-17)} \quad 0&= \textstyle
-\lambda_{{5}}\wp_{112}
+\wp_{12} \wp_{222}
-\wp_{22} \wp_{122}
-\wp_{233} \lambda_{{1}}
-4\wp_{13} \wp_{113}
+2\wp_{13} \wp_{122}
+2\wp_{22} \wp_{113}\\
&\quad -2\wp_{23} \wp_{112}
-2\wp_{12} \wp_{123}
+2\wp_{33} \wp_{111}
+2\wp_{11} \wp_{133}
+2\wp_{133} \lambda_{{2}}
+2\lambda_{{6}}\wp_{111}
\\
\bm{(-17)} \quad 0&= \textstyle 
-\wp_{22} \wp_{113}
+\wp_{11} \wp_{223}
+3\wp_{12} \wp_{123}
-\wp_{233} \lambda_{{1}}
+2\wp_{133}\lambda_{{2}}
-4\wp_{333} \lambda_{{0}}
+\wp_{123} \lambda_{{3}}\\
&\quad +2\wp_{13} \wp_{113}
-2\wp_{13} \wp_{122}
-\wp_{23}\wp_{112}
-\wp_{33} \wp_{111}
-\wp_{11} \wp_{133}
-2\wp_{113}\lambda_{{4}}
\\
\bm{(-17)} \quad 0&= \textstyle 
\wp_{22} \wp_{113}
+2\wp_{12} \wp_{123}
-\wp_{233} \lambda_{{1}}
+2\wp_{133} \lambda_{{2}}
-4\wp_{333} \lambda_{{0}}
-2\wp_{13} \wp_{122}
-\wp_{23} \wp_{112}
\\
\bm{(-17)} \quad 0&= \textstyle 
-\wp_{233} \lambda_{{1}}
-3\wp_{13} \wp_{113}
+3\wp_{11} \wp_{133}
+2\wp_{133} \lambda_{{2}}
+2\wp_{333}\lambda_{{0}}
-\wp_{23} \wp_{112}
-\wp_{12} \wp_{123}\\
&\quad +\wp_{13} \wp_{122}
+\wp_{22} \wp_{113}
\end{align*}
\begin{align*}
\bm{(-19)} \quad 0&= \textstyle
-2\wp_{22} \wp_{112}
+\wp_{11}\wp_{222}
+\wp_{12} \wp_{122}
-\lambda_{{1}}\wp_{223}
+2\lambda_{{2}}\wp_{123}
-4\wp_{233} \lambda_{{0}}
+\wp_{122} \lambda_{{3}}\\
&\quad +\lambda_{{5}}\wp_{111}
-2\lambda_{{4}}\wp_{112}
-2\wp_{113} \lambda_{{3}}
+3\wp_{133} \lambda_{{1}}
\\
\bm{(-19)} \quad 0&= \textstyle 
\wp_{133} \lambda_{{1}}
+\wp_{12} \wp_{113}
-2\wp_{233} \lambda_{{0}}
-\wp_{13} \wp_{112}
\\
\bm{(-19)} \quad 0&= \textstyle 
-3\wp_{133} \lambda_{{1}}
+\lambda_{{1}}\wp_{223}
-\wp_{12} \wp_{113}
-2\wp_{11} \wp_{123}
-2\lambda_{{2}}\wp_{123}
+4\wp_{233} \lambda_{{0}}
+\wp_{113} \lambda_{{3}} \\
&\quad +2\wp_{13} \wp_{112}
+\wp_{23} \wp_{111}
\\
\bm{(-21)} \quad 0&= \textstyle
-4\lambda_{{1}}\wp_{123}
+\lambda_{{1}}\wp_{222}
+\wp_{12} \wp_{112}
-2\wp_{11} \wp_{122}
-2\wp_{122} \lambda_{{2}}
+4\wp_{223} \lambda_{{0}}
+\wp_{22} \wp_{111} \\
&\quad +2\lambda_{{2}}\wp_{113}
-4\wp_{133} \lambda_{{0}}
+\wp_{112} \lambda_{{3}}
\\
\bm{(-21)} \quad 0&= \textstyle 
-\wp_{11} \wp_{113} -\lambda_{{1}}\wp_{123}
-2\wp_{133} \lambda_{{0}}+2\wp_{223} \lambda_{{0}}
+\wp_{13} \wp_{111}
\\
\bm{(-23)} \quad 0&= \textstyle
-\wp_{11} \wp_{112}
-\wp_{122} \lambda_{{1}}
-4\lambda_{{0}}\wp_{123}
+2\lambda_{{0}} \wp_{222}
+\lambda_{{1}}\wp_{113}
+\wp_{12} \wp_{111}
\end{align*}
\end{theorem}
\begin{proof}
The simplest way to construct these relations is through cross multiplication of the 4-index relations.  For example, substituting using equations (\ref{eq:27_p3333}) and (\ref{eq:27_p2333}) into
\[
\frac{\partial}{\partial u_2} \big( \wp_{3333} \big) - \frac{\partial}{\partial u_3} \big( \wp_{2333} \big) = 0
\]
gives the first relation  in the set above. \\
\end{proof}

The bilinear relations in the $(3,4)$-case were considered in \cite{eemop07} but again, a complete set was not available until now.

\begin{theorem} \label{thm:BL_34set}
Every bilinear relation associated with the cyclic $(3,4)$-curve may be given as a linear combination of the 21 below.
\begin{align*}
\bm{(-6)} \quad 0&= \textstyle  \wp_{222} +\lambda_{{3}}\wp_{333}
+ 2\wp_{23} \wp_{333} - 2\wp_{33} \wp_{233} \\
\bm{(-7)} \quad 0&= \textstyle  2\wp_{133} +\wp_{22} \wp_{333} +\wp_{23} \wp_{233}
- 2\wp_{33} \wp_{223}
\\
\bm{(-8)} \quad 0&= \textstyle  -4\wp_{123} +4\wp_{22} \wp_{233} -2\wp_{23} \wp_{223}
- 2\wp_{33} \wp_{222} \\
\bm{(-9)} \quad 0&= \textstyle  4\wp_{122} - \lambda_{{3}}\wp_{222} -  {\lambda_{{3}}}^{2}\wp_{333}
+ 4\lambda_{{2}}\wp_{333} - 2\wp_{23} \wp_{222} + 2\wp_{22} \wp_{223}
\\ &\qquad \textstyle
+ 8\wp_{13} \wp_{333} - 8\wp_{33} \wp_{133}
\\
\bm{(-10)} \quad 0&= \textstyle  -\wp_{12} \wp_{333} +\wp_{23} \wp_{133}
+ 2\wp_{33} \wp_{123} - 2\wp_{13} \wp_{233} \\
\bm{(-11)} \quad 0&= \textstyle  4\wp_{113} -3\wp_{12} \wp_{233} -\wp_{22} \wp_{133}
+ 2\wp_{23} \wp_{123} + 2\wp_{33} \wp_{122} \\
\bm{(-11)} \quad 0&= \textstyle  \wp_{113} -3\wp_{13} \wp_{223} + 2\wp_{22} \wp_{133}
+ 2\wp_{23} \wp_{123} - \wp_{33} \wp_{122} \\
\bm{(-12)} \quad 0&= \textstyle  \lambda_{{3}} \big( \wp_{13}\wp_{333} - \wp_{33}\wp_{133} \big)
+ \lambda_{{1}}\wp_{333} +\wp_{12} \wp_{223} - \wp_{23} \wp_{122} \\
\bm{(-12)} \quad 0&= \textstyle \wp_{13} \wp_{222} -\wp_{112}
- 2\lambda_{{3}} \big( \wp_{13} \wp_{333} -\wp_{33} \wp_{133} \big)
- 2\lambda_{{1}}\wp_{333}
- 2\wp_{22} \wp_{123} +\wp_{23} \wp_{122} \\
\bm{(-13)} \quad 0&= \textstyle  3\lambda_{{3}} \big( \wp_{13} \wp_{233} -\wp_{23} \wp_{133} \big)
- 2\lambda_{{2}}\wp_{133} +\lambda_{{1}}\wp_{233} + 3\wp_{12} \wp_{222}
\\ &\qquad \textstyle
- 3\wp_{22} \wp_{122} + 4\wp_{11} \wp_{333} + 4\wp_{13} \wp_{133} - 8\wp_{33} \wp_{113}
\\
\bm{(-14)} \quad 0&= \textstyle  -4\wp_{12} \wp_{133} + 4\wp_{13} \wp_{123}
- 2\wp_{23} \wp_{113} +2\wp_{33} \wp_{112}
\\
\bm{(-14)} \quad 0&= \textstyle  \frac{9}{2}\lambda_{{3}}\wp_{113} -3\lambda_{{2}}\wp_{123}
+ \frac{3}{2}\lambda_{{1}}\wp_{223} -3\wp_{11} \wp_{233} - \wp_{12} \wp_{133} - 2\wp_{13} \wp_{123}
\\ &\qquad \textstyle
+ 4\wp_{23} \wp_{113} + 2\wp_{33} \wp_{112}
\\
\bm{(-15)} \quad 0&= \textstyle  -\wp_{11} \wp_{223} -\frac{2}{3}\wp_{12} \wp_{123}
- \frac{1}{3}\wp_{13} \wp_{122} +\frac{2}{3}\wp_{22} \wp_{113} + \frac{4}{3}\wp_{23} \wp_{112}
+ \frac{1}{3}\lambda_{{1}}\wp_{222} -\wp_{111} \\
&\qquad \textstyle + \frac{3}{2}\lambda_{{3}}\wp_{112}
- \frac{2}{3}\lambda_{{2}} \big( \wp_{13}\wp_{333} -\wp_{33} \wp_{133} \big)
- \lambda_{{2}}\wp_{122}
- \frac{1}{6}\lambda_{{3}}\lambda_{{1}}\wp_{333} -\frac{4}{3}\lambda_{{0}}\wp_{333}
\\
\bm{(-15)} \quad 0&= \textstyle  - \frac{2}{3}\wp_{12} \wp_{123} + \frac{2}{3}\wp_{13} \wp_{122} -\frac{1}{3}\wp_{22} \wp_{113} + \frac{1}{3}\wp_{23} \wp_{112} \\
&\qquad \textstyle - \frac{2}{3}\lambda_{{2}} \big( \wp_{13} \wp_{333} - \wp_{33} \wp_{133} \big)
- \frac{1}{6}\lambda_{{1}}\wp_{222} - \frac{1}{6}\lambda_{{3}}\lambda_{{1}}\wp_{333}
- \frac{4}{3}\lambda_{{0}}\wp_{333}
\\
\bm{(-16)} \quad 0&= \textstyle
-\lambda_{{3}} \big( \wp_{11} \wp_{333} + \wp_{13} \wp_{133} -2\wp_{33} \wp_{113} \big)
- 2\lambda_{{2}} \big( \wp_{13} \wp_{233} - \wp_{23} \wp_{133} \big) \\
&\qquad \textstyle + 3\lambda_{{1}}\wp_{133} - 4\lambda_{{0}}\wp_{233}
- \wp_{11} \wp_{222} -\wp_{12} \wp_{122} + 2\wp_{22} \wp_{112}
\\
\bm{(-17)} \quad 0&= \textstyle  \frac{2}{3}\lambda_{{2}}\wp_{113}
- \frac{4}{3}\lambda_{{1}}\wp_{123} + 2\lambda_{{0}}\wp_{223} - \frac{4}{3}\wp_{11} \wp_{133}
+ \frac{2}{3}\wp_{13} \wp_{113} + \frac{2}{3}\wp_{33} \wp_{111}
\end{align*}
\begin{align*}
\bm{(-18)} \quad 0&= \textstyle  -\lambda_{{3}}\lambda_{{0}}\wp_{333}
+ \lambda_{{0}}\wp_{222} +\frac{2}{3}\lambda_{{2}}\wp_{112}
- \frac{4}{3}\wp_{11} \wp_{123} + \frac{1}{3}\wp_{12} \wp_{113}  \\
&\qquad \textstyle + \frac{1}{3}\wp_{13} \wp_{112}
+ \frac{2}{3}\wp_{23} \wp_{111} - \frac{4}{3}\lambda_{{1}}\wp_{122}
- \lambda_{{1}} \big( \wp_{13} \wp_{333} -\wp_{33} \wp_{133} \big)
\\
\bm{(-18)} \quad 0&= \textstyle  -\lambda_{{3}}\lambda_{{0}}\wp_{333}
- 2\lambda_{{0}}\wp_{222} - \frac{1}{3}\lambda_{{2}}\wp_{112}
- \frac{5}{3}\wp_{12} \wp_{113} + \frac{2}{3}\wp_{11} \wp_{123}  \\
&\qquad \textstyle + \frac{4}{3}\wp_{13} \wp_{112}
- \frac{1}{3}\wp_{23} \wp_{111} + \frac{2}{3}\lambda_{{1}}\wp_{122}
-\lambda_{{1}} \big( \wp_{13} \wp_{333} -\wp_{33} \wp_{133} \big)
\\
\bm{(-19)} \quad 0&= \textstyle  \frac{8}{3}\lambda_{{0}}\wp_{133}
- \frac{4}{9}{\lambda_{{2}}}^{2}\wp_{133} - 2\lambda_{{1}} \big( \wp_{13} \wp_{233}
- \wp_{23} \wp_{133} \big)
+\wp_{13} \wp_{133} -2\wp_{33} \wp_{113} \big)  \\
&\qquad \textstyle + \frac{2}{9}\lambda_{{2}}\lambda_{{1}}\wp_{233}
+ \frac{4}{3}\lambda_{{3}}\lambda_{{1}}\wp_{133}
- 2\lambda_{{3}}\lambda_{{0}}\wp_{233} - \frac{4}{3}\wp_{11} \wp_{122}
+ \frac{2}{3}\wp_{12} \wp_{112} + \frac{2}{3}\wp_{22} \wp_{111} \\
\bm{(-21)} \quad 0&= \textstyle  \frac{1}{2}{\lambda_{{1}}}^{2}\wp_{333}
- 2\lambda_{{2}}\lambda_{{0}}\wp_{333}
- 3\lambda_{{0}}\wp_{122} - 4\lambda_{{0}} \big( \wp_{13} \wp_{333}
- \wp_{33} \wp_{133} \big)
\\ &\qquad \textstyle
+ \frac{1}{2}\lambda_{{1}}\wp_{112} + \wp_{13} \wp_{111} - \wp_{11} \wp_{113}
\\
\bm{(-22)} \quad 0&= \textstyle  -\frac{1}{3}\lambda_{{2}}\lambda_{{1}}\wp_{133}
- \wp_{11} \wp_{112} + \wp_{12} \wp_{111}
- 4\lambda_{{0}} \big( \wp_{13} \wp_{233} -\wp_{23} \wp_{133} \big) \\
&\qquad \textstyle
- \frac{1}{3}\lambda_{{1}} \big( \wp_{11} \wp_{333} + \wp_{13} \wp_{133} - 2\wp_{33} \wp_{113} \big)
+ \frac{2}{3}{\lambda_{{1}}}^{2}\wp_{233}
+ 3\lambda_{{3}}\lambda_{{0}}\wp_{133} - 2\lambda_{{2}}\lambda_{{0}}\wp_{233}
\end{align*}
\end{theorem}
\begin{proof}
Similarly to Theorem \ref{thm:BL_27set}, these can be derived through cross-multiplication of the 4-index relations.  However, the existence of $Q_{1333}$ in the 4-index relations means that more care has to be taken in the choice of cross products.  In higher genus trigonal cases, (or in the case where $n>3$) the inclusion of further Q-functions in the basis makes this method increasingly tricky.  An alternative method to systematically find bilinear relations has been developed and is discussed in \cite{ME11}.  \\
\end{proof}

\section{Addition Formulae} \label{SEC_AF}

Here we discuss the addition formulae satisfied by the Abelian functions and present some new formulae associated with automorphisms of the curves.  We start by considering the formulae which generalise equation (\ref{eq:Intro_elliptic_add}) from the elliptic case.  Such a formula will exist for every $(n,s)$-curve as demonstrated by the following theorem.
\begin{theorem} \label{thm:2tadd}
Given an $(n,s)$-curve, the associated functions satisfy a \textit{two-term two-variable addition formula} of the form
\begin{equation} \label{eq:2tadd}
\frac{\sigma(\bu + \bv)\sigma(\bu - \bv)}{\sigma(\bu)^2\sigma(\bv)^2} = \sum_i c_i A_i(\bu)B_i(\bv)
\end{equation}
where the functions $A_i(\bu), B_i(\bv)$ belong to the basis for
$\Gamma \big( J, \mathcal{O}(2 \Theta^{[g-1]} ) \big)$ and the $c_i$ are constants.  Further, the polynomial on the right hand side will either be symmetric or anti-symmetric with respect to the change of variables $(\bu,\bv) \mapsto (\bv,\bu)$ when the $\sigma$-function is odd or even respectively.
\end{theorem}
\begin{proof}
Denote the left hand side of equation (\ref{eq:2tadd}) by LHS$(\bu,\bv)$.  Clearly this has poles of order at most two and so we just need to prove that it is Abelian.  We let $\ell$ be a point in the lattice and use the quasi-periodicity condition (\ref{eq:HG_quas}) with
$\Psi = \big( \eta'\bm{\ell'} + \eta''\bm{\ell''} \big)^T$ to check that
\begin{align*}
\text{LHS}(\bu+\ell,\bv) &= \frac{\sigma(\bu+\ell+\bv)\sigma(\bu+\ell-\bv)}{\sigma(\bu+\ell)^2\sigma(\bv)^2}
= \frac{ \chi(\ell)e^{\Psi [\bu+\bv+\frac{\ell}{2}]}\sigma(\bu+\bv)
         \chi(\ell)e^{\Psi [\bu-\bv+\frac{\ell}{2}]}\sigma(\bu-\bv) }
{\chi(\ell)^2e^{2\Psi [\bu+\frac{\ell}{2}]}\sigma(\bu)^2\sigma(\bv)^2} \\
&= \frac{\sigma(\bu+\bv)\sigma(\bu-\bv)}{\sigma(\bu)^2\sigma(\bv)^2}
e^{\Psi[\bu+\bv+\frac{\ell}{2}+\bu-\bv+\frac{\ell}{2}-2u-\ell]} = \text{LHS}(\bu,\bv).
\end{align*}
Then using the parity property (\ref{eq:sigparity}) with $k=1/24(n^2-1)(s^2-1)$ we can check that
\begin{align*}
\text{LHS}(\bu,\bv+\ell) &= \frac{\sigma(\bu+\bv+\ell)\sigma(\bu-\bv-\ell)}{\sigma(\bu)^2\sigma(\bv+\ell)^2}
= \frac{\sigma(\bv+\ell+\bu)(-1)^k\sigma(\bv+\ell-\bu)}{\sigma(\bv+\ell)^2\sigma(\bu)^2} \\
&= (-1)^k\text{LHS}(\bv+\ell,\bu) = (-1)^k\text{LHS}(\bv,\bu)
= (-1)^k\frac{\sigma(\bv+\bu)\sigma(\bv-\bu)}{\sigma(\bv)^2\sigma(\bu)^2} \\
&= (-1)^{2k}\frac{\sigma(\bu+\bv)\sigma(\bu-\bv)}{\sigma(\bu)^2\sigma(\bv)^2} = \text{LHS}(\bu,\bv).
\end{align*}
Hence the left hand side is Abelian with respect to both $\bu$ and $\bv$ and so the right hand side must be expressed using the basis elements for $\Gamma \big( J, \mathcal{O}(2 \Theta^{[g-1]} ) \big)$ as described.  The further symmetry property of the right hand side can be concluded by simply applying the symmetry property of $\sigma(\bu)$ to LHS$(\bv,\bu)$.  \\
\end{proof}

The coefficients in the right hand side of the formulae can be explicitly determined using the $\sigma$-expansion.  (See \cite{MEe09} for details of such calculations).  In \cite{eemop07} the authors showed that the functions associated with the cyclic $(3,4)$-curve satisfy
\begin{align}
\frac{\sigma(\bu + \bv)\sigma(\bu - \bv)}{\sigma(\bu)^2\sigma(\bv)^2}
&= - \wp_{11}(\bu) \textstyle
- \frac{1}{3}Q_{1333}(\bv)\wp_{33}(\bu) + \wp_{12}(\bv)\wp_{23}(\bu) + \wp_{13}(\bv)\wp_{22}(\bu) \nonumber  \\
&\qquad + \wp_{11}(\bv) \textstyle
+ \frac{1}{3}Q_{1333}(\bu)\wp_{33}(\bv) - \wp_{12}(\bu)\wp_{23}(\bv) - \wp_{13}(\bu)\wp_{22}(\bv). \label{eq:34_2t2v}
\end{align}
Note that the (3,4) $\sigma$-function is odd and hence the addition formula here is anti-symmetric in $(\bu,\bv)$.  The corresponding formula for the functions associated to the cyclic $(2,7)$-curve is
\begin{align}
\frac{\sigma(\bu + \bv)\sigma(\bu - \bv)}{\sigma(\bu)^2\sigma(\bv)^2}
&= \Delta(\bu) - \wp_{11}(\bv)\wp_{33}(\bu) - \wp_{22}(\bu)\wp_{13}(\bv) + \wp_{12}(\bv)\wp_{23}(\bu)
+ 2\wp_{13}(\bu)\wp_{13}(\bv) \nonumber \\
&\quad + \Delta(\bv) - \wp_{11}(\bu)\wp_{33}(\bv) - \wp_{22}(\bv)\wp_{13}(\bu) + \wp_{12}(\bu)\wp_{23}(\bv), \label{eq:27_2t2v}
\end{align}
as first established in \cite{bel97}.  This time the $\sigma$-function is even and hence the formula is symmetric in $(\bu,\bv)$.

In some cases there are more addition formulae associated with the functions, resulting from automorphisms of the curve equation.  Such addition formulae were the topic of \cite{emo11} which gave a thorough treatment of the genus one and two cases.
We will present two new genus three addition formulae associated with the cyclic $(3,4)$-curve.  The first of these is related to the automorphism on the curve (\ref{eq:c34}) given by the operator
\[
[\zeta]: (x,y) \mapsto (x, \zeta y), \quad \text{where } \zeta=\exp\left( \frac{2\pi i}{3} \right).
\]
So $\zeta$ is a cube root of unity and $[\zeta]$ an operator which multiplies $y$ by the root leaving the curve unchanged.  We extend this notation to define the sequence of operators and automorphisms,
\[
[\zeta^j]: (x,y) \mapsto (x, \zeta^j y), \quad \text{for } j \in \Z.
\]
We can check using the basis of differentials (\ref{eq:du34}) that these operators act on the variables $\bu$ as follows.
\begin{equation} \label{eq:zetau}
[\zeta^j]\bu = ( \zeta^j u_1, \zeta^j u_2, \zeta^{2j} u_3 ).
\end{equation}
The action of such operators on the lattice $\Lambda$ is stable, (it moves the points around but does not change the overall lattice).  This can be checked by considering the effect on the individual elements of the period matrices, (see \cite{on98} for more details).  We can now derive the following result for the $\sigma$-function which follows Lemma 4.2.5 in \cite{on98}.

\begin{lemma} \label{lem:sigzeta}
The $\sigma$-function associated to the cyclic $(3,4)$-curve satisfies
\begin{equation} \label{sigma_zeta}
\sigma( [\zeta^j]\bu ) = \zeta^j\sigma(\bu).
\end{equation}
\end{lemma}
\begin{proof}
Consider the quasi-periodicity of $\sigma( [\zeta^j]\bu )$.  If $\ell$ is a point on the lattice then
\begin{align*}
\sigma( [\zeta^j](\bu+\ell) ) = \sigma( [\zeta^j]\bu + [\zeta^j]\ell ).
\end{align*}
Since the lattice is stable under the action we know that $[\zeta^j]\ell$ is also on the lattice.  Hence by equation (\ref{eq:HG_quas})
\begin{align*}
\sigma( [\zeta^j](\bu+\ell) )
= \chi([\zeta^j]\ell) \sigma([\zeta^j]\bu) \exp \Big[ L \Big( [\zeta^j] \Big(\bu+\frac{\ell}{2}\Big), [\zeta]\ell \Big) \Big]
\end{align*}
In \cite{on98} the author shows that for an automorphism of a cyclic curve we have
\[
L([\zeta]\bu,\bv) = L(\bu,[\zeta^{-1}]\bv)
\]
and hence
\[
L([\zeta^j]\bu,[\zeta^j]\bv) = L(\bu,[\zeta^{-j}][\zeta^j]\bv) = L(\bu,\bv).
\]
Therefore, we have
\[
\sigma( [\zeta^j](\bu+\ell) )
= \chi([\zeta^j]\ell) \sigma([\zeta^j]\bu) \exp \Big[ L \Big( \bu+\frac{\ell}{2}, \ell \Big) \Big]
\]
We now consider the quotient
\[
\frac{\sigma( [\zeta^j](\bu+\ell) )}{\sigma( (\bu+\ell) )}
= \frac{\chi([\zeta^j]\ell)}{\chi(\ell)} \frac{\sigma([\zeta^j]\bu)}{\sigma(\bu)}
= \pm \frac{\sigma([\zeta^j]\bu)}{\sigma(\bu)}
\]
since $\chi(\ell) = \pm 1$. So we see that the function
\[
\frac{\sigma([\zeta^j]\bu)}{\sigma(\bu)}
\]
is bounded and entire (since the zero sets coincide).  Hence, by Liouville's theorem, the function is a constant.  Using the Schur-Weierstrass polynomial (\ref{eq:34SW}) we see that this constant is $\zeta^j$.  \\
\end{proof}

We can now derive the addition formula associated with these automorphisms.  Note that this is a more general version of the formula presented in Theorem 10.1 of \cite{eemop07}.
\begin{theorem} \label{thm:34_3t3v}
The functions associated to the cyclic $(3,4)$-curve satisfy
\begin{align*}
&\frac{\sigma(\bu+\bv+\bw)\sigma(\bu+[\zeta]\bv+[\zeta^2]\bw)\sigma(\bu+[\zeta^2]\bv+[\zeta]\bw) }
{ \sigma(\bu)^3 \sigma(\bv)^3\sigma(\bw)^3 } \\
&\qquad = f(\bu,\bv,\bw) + f(\bu,\bw,\bv) + f(\bv,\bu,\bw) + f(\bv,\bw,\bu) + f(\bw,\bu,\bv) + f(\bw,\bv,\bu)
\end{align*}
where
\begin{align*}
f(\bu,\bv,\bw)
&= \Big[ P_{30} + P_{27} + P_{24} + P_{21} + P_{18} + P_{15} + P_{12} + P_{9} + P_{6} + P_{3} + P_{0} \Big](\bu,\bv,\bw)
\end{align*}
and the polynomials $P_k(\bu,\bv,\bw)$ are as presented in Appendix \ref{APP_34add}.
\end{theorem}
\begin{proof}
Denote the left hand side of the formula by LHS$(\bu,\bv,\bw)$.  Using Lemma \ref{lem:sigzeta} and the parity property of the $\sigma$-function we first check that LHS$(\bu,\bv,\bw)$ is symmetric under all permutations of $(\bu,\bv,\bw)$.  Next consider the affect of $\bu \mapsto \bu + \ell$.
\begin{align*}
&\text{LHS}(\bu+\ell,\bv,\bw)
= \frac{\sigma(\bu+\ell+\bv+\bw)\sigma(\bu+\ell+[\zeta]\bv+[\zeta^2]\bw)\sigma(\bu+\ell+[\zeta^2]\bv+[\zeta]\bw)}
{\sigma(\bu+\ell)^3\sigma(\bv)^3\sigma(\bw)^3} \\
&\qquad = \text{LHS}(\bu,\bv,\bw)\frac{\chi(\ell)e^{\Psi[\bu+\bv+\bw+\ell/2]}
\chi(\ell)e^{\Psi[\bu+[\zeta]\bv+[\zeta^2]\bw+\ell/2]}
\chi(\ell)e^{\Psi[\bu+[\zeta^2]\bv+[\zeta]\bw+\ell/2]}}
{\chi(\ell)^3e^{3\Psi[\bu+\ell/2]}} \\
&\qquad = \text{LHS}(\bu,\bv,\bw) e^{\Psi\bv[1 + [\zeta] + [\zeta^2]]}e^{\Psi\bw[1 + [\zeta] + [\zeta^2]]}
\end{align*}
However, from (\ref{eq:zetau}) we see that
\begin{align*}
\bv(1 + [\zeta] + [\zeta^2]) =
\left( \begin{array}{l}
v_1 (1 + \zeta + \zeta^2 ) \\
v_2 (1 + \zeta + \zeta^2 ) \\
v_3 (1 + \zeta^2 + \zeta )
\end{array} \right) =
\left( \begin{array}{l}
0 \\ 0 \\ 0
\end{array} \right),
\end{align*}
and so LHS$(\bu,\bv,\bw)$ is Abelian with respect to $\bu$.  Further, since it is symmetric in $(\bu,\bv,\bw)$ we can conclude that it is Abelian in $\bv$ and $\bw$ as well.  Hence it may be expressed as
\[
\text{LHS}(\bu,\bv,\bw) = \sum_i c_i A_i(\bu)B_i(\bv)C_i(\bw)
\]
such that the $c_i$ are constants and the functions $A_i(\bu),B_i(\bv),C_i(\bw)$ belong to the basis for
$\Gamma \big( J, \mathcal{O}(3 \Theta^{[2]} ) \big)$, presented earlier in equation (\ref{eq:34_3pole}).  To determine the constants $c_i$ we use the $\sigma$-expansion.  The computations involved can be heavy and so it is essential that we take into account all the available simplifications.  We have already noted that $\text{LHS}(\bu,\bv,\bw)$ is symmetric under all permutations of $(\bu,\bv,\bw)$ and we reduce the number of independent $c_i$ by applying this property to the sum.  We can also check using the parity property of $\sigma(\bu)$ that $\text{LHS}(\bu,\bv,\bw)$ is even under $(\bu,\bv,\bw) \mapsto [-1](\bu,\bv,\bw)$.  Since we know the parity of the $\wp$-functions matches that of the number of their indices, we can check the parity of all basis functions and hence only include suitable combinations.  The biggest computational simplification come from the knowledge that $\text{LHS}(\bu,\bv,\bw)$ has total weight $-30$ which will drastically reduces the number of possible terms.  To further ease the time and memory constraints we implement code in Maple to efficiently expand the products of series so that only the relevant terms are considered.

We find that $\text{LHS}(\bu,\bv,\bw)$ is given as stated in the theorem.  For simplicity we group together the terms with common weight ratios into the polynomials $P_k(\bu,\bv,\bw)$ which contain the terms with weight $-k$ in the Abelian functions and weight $(-30-k)$ in the curve parameters.  These polynomials are presented in Appendix \ref{APP_34add} and are made available in the supplementary material. \\
\end{proof}

If we were to try and derive the corresponding addition formula for the $(2,7)$-case then we would be led to consider a curve automorphism $[\zeta]$ where the constant $\zeta=-1$ instead.  The automorphism addition formula would then coincide with the standard addition formula (\ref{eq:27_2t2v}).  Hence there is no corresponding addition formula to Theorem \ref{thm:34_3t3v} in the $(2,7)$-case, or rather it is the same as the tradition addition formula in equation (\ref{eq:27_2t2v}).

The final addition formula presented in this paper is satisfied by the functions associated to the reduction of the $(3,4)$-curve given by
$y^3 = x^4 + \lambda_0$.
This has a family of automorphisms,
\[
[i^j]: (x,y) \mapsto ((-i)^jx,y)
\]
where $i$ is the complex variable and $j \in \Z$.  The functions then satisfy the following formula.

\begin{theorem} \label{thm:34_4t2v}
The functions associated to the restricted $(3,4)$-curve, $y^3 = x^4 + \lambda_{0}$ satisfy
\[
\frac{\sigma(\bu+\bv)\sigma(\bu+[i]\bv)\sigma(\bu+[i^2]\bv)\sigma(\bu+[i^3]\bv) }{ \sigma(\bu)^4 \sigma(\bv)^4 }
= f(\bu,\bv) - f(\bv,\bu)
\]
where
\begin{align*}
f(\bu,\bv) &= \textstyle  - \frac{7}{9}\wp_{112}(\bv)\wp_{123}(\bu) - \frac{1}{6}\wp_{1111}(\bu)
+ \frac{1}{18}\wp_{112}(\bu)\partial_2\wp^{[11]}(\bv) - \frac{1}{6}F(\bu) \wp_{22}(\bv) \\
&\quad \textstyle + \frac{1}{27}\wp_{1113}(\bu)\wp_{3333}(\bv)
- \frac{1}{18}\partial_2 \wp^{[11]}(\bu)\partial_2 \wp^{[13]}(\bv)
- \frac{1}{216}\partial_2\partial_2 Q_{1333}(\bu)\wp_{2222}(\bv) \\
&\quad \textstyle + \frac{1}{6}\partial_2 \wp^{[33]}(\bu)\wp_{233}(\bv)
- \frac{2}{9}\wp_{1113}(\bu)\wp_{22}(\bv) + \frac{4}{9}\wp^{[22]}(\bu)Q_{1333}(\bv)
- \frac{2}{9}\wp_{1133}(\bu)Q_{1333}(\bv) \\
&\quad \textstyle + \frac{1}{54}\partial_2\partial_2Q_{1333}(\bu)\wp_{1333}(\bv)
- \frac{1}{18}\wp_{3333}(\bv)F(\bu) + \frac{4}{9}\wp^{[22]}(\bv)\wp_{1333}(\bu) \\
&\quad \textstyle - \frac{2}{9}\partial_2 \wp_{[13]}(\bu)\wp_{123}(\bv) - \frac{1}{36}\wp_{1133}(\bu)\wp_{2222}(\bv)
 +  \big[ \frac{1}{6}\wp_{2222}(\bv) + \frac{1}{18}\wp_{3333}(\bv)\wp_{22}(\bu) \big] \lambda_0.
\end{align*}
\end{theorem}
\begin{proof}
Using a similar approach to Lemma \ref{lem:sigzeta} we see that $\sigma([i]\bu)=-i\sigma(\bu)$.  We can then check that the left hand side is anti-symmetric under $(\bu,\bv) \mapsto (\bv,\bu)$ and is Abelian.  We can hence express it as a quadratic polynomial in the elements of the basis for
$\Gamma \big( J, \mathcal{O}(4 \Theta^{[2]} ) \big)$, derived earlier in Theorem (\ref{thm:34_4pole}).

We use the $\sigma$-function to determine the coefficients of each term.  Due to the restriction on the curve parameters, determining the coefficients is computationally easy in comparison to Theorem \ref{thm:34_3t3v}.  (However the construction of the basis in Theorem (\ref{thm:34_4pole}) required much effort.) \\
\end{proof}
There is a similar formula associated to the restricted $(2,7)$-curve $y^2 = x^7 + \lambda_0$ which has automorphisms
$[\iota^j]: (x,y) \mapsto (\iota^jx,y)$.  Here $\iota$ is a seventh root of unity and we consider,
\[
\frac{\prod_{k=0}^6 \sigma(\bu+[\iota^k]\bv)}{\sigma(\bu)^7\sigma(\bv)^7}.
\]
To evaluate this we will require a basis for the 7-pole vector space which has dimension $7^3=343$.

\newpage

\appendix

\section{Bases for the functions associated with the hyperelliptic curve of genus four} \label{APP_29}

In this Appendix we consider the general (2,9)-curve,
\begin{align*}
&y^2+(\mu_1x^4 + \mu_3x^3 + \mu_5x^2+\mu_7x+\mu_9)y \\
&\hspace*{1in} = x^9+\mu_2x^8+\mu_4x^7+\mu_6x^6+\mu_8x^5+\mu_{10}x^4
+\mu_{12}x^3 + \mu_{14}x^2 + \mu_{16}x + \mu_{18},
\end{align*}
which has genus four.  We construct the bases for standard Abelian functions associated with this curve, following the approach in Section \ref{SEC_Bases}.
\begin{theorem} \label{thm:29_2pole}
The basis for $\Gamma \big( J, \mathcal{O}(2 \Theta^{[3]} ) \big)$ is given by
\begin{align} \label{eq:29_2pole}
\begin{array}{ccccccccccccccccccccccccccc}
&\op&  \C1       &\op& \C\wp_{11} &\op& \C\wp_{12} &\op& \C\wp_{13}  &\op& \C\wp_{14} &\op& \C\wp_{22}  \\
&\op& \C\wp_{23} &\op& \C\wp_{24} &\op& \C\wp_{33} &\op& \C\wp_{34}  &\op& \C\wp_{44} &\op& \C \Delta_1 \\
&\op& \C\Delta_2 &\op& \C\Delta_3 &\op& \C\Delta_4 &\op& \C\Delta_5,
\end{array}
\end{align}
where
\begin{align*}
\Delta_1 & = \wp_{34} \wp_{23}-\wp_{34} \wp_{14}+\wp_{24}^2-\wp_{33} \wp_{24}
           +\wp_{44} \wp_{13}-\wp_{22} \wp_{44},\\
\Delta_2 & = \wp_{34} \wp_{13}+\wp_{24} \wp_{14}-\wp_{33} \wp_{14}-\wp_{12} \wp_{44},\\
\Delta_3 & = -\wp_{44} \wp_{11}+\wp_{14}^2-\wp_{23} \wp_{14}+\wp_{13} \wp_{24},\\
   \Delta_4 & = -2 \wp_{34} \wp_{11} +2 \wp_{13} \wp_{14}
            -2 \wp_{22} \wp_{14}+2 \wp_{12} \wp_{24},\\
\Delta_5 & = -\wp_{12} \wp_{23}+\wp_{22} \wp_{13}-\wp_{13}^2+\wp_{12} \wp_{14}
      -\wp_{11} \wp_{24}+\wp_{11} \wp_{33}.
\end{align*}
\end{theorem}
\begin{proof}
We follow the proof of Theorem \ref{thm:27_3pole}.  This time the dimension is $2^g=2^4=16$ and we have ten 2-index $\wp$-functions.  By testing arbitrary sums of quadratic terms in the 2-index $\wp$-functions we find that we can identify combinations that have poles of order only two.  We find five linearly independent combinations which can fill the missing basis entries.   \\
\end{proof}
Note that we could have alternatively used $Q$-functions as discussed in Section \ref{SEC_Bases_gen}.  However, The $\Delta$-functions are advantageous since it allows the theory to be completely realised in terms of 2 and 3-index $\wp$-functions.  As discussed in Section \ref{SEC_DE}, this appears to be a feature unique to the hyperelliptic cases.  Note that while the $\mathcal{T}, \mathcal{F}$ and $\mathcal{G}$-functions introduced in Section \ref{SEC_Bases} had reduced poles structures in general, these $\Delta$-function have poles of order two only in the (2,9)-case.

\begin{theorem} \label{thm:29_3pole}
The basis for $\Gamma \big( J, \mathcal{O}(3 \Theta^{[3]} ) \big)$ is given by
\begin{align} \label{eq:29_3pole}
\begin{array}{ccccccccccccccccccc}
(\ref{eq:29_2pole})&\op& \C\wp_{111} &\op& \C\wp_{112} &\op& \C\wp_{113} &\op& \C\wp_{114} &\op& \C\wp_{122} \\
                   &\op& \C\wp_{123} &\op& \C\wp_{124} &\op& \C\wp_{133} &\op& \C\wp_{134} &\op& \C\wp_{144} \\
                   &\op& \C\wp_{222} &\op& \C\wp_{223} &\op& \C\wp_{224} &\op& \C\wp_{233} &\op& \C\wp_{234} \\
                   &\op& \C\wp_{244} &\op& \C\wp_{333} &\op& \C\wp_{334} &\op& \C\wp_{344} &\op& \C\wp_{444} \\
                   &\op& \C\partial_1 \Delta_1 &\op& \C\partial_2 \Delta_1 &\op& \C\partial_3 \Delta_1
                   &\op& \C\partial_4 \Delta_1 &\op& \C\partial_3 \Delta_2 \\
                   &\op& \C\partial_4 \Delta_2 &\op& \C\partial_1 \Delta_3 &\op& \C\partial_2 \Delta_3
                   &\op& \C\partial_3 \Delta_3 &\op& \C\partial_4 \Delta_3 \\
                   &\op& \C\partial_1 \Delta_4 &\op& \C\partial_2 \Delta_4 &\op& \C\partial_3 \Delta_4
                   &\op& \C\partial_4 \Delta_4 &\op& \C\partial_1 \Delta_5 \\
                   &\op& \C\partial_2 \Delta_5 &\op& \C\partial_3 \Delta_5 &\op& \C\partial_4 \Delta_5
                   &\op& \C \T_{111333}        &\op& \C \T_{111334}        \\
                   &\op& \C \T_{111344}        &\op& \C \T_{112334}        &\op& \C \T_{122244}
                   &\op& \C \T_{122334} &\op& \C \T_{122344}               \\
                   &\op& \C \T_{122444} &\op& \C \T_{123333} &\op& \C \T_{123334}
                   &\op& \C \T_{123444} &\op& \C \T_{124444}               \\
                   &\op& \C \T_{133333} &\op& \C \T_{144444} &\op& \C \T_{222344}
                   &\op& \C \T_{222444} &\op& \C \T_{223444}               \\
                   &\op& \C \T_{224444} &\op& \C \T_{233333} &\op& \C \T_{233334}
                   &\op& \C \T_{233444} &\op& \C \T_{333334} \\
                   &\op& \C \T_{333444} &\op& \C \T_{334444} &\op& \C \T_{344444} &\op& \C U_1  &\op& \C U_2.
\end{array}
\end{align}
where
\begin{align*}
U_1 &= \textstyle 2\wp_{233}\wp_{2223} - \wp_{222}\wp_{2333} - \frac{3}{2}\wp_{223}\wp_{2233} - \wp_{33}\wp_{22223}
+ \wp_{22}\wp_{22333} + \frac{1}{2}\wp_{333}\wp_{2222}  \\
& \quad \textstyle + 6\wp_{22}\wp_{33}\wp_{223} + 6\wp_{23}\wp_{33}\wp_{222}
+ 3\wp_{23}^2\wp_{223} - 12\wp_{22}\wp_{23}\wp_{233} - 3\wp_{22}^2\wp_{333},
\\
U_2 &= \textstyle 2\wp_{234}\wp_{2223} - \wp_{222}\wp_{2334} - \frac{1}{2}\wp_{224}\wp_{2233} - \wp_{223}\wp_{2234}
- \wp_{34}\wp_{22223} + \wp_{22}\wp_{22334}  \\
&\quad \textstyle + \frac{1}{2}\wp_{334}\wp_{2222} + 6\wp_{22}\wp_{34}\wp_{223}
+ 6\wp_{23}\wp_{34}\wp_{222} + 3\wp_{23}^2\wp_{224} - 12\wp_{22}\wp_{23}\wp_{234} - 3\wp_{22}^2\wp_{334}.
\end{align*}
\end{theorem}
\begin{proof}
This time the dimension is $3^g=3^4=81$ and we find that we can identify 54 functions using the basis (\ref{eq:27_2pole}) and its derivatives.  We then proceed to add 36 of the $\mathcal{T}$-functions defined in equation (\ref{eq:Tgen}).
To find the final two functions we used a new class formed from a sum of $\wp$-functions in which each term has seven indices.
\\
\end{proof}
In the (2,9)-case $\sigma(\bu)$ has weight 10 and hence the minimal weight for functions in this basis is $30$,
in accordance with Lemma \ref{lem:minwt}.
This is achieved by the function  $\mathcal{T}_{111333}$, although this is not a unique choice.
For example, it could be replaced by $\mathcal{T}_{222222}$.
Both functions will reduce to a constant over $\sigma(\bu)^3$.

\section{The 4-index relations associated to genus three curves} \label{APP_4index}

The complete set of 4-index relations associated to the $(2,7)$-curve is given below.
\begin{align*}
\bm{(-4)} \quad
\wp_{3333} &= 4\wp_{23} +4\wp_{33}\lambda_{{6}}+2\lambda_{{5}}+6\wp_{33}^{2} \\
\bm{(-6)} \quad
\wp_{2333} &= 6\wp_{13} -2\wp_{22} +4\wp_{23} \lambda_{{6}}+6\wp_{23} \wp_{33} \\
\bm{(-8)} \quad
\wp_{2233} &= -2\wp_{12} +4\wp_{13} \lambda_{{6}}+2\wp_{23} \lambda_{{5}}
+2\wp_{22} \wp_{33} +4\wp_{23}^{2} \\
\bm{(-8)} \quad
\wp_{1333} &= -2\wp_{12} +4\wp_{13} \lambda_{{6}}+6\wp_{13} \wp_{33} \\
\bm{(-10)} \quad
\wp_{2223} &= -6\wp_{11} +4\wp_{13} \lambda_{{5}}
+4\wp_{23} \lambda_{{4}}-2\wp_{33} \lambda_{{3}}-4\lambda_{{2}}+6\wp_{22} \wp_{23} \\
\bm{(-10)} \quad
\wp_{1233} &= 2\wp_{13} \lambda_{{5}}+2\wp_{12} \wp_{33} +4\wp_{13}\wp_{23} \\
\bm{(-12)} \quad \wp_{1223} &=2\wp_{12} \wp_{23} +4\wp_{13} \wp_{22} +4\wp_{13} \lambda_{{4}}
- 2\lambda_{{1}}-2\ \wp_{13}^{2}+2\wp_{11} \wp_{33} \\
\bm{(-12)} \quad
\wp_{2222} &= 6\wp_{22}^{2}-6\lambda_{{1}}+12\wp_{12} \wp_{23} +12 \wp_{13}^{2}-12\wp_{13} \wp_{22}
- 12\wp_{11} \wp_{33} +2\lambda_{{5}}\lambda_{{3}}-8\lambda_{{6}}\lambda_{{2}} \\
&\quad \textstyle +4\wp_{22} \lambda_{{4}} - 12\wp_{33} \lambda_{{2}} + 4\wp_{12} \lambda_{{5}} + 4\wp_{23}\lambda_{{3}}
- 12\wp_{11} \lambda_{{6}} \\
\bm{(-12)} \quad
\wp_{1133} &= 2\wp_{12}\wp_{23} + 6\wp_{13}^2 - 2\wp_{13}\wp_{22} \\
\bm{(-14)} \quad
\wp_{1123} &= 2\wp_{11} \wp_{23} +4\wp_{12} \wp_{13} +2\wp_{13} \lambda_{{3}}-4\lambda_{{0}} \\
\bm{(-14)} \quad
\wp_{1222} &= 6\wp_{12} \wp_{22} -2\wp_{11} \lambda_{{5}}+4\wp_{13} \lambda_{{3}}-6\wp_{33} \lambda_{{1}}
+4\wp_{12} \lambda_{{4}}-8\lambda_{{0}}-4\lambda_{{6}}\lambda_{{1}}\\
\bm{(-16)} \quad
\wp_{1113} &= -2\wp_{23} \lambda_{{1}}+4\wp_{33} \lambda_{{0}}+4\wp_{13} \lambda_{{2}}+6\wp_{11} \wp_{13} \\
\bm{(-16)} \quad
\wp_{1122} &= -2\wp_{23} \lambda_{{1}}+2\wp_{12} \lambda_{{3}}-8\wp_{33} \lambda_{{0}}
+4\wp_{13} \lambda_{{2}}+2\wp_{11} \wp_{22} +4\wp_{12}^{2}-8\lambda_{{6}}\lambda_{{0}} \\
\bm{(-18)} \quad
\wp_{1112} &= 6\wp_{11} \wp_{12} -4\lambda_{{5}}\lambda_{{0}}-2\wp_{22}\lambda_{{1}}+6\wp_{13} \lambda_{{1}}
-8\wp_{23} \lambda_{{0}}+4\wp_{12} \lambda_{{2}} \\
\bm{(-20)} \quad
\wp_{1111} &= 6\wp_{11}^{2}+2\lambda_{{3}}\lambda_{{1}}-8\lambda_{{4}}\lambda_{{0}}+4\wp_{11}\lambda_{{2}}
+4\wp_{12} \lambda_{{1}}-12\wp_{22} \lambda_{{0}}+16\wp_{13}\lambda_{{0}}
\end{align*}

\noindent The complete set of 4-index relations associated to the $(3,4)$-curve is given below.
\begin{align*}
\bm{(-4)} \quad
\wp_{3333}  &= -3\wp_{22} +6  \wp_{33}^{2} \\
\bm{(-5)} \quad
\wp_{2333}  &= 3\wp_{33}\lambda_{{3}} +6\wp_{23} \wp_{33} \\
\bm{(-6)} \quad
\wp_{2233}  &= 4\wp_{13} +3\wp_{23}\lambda_{{3}}+2\lambda_{{2}}+2\wp_{22} \wp_{33} +4  \wp_{23}^{2} \\
\bm{(-7)} \quad
\wp_{2223}  &= 3\wp_{22}\lambda_{{3}}+6\wp_{22} \wp_{23} \\
\bm{(-8)} \quad
\wp_{2222}  &= -3\wp_{33} {\lambda_{{3}}}^{2}+12\wp_{33}\lambda_{{2}} -4Q_{{1333}}+6  \wp_{22}^{2}
\end{align*}
\begin{align*}
\bm{(-8)} \quad
\wp_{1333}  &= Q_{1333} + 6\wp_{13}\wp_{33} \\
\bm{(-9)} \quad
\wp_{1233}  &= 3\wp_{13}\lambda_{{3}}+\lambda_{{1}}+2\wp_{12} \wp_{33} +4\wp_{13} \wp_{23} \\
\bm{(-10)} \quad
\wp_{1223}  &= -2\wp_{11} +3\wp_{12}\lambda_{{3}}+4\wp_{12} \wp_{23} +2\wp_{13} \wp_{22} \\
\bm{(-11)} \quad
\wp_{1222}  &= -Q_{{1333}}\lambda_{{3}}+6\wp_{33}\lambda_{{1}}+6\wp_{12} \wp_{22} \\
\bm{(-12)} \quad
\wp_{1133}  &= 2\wp_{13}\lambda_{{2}}-\wp_{23}\lambda_{{1}}+2\wp_{11} \wp_{33}+4  \wp_{13}^{2} \\
\bm{(-13)} \quad
\wp_{1123}  &= 2\wp_{12}\lambda_{{2}}-\wp_{22}\lambda_{{1}}+2\wp_{11} \wp_{23}+4\wp_{12} \wp_{13} \\
\bm{(-14)} \quad
\wp_{1122}  &= \textstyle -\frac{2}{3}Q_{{1333}}\lambda_{{2}}+\wp_{33}\lambda_{{3}}\lambda_{{1}}+8\wp_{33}\lambda_{{0}}
+2\wp_{11} \wp_{22} +4  \wp_{12}^{2} \\
\bm{(-16)} \quad
\wp_{1113}  &= 3\wp_{12}\lambda_{{1}}-6\wp_{22}\lambda_{{0}}+6\wp_{11} \wp_{13} \\
\bm{(-17)} \quad
\wp_{1112}  &= -Q_{{1333}}\lambda_{{1}}+6\wp_{33}\lambda_{{3}}\lambda_{{0}}+6\wp_{11} \wp_{12} \\
\bm{(-20)} \quad
\wp_{1111}  &= -4Q_{{1333}}\lambda_{{0}}-3\wp_{33} {\lambda_{{1}}}^{2}
+12\wp_{33}\lambda_{{2}}\lambda_{{0}}+6  \wp_{11}^{2}
\end{align*}

\section{Quadratic 3-index relations associated with the cyclic hyperelliptic curve of genus three} \label{SEC_APP27}

This appendix contains the complete set of quadratic 3-index relations associated to the cyclic $(2,7)$-curve.
Note that $\Delta$ is a function quadratic in the 2-index $\wp$-functions, as defined in equation (\ref{eq:27Delta}).  It was used in the basis of fundamental Abelian functions (\ref{eq:27_2pole}), and occurs here only linearly, or multiplied by a single 2-index $\wp$-function.  Hence each relation can be rewritten as a polynomial in 2-index $\wp$-functions of degree three.  The relations are presented in 
weight order as indicated by the number in brackets.

\begin{align*}
&\bm{(-6)} &  \wp_{333}^2 &= \textstyle 4 \wp_{33}^{3}
+4\lambda_{{4}}
+4\lambda_{{5}}\wp_{33}
+4\lambda_{{6}} \wp_{33}^{2}
-4\wp_{13}
+4\wp_{22}
+4\wp_{33} \wp_{23}
\\
&\bm{(-8)} &  \wp_{233}\wp_{333} &= \textstyle 4\wp_{23}\wp_{33}^{2}
+2\lambda_{{3}}
+2\lambda_{{5}}\wp_{23}
+4\lambda_{{6}}\wp_{33}\wp_{23}
+2\wp_{12}
+4\wp_{33}\wp_{13} \\
& & & \quad \textstyle -2\wp_{33}\wp_{22}
+2 \wp_{23}^{2}
\\
&\bm{(-10)} & \wp_{233}^{2} &= \textstyle 8\wp_{23}\wp_{13}
-4\wp_{23}\wp_{22}
+4\wp_{11}
+4\lambda_{{2}}
+4\wp_{23}^{2}\wp_{33}
+4\lambda_{{6}}\wp_{23}^{2}
\\
&\bm{(-10)} & \wp_{223}\wp_{333} &= \textstyle -4\wp_{33}\wp_{12}
-2\wp_{23}\wp_{13} + 4\wp_{23}\wp_{22} + 4\lambda_{{6}}\wp_{33}\wp_{13}
+2\lambda_{{5}}\wp_{33}\wp_{23}
+4\lambda_{{4}}\wp_{23} \\
& & & \quad \textstyle -2\wp_{11} +2\lambda_{{5}}\wp_{13}
-2\lambda_{{3}}\wp_{33}+2\wp_{22}\wp_{33}^{2}
+2\wp_{23}^{2}\wp_{33}
\\
&\bm{(-10)} & \wp_{133}\wp_{333} &= \textstyle -2\wp_{33}\wp_{12}
+2\wp_{23}\wp_{13}
+4\lambda_{{6}}\wp_{33}\wp_{13}
-2\wp_{11}+2\lambda_{{5}}\wp_{13}
+4\wp_{13}\wp_{33}^{2}
\\
&\bm{(-12)} & \wp_{223} \wp_{233}  &= \textstyle2\lambda_{{1}}
+4\lambda_{{6}}\wp_{23} \wp_{13}
+2\wp_{22} \wp_{23} \wp_{33}
+2  \wp_{23}^{3}
-2\wp_{13} \wp_{22}
+4  \wp_{13}^{2}
-4\wp_{11} \wp_{33} \\
& & & \quad \textstyle +2\lambda_{{5}}  \wp_{23}^{2}
-4\lambda_{{2}}\wp_{33} +2\lambda_{{3}}\wp_{23}
\\
&\bm{(-12)} & \wp_{222} \wp_{333}  &= \textstyle-16  \wp_{13}^{2}
-4  \wp_{22}^{2}
+20\wp_{13} \wp_{22} -2\wp_{12} \wp_{23}
+6\wp_{22} \wp_{23} \wp_{33}
+4\wp_{11} \wp_{33}
-2  \wp_{23}^{3} \\
& & & \quad \textstyle -2\lambda_{{5}}\lambda_{{3}}
-8\lambda_{{6}}\wp_{23} \wp_{13}
-2\lambda_{{5}}^{2}\wp_{23}
-2\lambda_{{5}}\wp_{12}
-4\lambda_{{4}}\wp_{22}
+16\lambda_{{4}}\wp_{13}
-4\lambda_{{3}}  \wp_{33}^{2} \\
& & & \quad \textstyle +8\lambda_{{4}}\wp_{33} \wp_{23}
-2\lambda_{{5}}\wp_{33} \wp_{22}
-4\lambda_{{6}}\wp_{33} \wp_{12}
+8\lambda_{{5}}\wp_{33} \wp_{13}
+8\lambda_{{6}}\wp_{23} \wp_{22} \\
& & & \quad \textstyle +8\lambda_{{6}}\lambda_{{4}}\wp_{23}
-4\lambda_{{6}}\lambda_{{3}}\wp_{33}
-4\lambda_{{5}}  \wp_{23}^{2}
-2\lambda_{{3}}\wp_{23}
\\
&\bm{(-12)} & \wp_{133} \wp_{233}  &= \textstyle 2\lambda_{{1}}
+4\lambda_{{6}}\wp_{23} \wp_{13}
+4\wp_{13} \wp_{23} \wp_{33}
-2\wp_{13} \wp_{22}
-2\wp_{12} \wp_{23}
+4  \wp_{13}^{2}
\\
&\bm{(-12)} & \wp_{123} \wp_{333}  &= \textstyle 2\wp_{11} \wp_{33}
-4\wp_{13}^{2}
+4\wp_{13}\wp_{22}
+4\lambda_{{4}}\wp_{13}
+2\wp_{12}\wp_{33}^{2} \\
& & & \quad \textstyle +2\wp_{13}\wp_{23}\wp_{33}
+2\lambda_{{5}}\wp_{33}\wp_{13}
\\
&\bm{(-14)} &  \wp_{223}^{2} &= \textstyle 4\wp_{33}\Delta
-4\wp_{23} \wp_{11}
-4\wp_{13} \wp_{12} +4\lambda_{{0}}
-4\lambda_{{3}}\wp_{33} \wp_{23}
+4\lambda_{{6}}\wp_{13}^{2} \\
& & & \quad \textstyle +4\lambda_{{5}}\wp_{23} \wp_{13}
+4\lambda_{{2}}\wp_{33}^{2}
+4\lambda_{{4}}\wp_{23}^{2}
-4\lambda_{{1}}\wp_{33}
+4\wp_{22}\wp_{23}^{2}
\\
&\bm{(-14)} & \wp_{133} \wp_{223}  &= \textstyle 4\lambda_{{0}}
-2\wp_{23} \wp_{11}
-4\wp_{13} \wp_{12}
+4\lambda_{{6}}\wp_{13}^{2}
+2\lambda_{{5}}\wp_{23} \wp_{13}
-2\lambda_{{1}}\wp_{33}\\
& & & \quad \textstyle +2\wp_{13}\wp_{23}^{2}
+2\wp_{13}\wp_{33}\wp_{22}
\\
\end{align*}
\begin{align*}
&\bm{(-14)} &  \wp_{133}^{2} &= \textstyle 4\wp_{13}^{2}\wp_{33}
-4\wp_{13} \wp_{12}
+4\lambda_{{0}}+4\lambda_{{6}}  \wp_{13}^{2}
\\
&\bm{(-14)} & \wp_{222} \wp_{233}  &= \textstyle-8\wp_{33} \Delta
-4\wp_{23} \wp_{11} +4\lambda_{{6}}\lambda_{{3}}\wp_{23} -4\lambda_{{6}}\Delta
-4\lambda_{{6}}\wp_{13}^{2}
-8\lambda_{{6}}\lambda_{{2}}\wp_{33} \\
& & & \quad \textstyle +2\lambda_{{5}}\wp_{23} \wp_{22}
+8\wp_{13} \wp_{12}
-2\wp_{22} \wp_{12}
-4\lambda_{{2}}\lambda_{{5}}
+4\lambda_{{3}}\wp_{33} \wp_{23}
+2\wp_{22}^{2}\wp_{33} \\
& & & \quad \textstyle +2\wp_{22}\wp_{23}^{2}
-8\lambda_{{2}}  \wp_{33}^{2}
+4\lambda_{{6}}\wp_{13} \wp_{22}
-4\lambda_{{6}}\wp_{11} \wp_{33}
-4\lambda_{{1}}\wp_{33}
-4\lambda_{{5}}\wp_{11} \\
& & & \quad \textstyle +8\lambda_{{3}}\wp_{13}
-2\lambda_{{3}}\wp_{22}
-4\lambda_{{2}}\wp_{23}
\\
&\bm{(-14)} & \wp_{123} \wp_{233}  &= \textstyle2\wp_{13} \wp_{12}
+2\lambda_{{5}}\wp_{23} \wp_{13}
-2\lambda_{{1}}\wp_{33}
+2\wp_{13}   \wp_{23}^{2}
+2\wp_{12} \wp_{33} \wp_{23}
+2\lambda_{{3}}\wp_{13}
\\
&\bm{(-14)} & \wp_{122} \wp_{333}  &= \textstyle4\wp_{33} \Delta
-2\lambda_{{5}}\wp_{33} \wp_{12}
-8\lambda_{{6}}  \wp_{13}^{2}
+8\lambda_{{6}}\lambda_{{4}}\wp_{13}
+8\lambda_{{4}}\wp_{33} \wp_{13}
+6\wp_{13} \wp_{12} \\
& & & \quad \textstyle -4\lambda_{{5}}\wp_{23} \wp_{13}
-2{\lambda_{{5}}}^{2}\wp_{13}
-4\lambda_{{4}}\wp_{12}
+2\lambda_{{5}}\wp_{11}
+2\lambda_{{3}}\wp_{13}
+8\lambda_{{6}}\wp_{13} \wp_{22} \\
& & & \quad \textstyle +4\lambda_{{6}}\wp_{11} \wp_{33}
-2\wp_{13}  \wp_{23}^{2}
+2\wp_{13} \wp_{33} \wp_{22}
+4\wp_{12} \wp_{33} \wp_{23}
-4\wp_{22} \wp_{12} +4\wp_{23} \wp_{11}
\\
&\bm{(-14)} & \wp_{113} \wp_{333}  &= \textstyle-2\wp_{33} \Delta
+2\wp_{23} \wp_{11}
+2\wp_{13} \wp_{12}
+2  \wp_{13}^{2}\wp_{33}
+2\wp_{11}  \wp_{33}^{2}
+2\lambda_{{3}}\wp_{13}
\\
&\bm{(-16)} & \wp_{222} \wp_{223}  &= \textstyle
-4\lambda_{{6}}\wp_{23} \wp_{11}
-4\wp_{23}\Delta
+4\lambda_{{5}}\wp_{12} \wp_{23}
-4\lambda_{{5}}\wp_{11} \wp_{33}
+4\lambda_{{4}}\wp_{23} \wp_{22} \\
& & & \quad \textstyle -2\lambda_{{3}}\wp_{33} \wp_{22}
-4\lambda_{{6}}\lambda_{{1}}\wp_{33}
+2\lambda_{{5}}\lambda_{{3}}\wp_{23}
-4\lambda_{{2}}\wp_{33} \wp_{23}
-4\lambda_{{2}}\lambda_{{5}}\wp_{33}
-8\lambda_{{4}}\wp_{11} \\
& & & \quad \textstyle +2\lambda_{{3}}\wp_{23}^{2}
+4\lambda_{{3}}\wp_{12}
+8\lambda_{{2}}\wp_{13}
-8\lambda_{{2}}\wp_{22}
-2\lambda_{{1}}\wp_{23}
-8\lambda_{{0}}\wp_{33} \\
& & & \quad \textstyle +4\wp_{22}^{2}\wp_{23}
+4\lambda_{{5}}  \wp_{13}^{2}
-6\wp_{22} \wp_{11}
+2\wp_{12}^{2}
+2{\lambda_{{3}}}^{2}
-8\lambda_{{2}}\lambda_{{4}}
-2\lambda_{{5}}\lambda_{{1}}
\\
&\bm{(-16)} & \wp_{133} \wp_{222}  &= \textstyle
-4\lambda_{{6}}\wp_{23} \wp_{11}
-4\wp_{23}\Delta
+2\lambda_{{5}}\wp_{12} \wp_{23}
-4\lambda_{{6}}\lambda_{{1}}\wp_{33}
-2\wp_{12}   \wp_{23}^{2} \\
& & & \quad \textstyle -2\lambda_{{1}}\wp_{23} -8\lambda_{{0}}\wp_{33}
+4\lambda_{{5}}  \wp_{13}^{2}
-8\wp_{13} \wp_{11}
+2\wp_{22} \wp_{11} \\
& & & \quad \textstyle -2\lambda_{{5}}\lambda_{{1}}
+4\wp_{13} \wp_{23} \wp_{22}
+2\wp_{12} \wp_{33} \wp_{22}
-4\lambda_{{1}}  \wp_{33}^{2}
\\
&\bm{(-16)} & \wp_{123} \wp_{223}  &= \textstyle2\wp_{23} \Delta
+2\wp_{12}\wp_{23}^{2}
-4\lambda_{{0}}\wp_{33} -2\wp_{13} \wp_{11}
+2\lambda_{{5}}  \wp_{13}^{2}
-2\lambda_{{3}}\wp_{33} \wp_{13} \\
& & & \quad \textstyle +4\lambda_{{4}}\wp_{23} \wp_{13}
+2\wp_{13} \wp_{23} \wp_{22}
+2\lambda_{{1}}  \wp_{33}^{2}
\\
&\bm{(-16)} & \wp_{123} \wp_{133}  &= \textstyle
-4\lambda_{{0}}\wp_{33}
-2\wp_{13} \wp_{11}
+2  \wp_{13}^{2}\wp_{23}
+2\wp_{12} \wp_{13} \wp_{33}
+2\lambda_{{5}}  \wp_{13}^{2}
\\
&\bm{(-16)} & \wp_{122} \wp_{233}  &= \textstyle 2\lambda_{{5}}\wp_{12} \wp_{23}
-2\lambda_{{5}}\wp_{11} \wp_{33}
-4\lambda_{{6}}\lambda_{{1}}\wp_{33}
+2\wp_{12}\wp_{23}^{2}
-2\lambda_{{3}}\wp_{12}
+4\lambda_{{2}}\wp_{13} \\
& & & \quad \textstyle -4\lambda_{{1}}\wp_{23}
-2\lambda_{{5}}  \wp_{13}^{2}
+4\wp_{13} \wp_{11}
-2\wp_{12}^{2}
-2\lambda_{{5}}\lambda_{{1}}
+4\lambda_{{6}}\wp_{13} \wp_{12} \\
& & & \quad \textstyle +4\lambda_{{6}}\lambda_{{3}}\wp_{13}
+4\lambda_{{3}}\wp_{33} \wp_{13}
+2\wp_{12} \wp_{33} \wp_{22}
+2\lambda_{{5}}\Delta
-4\lambda_{{1}}\wp_{33}^{2}
\\
&\bm{(-16)} & \wp_{113} \wp_{233}  &= \textstyle
-2\wp_{23} \Delta
+4\lambda_{{2}}\wp_{13}
-2\lambda_{{1}}\wp_{23}
+4\wp_{13} \wp_{11}
+2  \wp_{13}^{2}\wp_{23}
+2\wp_{11} \wp_{23} \wp_{33}
\\
&\bm{(-16)} & \wp_{112} \wp_{333}  &= \textstyle4\lambda_{{6}}\wp_{23} \wp_{11}
+2\wp_{23}\Delta
-2\lambda_{{5}}\wp_{11} \wp_{33}
-2\lambda_{{3}}\wp_{12}
-2\wp_{13}^{2}\wp_{23}
-2\lambda_{{5}}  \wp_{13}^{2} \\
& & & \quad \textstyle +4\wp_{13} \wp_{11}
-2\wp_{22} \wp_{11}
-2  \wp_{12}^{2}
+4\lambda_{{6}}\wp_{13} \wp_{12}
+4\lambda_{{6}}\lambda_{{3}}\wp_{13}\\
& & & \quad \textstyle +4\lambda_{{3}}\wp_{33} \wp_{13}
+4\wp_{12} \wp_{13} \wp_{33}
+2\wp_{11} \wp_{23} \wp_{33}
+2\lambda_{{5}}\Delta
\\
&\bm{(-18)} &    \wp_{222}^{2} &= \textstyle
-8\lambda_{{3}}\wp_{23} \wp_{13}
+8\wp_{23}\lambda_{{5}}\lambda_{{2}}
-8\lambda_{{5}}\lambda_{{3}}\wp_{13}
+4\wp_{23} \wp_{22} \lambda_{{3}}
-16\lambda_{{6}}\lambda_{{4}}\wp_{11} \\
& & & \quad \textstyle -16\lambda_{{6}}\lambda_{{2}}\wp_{22}
+4\lambda_{{5}}\lambda_{{3}}\wp_{22}
-16\lambda_{{2}}\lambda_{{4}}\wp_{33}
-8\wp_{33} \lambda_{{1}}\lambda_{{5}}
+8\lambda_{{3}}\wp_{12} \lambda_{{6}} \\
& & & \quad \textstyle +16\lambda_{{0}}  \wp_{33}^{2}
-4\wp_{12}^{2}\wp_{33}
-4\wp_{11}   \wp_{23}^{2}
+8\wp_{12} \wp_{13} \wp_{23}
-8\lambda_{{1}}\wp_{33} \wp_{23} \\
& & & \quad \textstyle +16\lambda_{{6}}\wp_{13} \wp_{11}
+16\lambda_{{6}}\lambda_{{2}}\wp_{13}
+4\wp_{11} \wp_{22} \wp_{33}
-4  \wp_{13}^{2}\wp_{22}
+16\lambda_{{1}}\wp_{13} \\
& & & \quad \textstyle -16\lambda_{{1}}\lambda_{{4}}
-16\lambda_{{6}}\lambda_{{4}}\lambda_{{2}}
-16\wp_{22} \wp_{11} \lambda_{{6}}
-16\wp_{22} \lambda_{{2}}\wp_{33}
+4\lambda_{{5}}\wp_{22} \wp_{12} \\
& & & \quad \textstyle +4\lambda_{{6}}{\lambda_{{3}}}^{2}
+4{\lambda_{{5}}}^{2}\lambda_{{2}}
+4  \wp_{23}^{2}\lambda_{{2}}
+4  \wp_{22}^{2}\lambda_{{4}}
+4  \wp_{12}^{2}\lambda_{{6}}
+4{\lambda_{{5}}}^{2}\wp_{11} \\
& & & \quad \textstyle +4{\lambda_{{3}}}^{2}\wp_{33}
-16\lambda_{{4}}\Delta-8\lambda_{{3}}\wp_{11}
+32\wp_{13} \Delta-20\wp_{22} \Delta
-8\lambda_{{5}}\wp_{13} \wp_{12}
+4  \wp_{22}^{3} \\
& & & \quad \textstyle -16\lambda_{{1}}\wp_{22}
+16\lambda_{{2}}\wp_{33} \wp_{13} -8\wp_{12} \wp_{11}
+16\lambda_{{4}}\wp_{13}^{2}
\\
&\bm{(-18)} &   \wp_{123} \wp_{222}  &= \textstyle
-2\wp_{33} \lambda_{{1}}\lambda_{{5}}
+4\wp_{13} \wp_{22} \lambda_{{4}}
+4\lambda_{{4}}\wp_{11} \wp_{33}
-2\wp_{23} \lambda_{{5}}\wp_{11}
-4\lambda_{{1}}\wp_{22} \\
& & & \quad \textstyle
+8\wp_{13} \Delta
-2\wp_{22}\Delta-2\wp_{12} \wp_{11}
+4\wp_{12} \wp_{13} \wp_{23}
-2\lambda_{{1}}\wp_{33} \wp_{23}
+2\wp_{11} \wp_{22} \wp_{33}\\
& & & \quad \textstyle +4\lambda_{{4}}  \wp_{13}^{2}
+8\lambda_{{0}}  \wp_{33}^{2}
-2  \wp_{12}^{2}\wp_{33}
-2\wp_{11}   \wp_{23}^{2}
-2\lambda_{{3}}\wp_{33} \wp_{12}
+2\wp_{13}   \wp_{22}^{2} \\
& & & \quad \textstyle -2\lambda_{{3}}\wp_{11}
-4\lambda_{{4}}\Delta
-2  \wp_{13}^{2}\wp_{22}
+4\lambda_{{1}}\wp_{13}
-4\lambda_{{1}}\lambda_{{4}}
+2\wp_{12} \wp_{23} \wp_{22}
\\
&\bm{(-18)} &    \wp_{123}^{2} &= \textstyle 4\wp_{13}\Delta
+4\wp_{12} \wp_{13} \wp_{23}
+4\lambda_{{4}}  \wp_{13}^{2}
+4\lambda_{{0}}  \wp_{33}^{2}
\\
\end{align*}
\begin{align*}
&\bm{(-18)} &   \wp_{122} \wp_{223}  &= \textstyle
-4\lambda_{{1}}\lambda_{{4}}
-2\wp_{33}\lambda_{{1}}\lambda_{{5}}
+4\lambda_{{5}}\wp_{13} \wp_{12}
+4\lambda_{{3}}\wp_{23} \wp_{13}
+4\wp_{13} \wp_{22} \lambda_{{4}}
-4\lambda_{{1}}\wp_{22} \\
& & & \quad \textstyle
-8\lambda_{{6}}\lambda_{{0}}\wp_{33}
-4\lambda_{{6}}\wp_{13} \wp_{11}
-4\lambda_{{5}}\lambda_{{0}}
-12\wp_{13} \Delta
+4\wp_{22}\Delta
-4\wp_{12} \wp_{13} \wp_{23} \\
& & & \quad \textstyle +2\wp_{13}^{2}\wp_{22}
+6\lambda_{{1}}\wp_{13}
-8\lambda_{{0}}\wp_{23}
-8\lambda_{{4}}  \wp_{13}^{2}
-8\lambda_{{0}}  \wp_{33}^{2}
+2  \wp_{12}^{2}\wp_{33}
+2\wp_{11}   \wp_{23}^{2} \\
& & & \quad \textstyle -4\lambda_{{2}}\wp_{33} \wp_{13}
-2\wp_{11} \wp_{22} \wp_{33}
-2\lambda_{{3}}\wp_{11}
+2\lambda_{{5}}\lambda_{{3}}\wp_{13}
+4\wp_{12}\wp_{23} \wp_{22}
\\
&\bm{(-18)} &   \wp_{122} \wp_{133}  &= \textstyle
-4\wp_{13} \Delta
+2\wp_{12} \wp_{11}
-4\lambda_{{5}}\lambda_{{0}}
+2  \wp_{12}^{2}\wp_{33}
+2  \wp_{13}^{2}\wp_{22}\\
& & & \quad \textstyle -8\lambda_{{6}}\lambda_{{0}}\wp_{33}
+2\lambda_{{5}}\wp_{13} \wp_{12}
-4\lambda_{{6}}\wp_{13} \wp_{11}
+2\lambda_{{1}}\wp_{13}
-8\lambda_{{0}}\wp_{23}
-8\lambda_{{0}}  \wp_{33}^{2}
\\
&\bm{(-18)} &   \wp_{113} \wp_{223}  &= \textstyle
-4\wp_{13} \Delta
+2\wp_{11}\wp_{23}^{2}
+2\lambda_{{3}}\wp_{23} \wp_{13}
+2\lambda_{{1}}\wp_{33} \wp_{23}
-4\lambda_{{2}}\wp_{33} \wp_{13}
+2  \wp_{13}^{2}\wp_{22} \\
& & & \quad \textstyle  +2\lambda_{{1}}\wp_{13} -4\lambda_{{0}}\wp_{23}
\\
&\bm{(-18)} &   \wp_{113} \wp_{133}  &= \textstyle -2\wp_{13} \Delta
+2\wp_{11} \wp_{13}\wp_{33}
+2  \wp_{13}^{3}
+2\lambda_{{1}}\wp_{13}
-4\lambda_{{0}}\wp_{23}
\\
&\bm{(-18)} &   \wp_{112} \wp_{233}  &= \textstyle 8\wp_{13} \Delta
-2\wp_{22} \Delta
-4\wp_{12} \wp_{11}
+4\wp_{12} \wp_{13} \wp_{23}
-4\lambda_{{1}}\wp_{33} \wp_{23}
+8\lambda_{{2}}\wp_{33} \wp_{13} \\
& & & \quad \textstyle +8\lambda_{{6}}\wp_{13} \wp_{11}
+8\lambda_{{6}}\lambda_{{2}}\wp_{13}
+2\wp_{11} \wp_{22} \wp_{33}
-4\lambda_{{6}}\lambda_{{1}}\wp_{23}
-4\lambda_{{2}}\wp_{12} \\
& & & \quad \textstyle -2  \wp_{13}^{2}\wp_{22}
-4\lambda_{{1}}\wp_{13} +2\lambda_{{1}}\wp_{22}
\\
&\bm{(-18)} &   \wp_{111} \wp_{333}  &= \textstyle 2\lambda_{{5}}\wp_{13} \wp_{12}
+2\lambda_{{3}}\wp_{23} \wp_{13}
+6\wp_{11} \wp_{13} \wp_{33}
-4\lambda_{{4}}\wp_{11} \wp_{33}
+2\wp_{23} \lambda_{{5}}\wp_{11} \\
& & & \quad \textstyle -4\wp_{12} \wp_{13} \wp_{23}
-2\wp_{11} \wp_{22} \wp_{33}
+2  \wp_{13}^{2}\wp_{22}
-4\lambda_{{4}}  \wp_{13}^{2}
+2  \wp_{12}^{2}\wp_{33}
+2\wp_{11}  \wp_{23}^{2} \\
& & & \quad \textstyle +2\lambda_{{5}}\lambda_{{3}}\wp_{13}
+4\lambda_{{4}}\Delta -10\wp_{13} \Delta+4\wp_{22}\Delta
-2\wp_{13}^{3} +2\lambda_{{3}}\wp_{33} \wp_{12}
\\
&\bm{(-20)} &   \wp_{122} \wp_{222}  &= \textstyle
-8\lambda_{{6}}\lambda_{{4}}\lambda_{{1}}
-4\lambda_{{2}}\wp_{33} \wp_{12}
-6\lambda_{{1}}\wp_{33} \wp_{22}
+4\lambda_{{3}}\wp_{13} \wp_{22}
-16\lambda_{{0}}\wp_{22}
+2\lambda_{{1}}  \wp_{23}^{2} \\
& & & \quad \textstyle +4\lambda_{{5}}\lambda_{{1}}\wp_{23}
+8\lambda_{{6}}\lambda_{{1}}\wp_{13}
-8\lambda_{{5}}\lambda_{{0}}\wp_{33}
+4\lambda_{{4}}\wp_{22} \wp_{12}
-8\lambda_{{1}}\lambda_{{4}}\wp_{33} \\
& & & \quad \textstyle -4\lambda_{{6}}\wp_{12} \wp_{11}
-8\lambda_{{6}}\lambda_{{1}}\wp_{22}
-4\lambda_{{6}}\lambda_{{3}}\wp_{11}
+4\wp_{12}\wp_{22}^{2}
+2\lambda_{{5}}  \wp_{12}^{2}
-4\wp_{12} \Delta \\
& & & \quad \textstyle -16\lambda_{{4}}\lambda_{{0}}
-2\lambda_{{1}}\wp_{12}
-4\lambda_{{3}}\Delta-4\lambda_{{2}}\wp_{11}
+2\lambda_{{5}}\lambda_{{3}}\wp_{12} +16\lambda_{{0}}\wp_{13}
+2{\lambda_{{5}}}^{2}\lambda_{{1}} \\
& & & \quad \textstyle -2\lambda_{{5}}\wp_{22} \wp_{11}
-2\lambda_{{3}}\lambda_{{1}}
\\
&\bm{(-20)} &   \wp_{122} \wp_{123}  &= \textstyle 2\wp_{12}^{2}\wp_{23}
-8\lambda_{{4}}\lambda_{{0}}
+2\wp_{12} \Delta
+2\wp_{11}^{2}
+2\wp_{12} \wp_{22} \wp_{13}
+4\lambda_{{4}}\wp_{13} \wp_{12}\\
& & & \quad \textstyle -2\lambda_{{5}}\wp_{13} \wp_{11}
-2\lambda_{{1}}\wp_{33} \wp_{13}
-4\lambda_{{5}}\lambda_{{0}}\wp_{33}
-8\lambda_{{0}}\wp_{22}
+2\lambda_{{3}}  \wp_{13}^{2}
+8\lambda_{{0}}\wp_{13}
\\
&\bm{(-20)} &   \wp_{113} \wp_{222}  &= \textstyle 2\lambda_{{3}}\wp_{11}\wp_{33}
+8\lambda_{{0}}\wp_{33} \wp_{23}
-2\wp_{12}^{2}\wp_{23}
-2\lambda_{{3}}\lambda_{{1}}
-4\wp_{12} \Delta
-4\wp_{11}^{2}\\
& & & \quad \textstyle -8\lambda_{{1}}\wp_{33} \wp_{13}
-4\lambda_{{2}}\wp_{33} \wp_{12}
+2\lambda_{{1}}\wp_{33} \wp_{22}
+2\lambda_{{3}}\wp_{13} \wp_{22}
+2\lambda_{{3}}  \wp_{13}^{2}
-2\lambda_{{1}}\wp_{12} \\
& & & \quad \textstyle +4\wp_{12} \wp_{22} \wp_{13}
+2\wp_{11} \wp_{22} \wp_{23}
-2\lambda_{{3}}\Delta
-4\lambda_{{2}}\wp_{11}
\\
&\bm{(-20)} &   \wp_{113} \wp_{123}  &= \textstyle 4\lambda_{{0}}\wp_{33}\wp_{23}
+2  \wp_{13}^{2}\wp_{12}
+2\wp_{11} \wp_{13} \wp_{23}
-2\lambda_{{1}}\wp_{33} \wp_{13}
+2\lambda_{{3}}  \wp_{13}^{2}
\\
&\bm{(-20)} &   \wp_{112} \wp_{223}  &= \textstyle
-8\lambda_{{0}}\wp_{33} \wp_{23}
+2\wp_{12}^{2}\wp_{23}
-2\lambda_{{3}}\lambda_{{1}}
-4\wp_{11}^{2}
+2\wp_{11}\wp_{22} \wp_{23}
+4\lambda_{{5}}\wp_{13} \wp_{11}\\
& & & \quad \textstyle
+4\lambda_{{1}}\wp_{33} \wp_{13}
+4\lambda_{{2}}\lambda_{{5}}\wp_{13}
-8\lambda_{{6}}\lambda_{{0}}\wp_{23}
+4\lambda_{{2}}\wp_{23} \wp_{13}
+2\lambda_{{3}}\wp_{13} \wp_{22}
-4\lambda_{{2}}\wp_{11}\\
& & & \quad \textstyle +4\lambda_{{6}}\lambda_{{1}}\wp_{13}
+4\lambda_{{0}}\wp_{22}
-2\lambda_{{1}}  \wp_{23}^{2}
-4\lambda_{{3}}  \wp_{13}^{2}
-4\lambda_{{1}}\wp_{12}
-8\lambda_{{0}}\wp_{13}
-2\lambda_{{5}}\lambda_{{1}}\wp_{23}
\\
&\bm{(-20)} &   \wp_{112} \wp_{133}  &= \textstyle
-8\lambda_{{0}}\wp_{33} \wp_{23}
+2\wp_{13}^{2}\wp_{12}
-2\wp_{12}\Delta
+2\wp_{11} \wp_{33} \wp_{12}
+4\lambda_{{1}}\wp_{33} \wp_{13}\\
& & & \quad \textstyle -8\lambda_{{6}}\lambda_{{0}}\wp_{23}
+4\lambda_{{6}}\lambda_{{1}}\wp_{13}
+4\lambda_{{0}}\wp_{22}
-2\lambda_{{1}}\wp_{12}
-8\lambda_{{0}}\wp_{13}
\\
&\bm{(-20)} & \wp_{111} \wp_{233}  &= \textstyle -2\lambda_{{3}}\wp_{11}\wp_{33}
-2\wp_{13}^{2}\wp_{12}
+2\wp_{12}\Delta
+4\wp_{11} \wp_{13} \wp_{23}
+2\wp_{11} \wp_{33} \wp_{12} \\
& & & \quad \textstyle +4\lambda_{{2}}\lambda_{{5}}\wp_{13}
+4\lambda_{{2}}\wp_{23} \wp_{13}
+4\lambda_{{2}}\wp_{33} \wp_{12}
-2\lambda_{{1}}\wp_{33} \wp_{22}
-2\lambda_{{5}}\lambda_{{1}}\wp_{23} \\
& & & \quad \textstyle +4\lambda_{{5}}\wp_{13} \wp_{11}
+4\lambda_{{1}}\wp_{33} \wp_{13}
-2\lambda_{{1}}\wp_{23}^{2}
-2\lambda_{{3}}\wp_{13}^{2}
+2\lambda_{{3}}\Delta
\\
&\bm{(-22)} &   \wp_{122}^{2} &= \textstyle -16\lambda_{{6}}\lambda_{{4}}\lambda_{{0}}
+4\wp_{12}^{2}\wp_{22}
+4\wp_{11} \Delta
+4{\lambda_{{5}}}^{2}\lambda_{{0}}
-8\lambda_{{3}}\lambda_{{0}}
-4\lambda_{{5}}\wp_{12} \wp_{11} \\
& & & \quad \textstyle -4\lambda_{{1}}\wp_{11}
+4\lambda_{{2}}\wp_{13}^{2}
-16\lambda_{{4}}\lambda_{{0}}\wp_{33}
-4\lambda_{{1}}\wp_{23} \wp_{13}
+8\lambda_{{5}}\lambda_{{0}}\wp_{23}\\
& & & \quad \textstyle -8\lambda_{{0}}\wp_{33} \wp_{22}
-8\lambda_{{0}}\wp_{12}
+4\lambda_{{3}}\wp_{13} \wp_{12}
-4\lambda_{{1}}\wp_{33} \wp_{12}
+4\lambda_{{6}}\wp_{11}^{2} \\
& & & \quad \textstyle +8\lambda_{{0}}\wp_{23}^{2}
+16\lambda_{{6}}\lambda_{{0}}\wp_{13}
+4\lambda_{{4}}\wp_{12}^{2}
-16\lambda_{{6}}\lambda_{{0}}\wp_{22}
\\
&\bm{(-22)} &  \wp_{113} \wp_{122}  &= \textstyle2\wp_{13}\wp_{12}^{2}
-4\lambda_{{3}}\lambda_{{0}}
+2\wp_{11} \wp_{13} \wp_{22}
-8\lambda_{{0}}\wp_{33} \wp_{13}
+4\lambda_{{0}}  \wp_{23}^{2}
-2\lambda_{{1}}\wp_{11}\\
& & & \quad \textstyle  +4\lambda_{{2}}  \wp_{13}^{2}
-4\lambda_{{1}}\wp_{23} \wp_{13}
+4\lambda_{{0}}\wp_{33} \wp_{22}
-4\lambda_{{0}}\wp_{12}
+2\lambda_{{3}}\wp_{13} \wp_{12}
-2\lambda_{{1}}\wp_{33} \wp_{12}
\\
\end{align*}
\begin{align*}
&\bm{(-22)} &   \wp_{113}^{2} &= \textstyle 4\wp_{11}\wp_{13}^{2}
+4\lambda_{{0}}\wp_{23}^{2}
+4\lambda_{{2}}\wp_{13}^{2}
-4\lambda_{{1}}\wp_{23} \wp_{13}
\\
&\bm{(-22)} &  \wp_{112} \wp_{222}  &= \textstyle 2\wp_{11}\wp_{22}^{2}
+2\wp_{12}^{2}\wp_{22}
-4\lambda_{{6}}\lambda_{{3}}\lambda_{{1}}
-8\wp_{11} \Delta-8\lambda_{{3}}\lambda_{{0}}
-8\lambda_{{2}}\wp_{11} \wp_{33}\\
& & & \quad \textstyle
-4\lambda_{{3}}\lambda_{{1}}\wp_{33}
-8\lambda_{{6}}\lambda_{{2}}\wp_{11}
+2\lambda_{{3}}\wp_{22} \wp_{12}
-2\lambda_{{5}}\lambda_{{1}}\wp_{22}
+8\lambda_{{5}}\lambda_{{1}}\wp_{13} \\
& & & \quad \textstyle -8\lambda_{{0}}\wp_{23}^{2}
-4\lambda_{{1}}\wp_{11}
+8\lambda_{{1}}\wp_{23} \wp_{13}
-8\lambda_{{5}}\lambda_{{0}}\wp_{23}
-8\lambda_{{0}}\wp_{12}
-4\lambda_{{1}}\wp_{33} \wp_{12}\\
& & & \quad \textstyle +4\lambda_{{2}}\wp_{12} \wp_{23}
+4\lambda_{{5}}\wp_{12} \wp_{11}
-2\lambda_{{1}}\wp_{23} \wp_{22}
-4\lambda_{{6}}\lambda_{{1}}\wp_{12}
-8\lambda_{{6}}\wp_{11}^{2}
+4\lambda_{{2}}\lambda_{{5}}\wp_{12}
\\
&\bm{(-22)} &  \wp_{112} \wp_{123}  &= \textstyle 2\wp_{13}\wp_{12}^{2}
-4\lambda_{{3}}\lambda_{{0}}
+2\wp_{11} \wp_{23} \wp_{12}
+2\lambda_{{5}}\lambda_{{1}}\wp_{13}
-4\lambda_{{0}}\wp_{23}^{2}\\
& & & \quad \textstyle -2\lambda_{{1}}\wp_{11}
+2\lambda_{{1}}\wp_{23}\wp_{13}
-4\lambda_{{5}}\lambda_{{0}}\wp_{23}
-4\lambda_{{0}}\wp_{12}
+2\lambda_{{3}}\wp_{13} \wp_{12}
\\
&\bm{(-22)} &  \wp_{111} \wp_{223}  &= \textstyle -2\wp_{13}\wp_{12}^{2}
+4\wp_{11}\Delta
+2\wp_{11}\wp_{13}\wp_{22}
+4\wp_{11}\wp_{23} \wp_{12}
-4\lambda_{{2}}\wp_{11} \wp_{33} -2{\lambda_{{3}}}^{2}\wp_{13} \\
& & & \quad \textstyle +2\lambda_{{3}}\lambda_{{1}}\wp_{33}
+8\lambda_{{4}}\wp_{13} \wp_{11}
-2\lambda_{{3}}\wp_{23} \wp_{11}
-4\lambda_{{1}}\lambda_{{4}}\wp_{23}
+2\lambda_{{5}}\lambda_{{1}}\wp_{13} \\
& & & \quad \textstyle +6\lambda_{{1}}\wp_{23} \wp_{13}
-4\lambda_{{5}}\lambda_{{0}}\wp_{23}
-4\lambda_{{0}}\wp_{33} \wp_{22}
-4\lambda_{{3}}\wp_{13} \wp_{12}
+4\lambda_{{1}}\wp_{33} \wp_{12} \\
& & & \quad \textstyle +8\lambda_{{2}}\lambda_{{4}}\wp_{13}
+8\lambda_{{2}}\Delta -4\lambda_{{1}}\wp_{23} \wp_{22}
+8\lambda_{{0}}\wp_{33} \wp_{13} +8\lambda_{{2}}\wp_{12} \wp_{23}
-4\lambda_{{0}}\wp_{23}^{2}
\\
&\bm{(-22)} &  \wp_{111} \wp_{133}  &= \textstyle 2\wp_{11}^{2}\wp_{33}
-2\wp_{11}\Delta
+8\lambda_{{0}}\wp_{33} \wp_{13}
+2\wp_{11}\wp_{13}^{2}
+2\lambda_{{5}}\lambda_{{1}}\wp_{13}\\
& & & \quad \textstyle -4\lambda_{{0}}  \wp_{23}^{2}
+2\lambda_{{1}}\wp_{23} \wp_{13}
-4\lambda_{{5}}\lambda_{{0}}\wp_{23}
-4\lambda_{{0}}\wp_{33} \wp_{22}
+2\lambda_{{1}}\wp_{33} \wp_{12}
\\
&\bm{(-24)} & \wp_{112} \wp_{122}  &= \textstyle 2\wp_{11} \wp_{12} \wp_{22}
+2{\lambda_{{1}}}^{2}
-8\lambda_{{2}}\lambda_{{0}}
-8\lambda_{{6}}\lambda_{{3}}\lambda_{{0}}
+2\wp_{12}^{3} -8\lambda_{{3}}\lambda_{{0}}\wp_{33} \\
& & & \quad \textstyle
+2\lambda_{{5}}\lambda_{{1}}\wp_{12}
+4\lambda_{{1}}\wp_{13}^{2}
-8\lambda_{{6}}\lambda_{{0}}\wp_{12}
+4\lambda_{{2}}\wp_{13} \wp_{12}
+8\lambda_{{5}}\lambda_{{0}}\wp_{13}
-4\lambda_{{1}}\wp_{11} \wp_{33}\\
& & & \quad \textstyle -2\lambda_{{1}}\wp_{13} \wp_{22}
-8\lambda_{{0}}\wp_{33} \wp_{12}
-4\lambda_{{6}}\lambda_{{1}}\wp_{11}
-8\lambda_{{0}}\wp_{11}
-4\lambda_{{5}}\lambda_{{0}}\wp_{22}
+2\lambda_{{3}}\wp_{12}^{2}
\\
&\bm{(-24)} & \wp_{112} \wp_{113}  &= \textstyle 4\wp_{11} \wp_{12} \wp_{13}
+2{\lambda_{{1}}}^{2}
-8\lambda_{{2}}\lambda_{{0}}
+4\lambda_{{0}}\wp_{23} \wp_{22}
+2\lambda_{{1}}\Delta\\
& & & \quad \textstyle -8\lambda_{{0}}\wp_{23} \wp_{13}
+6\lambda_{{1}}\wp_{13}^{2}
+4\lambda_{{2}}\wp_{13} \wp_{12}
-2\lambda_{{1}}\wp_{11} \wp_{33}
-4\lambda_{{1}}\wp_{13} \wp_{22}
-8\lambda_{{0}}\wp_{11}
\\
&\bm{(-24)} & \wp_{111}\wp_{222}  &= \textstyle 18\lambda_{{1}}\wp_{13} \wp_{22}
+8\lambda_{{3}}\wp_{13}\wp_{11}
+2\lambda_{{1}}\wp_{11}\wp_{33}
+6\wp_{11}\wp_{12}\wp_{22}
-16\lambda_{{4}}\lambda_{{0}}\wp_{23} \\
& & & \quad \textstyle +8\lambda_{{4}}\wp_{12} \wp_{11}
+8\lambda_{{3}}\lambda_{{0}}\wp_{33}
-14\lambda_{{1}}  \wp_{13}^{2}
-4\lambda_{{3}}  \wp_{12}^{2}
+2\lambda_{{1}}\Delta
-4\lambda_{{5}} \wp_{11}^{2} \\
& & & \quad \textstyle -16\lambda_{{0}}\wp_{23}\wp_{22}
+16\lambda_{{0}}\wp_{23}\wp_{13}
+16\lambda_{{1}}\lambda_{{4}}\wp_{13}
+8\lambda_{{0}}\wp_{33}\wp_{12}
-2\lambda_{{5}}\lambda_{{1}}\wp_{12} \\
& & & \quad \textstyle -2\lambda_{{3}}\lambda_{{1}}\wp_{23}
-4\lambda_{{2}}\lambda_{{5}}\wp_{11}
-4\lambda_{{2}}\wp_{23}\wp_{11}
+8\lambda_{{2}}\lambda_{{4}}\wp_{12}
-2\lambda_{{3}}\wp_{22}\wp_{11}
-2{\lambda_{{3}}}^{2}\wp_{12} \\
& & & \quad \textstyle
-4\lambda_{{1}}\lambda_{{4}}\wp_{22}
+8\lambda_{{2}}\wp_{22}\wp_{12}
-4\lambda_{{1}}\wp_{22}^{2}
-8\lambda_{{2}}\wp_{13}\wp_{12}
-2\lambda_{{5}}\lambda_{{3}}\lambda_{{1}}
-2\wp_{12}^{3}
\\
&\bm{(-24)} & \wp_{111} \wp_{123}  &= \textstyle 2\wp_{11}^{2}\wp_{23}
+2\wp_{11}\wp_{12} \wp_{13}
-8\lambda_{{0}}\wp_{23} \wp_{22}
+8\lambda_{{0}}\wp_{23} \wp_{13}
+4\lambda_{{1}}\lambda_{{4}}\wp_{13} \\
& & & \quad \textstyle +4\lambda_{{3}}\lambda_{{0}}\wp_{33}
-4\lambda_{{1}}\wp_{13}^{2}
+2\lambda_{{1}}\wp_{11} \wp_{33}
+4\lambda_{{1}}\wp_{13} \wp_{22}
-8\lambda_{{4}}\lambda_{{0}}\wp_{23} \\
& & & \quad \textstyle +2\lambda_{{3}}\wp_{13} \wp_{11}
+4\lambda_{{0}}\wp_{33} \wp_{12}
\\
&\bm{(-26)} &   \wp_{112}^{2} &= \textstyle -16\lambda_{{0}}\Delta
+4\lambda_{{6}}{\lambda_{{1}}}^{2}
-16\lambda_{{6}}\lambda_{{2}}\lambda_{{0}}
+4\wp_{11}\wp_{12}^{2}
+8\lambda_{{1}}\wp_{13} \wp_{12}
-16\lambda_{{6}}\lambda_{{0}}\wp_{11}\\
& & & \quad \textstyle -16\lambda_{{0}}\wp_{12} \wp_{23}
-16\lambda_{{2}}\lambda_{{0}}\wp_{33}
+4\lambda_{{0}}\wp_{22}^{2}
-4\lambda_{{1}}\wp_{22} \wp_{12}
+4\lambda_{{2}}\wp_{12}^{2}
+4{\lambda_{{1}}}^{2}\wp_{33}
\\
&\bm{(-26)} & \wp_{111} \wp_{122}  &= \textstyle 2\wp_{11}^{2}\wp_{22}
-4\lambda_{{5}}\lambda_{{3}}\lambda_{{0}}
+2\wp_{11}\wp_{12}^{2}
-8\lambda_{{3}}\lambda_{{0}}\wp_{23}
-4\lambda_{{1}}\wp_{23} \wp_{11}
+4\lambda_{{2}}\wp_{13} \wp_{11} \\
& & & \quad \textstyle
+8\lambda_{{0}}\wp_{11}\wp_{33}
-2\lambda_{{1}}\wp_{13}\wp_{12}
+4\lambda_{{1}}\lambda_{{4}}\wp_{12}
-2\lambda_{{5}}\lambda_{{1}}\wp_{11}
-4\lambda_{{5}}\lambda_{{0}}\wp_{12} \\
& & & \quad \textstyle -8\lambda_{{4}}\lambda_{{0}}\wp_{22}
-8\lambda_{{0}}\wp_{12} \wp_{23}
+8\lambda_{{2}}\lambda_{{0}}\wp_{33}
-8\lambda_{{0}}\wp_{22}^{2}
+4\lambda_{{1}}\wp_{22} \wp_{12}
-2{\lambda_{{1}}}^{2}\wp_{33} \\
& & & \quad \textstyle +24\lambda_{{0}}\wp_{13}\wp_{22}
+16\lambda_{{4}}\lambda_{{0}}\wp_{13}
+2\lambda_{{3}}\lambda_{{1}}\wp_{13}
+2\lambda_{{3}}\wp_{12} \wp_{11}
-16\lambda_{{0}}\wp_{13}^{2}
\\
&\bm{(-26)} & \wp_{111} \wp_{113}  &= \textstyle 4\wp_{11}^{2}\wp_{13}
-4\lambda_{{3}}\lambda_{{0}}\wp_{23}
-2\lambda_{{1}}\wp_{23}\wp_{11}
+4\lambda_{{2}}\wp_{13}\wp_{11}
+8\lambda_{{0}}\wp_{11}\wp_{33}\\
& & & \quad \textstyle +2\lambda_{{1}}\wp_{13}\wp_{12}
+2\lambda_{{3}}\lambda_{{1}}\wp_{13}
-4\lambda_{{0}}\wp_{12} \wp_{23}
+8\lambda_{{2}}\lambda_{{0}}\wp_{33}
-2{\lambda_{{1}}}^{2}\wp_{33}
\\
&\bm{(-28)} & \wp_{111} \wp_{112}  &= \textstyle 2\lambda_{{5}}{\lambda_{{1}}}^{2}
-8\lambda_{{2}}\lambda_{{5}}\lambda_{{0}}
+4\wp_{11}^{2}\wp_{12}
-8\lambda_{{0}}\wp_{23}\wp_{11}
+2\lambda_{{3}}\lambda_{{1}}\wp_{12}
-4\lambda_{{0}}\wp_{22} \wp_{12}\\
& & & \quad \textstyle -8\lambda_{{2}}\lambda_{{0}}\wp_{23}
+8\lambda_{{0}}\wp_{13}\wp_{12}
+4\lambda_{{1}}\wp_{13}\wp_{11}
-4\lambda_{{3}}\lambda_{{0}}\wp_{22}
+4\lambda_{{2}}\wp_{12}\wp_{11} \\
& & & \quad \textstyle -2\lambda_{{1}}\wp_{22}\wp_{11}
+8\lambda_{{3}}\lambda_{{0}}\wp_{13}
-8\lambda_{{5}}\lambda_{{0}}\wp_{11}
+2\lambda_{{1}}\wp_{12}^{2}
+2{\lambda_{{1}}}^{2}\wp_{23}
\\
&\bm{(-30)} &   \wp_{111}^{2} &= \textstyle 4\wp_{11}^{3}
+4{\lambda_{{1}}}^{2}\lambda_{{4}}
+4{\lambda_{{3}}}^{2}\lambda_{{0}}
-16\lambda_{{2}}\lambda_{{4}}\lambda_{{0}}
+8\lambda_{{3}}\lambda_{{0}}\wp_{12}
+4\lambda_{{0}}\wp_{12}^{2}
+4\lambda_{{2}}\wp_{11}^{2} \\
& & & \quad \textstyle
-16\lambda_{{2}}\lambda_{{0}}\wp_{22}
+16\lambda_{{0}}\wp_{13} \wp_{11}
-16\lambda_{{0}}\wp_{22} \wp_{11}
+4\lambda_{{1}}\wp_{12} \wp_{11}
+16\lambda_{{2}}\lambda_{{0}}\wp_{13} \\
& & & \quad \textstyle
+4\lambda_{{3}}\lambda_{{1}}\wp_{11}
-16\lambda_{{4}}\lambda_{{0}}\wp_{11}
-4{\lambda_{{1}}}^{2}\wp_{13}
+4{\lambda_{{1}}}^{2}\wp_{22}
\end{align*}

\section{Quadratic 3-index relations associated with the cyclic trigonal curve of genus three} \label{APP_34quad}

This appendix contains the complete set of quadratic 3-index relations associated to the cyclic $(3,4)$-curve.
Note that $Q_{1333}$ was the sole $Q$-function used in the basis of fundamental Abelian functions, (\ref{eq:34_2pole}).  It appears here as a linear term, or multiplied by a single 2-index $\wp$-function.  The relations are presented in decreasing weight order as indicated by the number in brackets.

\begin{align*}
&\bm{(-6)} &  \wp_{333}^2 &= \textstyle 4\wp_{33}^3 + 4\wp_{13} + \wp_{23}^2 - 4\wp_{22}\wp_{33}
\\
&\bm{(-7)} &  \wp_{233}\wp_{333} &= \textstyle 4\wp_{23}\wp_{33}^2
- 2\wp_{12} - \wp_{22}\wp_{23} + 2\lambda_{3}\wp_{33}^2
\\
&\bm{(-8)} &  \wp_{233}^2 &= \textstyle 4\wp_{23}^2\wp_{33} + \wp_{22}^2
+ 4\lambda_{3}\wp_{23}\wp_{33} + 4\wp_{33}\lambda_{2} - \frac{4}{3}Q_{1333}
\\
&\bm{(-8)} &  \wp_{223}\wp_{333} &= \textstyle 2\wp_{23}^2\wp_{33}
+ 2\wp_{22}\wp_{33}^2 - 2\wp_{22}^2 + \lambda_{3}\wp_{23}\wp_{33}
+ 4\wp_{13}\wp_{33} + \frac{2}{3}Q_{1333}
\\
&\bm{(-9)} &  \wp_{223}\wp_{233} &= \textstyle 2\wp_{13}\lambda_{3} + 2\lambda_{1}
+ 4\wp_{13}\wp_{23} + 2\wp_{23}\wp_{22}\wp_{33} + 2\wp_{23}^3
\\& & &\quad \textstyle
+ \wp_{22}\wp_{33}\lambda_{3} + 2\wp_{23}^2\lambda_{3} + 2\wp_{23}\lambda_{2}
\\
&\bm{(-9)} &  \wp_{222}\wp_{333} &= \textstyle 6\wp_{23}\wp_{22}\wp_{33}
- 4\wp_{12}\wp_{33} - 2\wp_{23}^3 + 4\wp_{22}\wp_{33}\lambda_{3}
- 8\wp_{13}\wp_{23}
- \wp_{23}^2\lambda_{3} - 4\wp_{13}\lambda_{3}
\\
&\bm{(-10)} &  \wp_{223}^2 &= \textstyle 4\wp_{23}^2\wp_{22} + 4\wp_{11} - 4\wp_{12}\wp_{23}
+ 4\wp_{13}\wp_{22} - 4\wp_{12}\lambda_{3} \\
& & &\quad \textstyle + 4\lambda_{3}\wp_{22}\wp_{23} + \lambda_{3}^2\wp_{33}^2
+ 4\lambda_{2}\wp_{22} - 4\lambda_{2}\wp_{33}^2 + \frac{4}{3}\wp_{33}Q_{1333}
\\
&\bm{(-10)} &  \wp_{222}\wp_{233} &= \textstyle 2\wp_{23}^2\wp_{22} + 2\wp_{22}^2\wp_{33}
+ 4\wp_{12}\wp_{23} + 2\wp_{12}\lambda_{3} + \lambda_{3}\wp_{22}\wp_{23} \\
& & &\quad \textstyle - 2\lambda_{3}^2\wp_{33}^2 + 8\lambda_{2}\wp_{33}^2 - \frac{8}{3}\wp_{33}Q_{1333}
\\
&\bm{(-10)} &  \wp_{133}\wp_{333} &= \textstyle \wp_{12}\wp_{23} - 2\wp_{13}\wp_{22}
+ 4\wp_{13}\wp_{33}^2 + \frac{2}{3}\wp_{33}Q_{1333}
\\
&\bm{(-11)} &  \wp_{222}\wp_{223} &= \textstyle 2\wp_{22}^2\lambda_{3}
- \lambda_{3}^2\wp_{23}\wp_{33} + 4\wp_{22}^2\wp_{23} + 4\lambda_{2}\wp_{23}\wp_{33}
+ 4\wp_{33}\lambda_{1}
\\& & &\quad \textstyle
- \frac{2}{3}\lambda_{3}Q_{1333} - \frac{4}{3}\wp_{23}Q_{1333}
\\
&\bm{(-11)} &  \wp_{133}\wp_{233} &= \textstyle 2\lambda_{3}\wp_{13}\wp_{33}
+ 2\wp_{33}\lambda_{1} + \wp_{12}\wp_{22} + 4\wp_{33}\wp_{23}\wp_{13} + \frac{2}{3}\wp_{23}Q_{1333}
\\
&\bm{(-11)} &  \wp_{123}\wp_{333} &= \textstyle 2\wp_{12}\wp_{33}^2 - 2\wp_{12}\wp_{22}
+ 2\wp_{33}\wp_{23}\wp_{13} + 2\lambda_{3}\wp_{13}\wp_{33} - \frac{1}{3}\wp_{23}Q_{1333}
\\
&\bm{(-12)} &
\wp_{222}^2 &= \textstyle 4\wp_{22}^3 + 8\wp_{11}\wp_{33} - 8\wp_{13}^2
- 4\wp_{33}\wp_{12}\lambda_{3} + 4\wp_{23}\wp_{13}\lambda_{3} + 4\wp_{13}\lambda_{3}^2 \\
& & &\quad \textstyle - 4\wp_{22}\wp_{33}\lambda_{3}^2 + \wp_{23}^2\lambda_{3}^2
- 8\wp_{13}\lambda_{2} + 16\wp_{33}\lambda_{2}\wp_{22} - 4\wp_{23}^2\lambda_{2} \\
& & &\quad \textstyle
- 8\lambda_{0} - 4\wp_{22}Q_{1333} - 4\wp_{23}\lambda_{1}
\\
&\bm{(-12)} &
\wp_{133}\wp_{223} &= \textstyle 2\wp_{13}\wp_{33}\wp_{22} + 2\wp_{13}\wp_{23}^2 + 2\wp_{11}\wp_{33}
+ 2\wp_{13}^2 - \wp_{33}\wp_{12}\lambda_{3} \\
& & &\quad \textstyle
+ 2\wp_{13}\lambda_{2} + 2\wp_{23}\wp_{13}\lambda_{3}
+ 2\lambda_{0} + \frac{2}{3}\wp_{22}Q_{1333}
\\
&\bm{(-12)} &
\wp_{123}\wp_{233} &= \textstyle 2\wp_{13}\lambda_{2} + 2\wp_{13}^2 + 2\wp_{33}\wp_{12}\lambda_{3}
+ 2\wp_{23}\wp_{33}\wp_{12} + 2\wp_{13}\wp_{23}^2 \\
& & &\quad \textstyle - 2\wp_{11}\wp_{33} + 2\wp_{23}\wp_{13}\lambda_{3}
+ 2\lambda_{0} - \frac{1}{3}\wp_{22}Q_{1333}
\\
&\bm{(-12)} &
\wp_{122}\wp_{333} &= \textstyle 2\wp_{13}\wp_{33}\wp_{22} - 2\wp_{13}\wp_{23}^2
+ 4\wp_{23}\wp_{33}\wp_{12} - 2\wp_{11}\wp_{33} - 6\wp_{13}^2 + 2\lambda_{0} \\
& & &\quad \textstyle + 2\wp_{33}\wp_{12}\lambda_{3} - \wp_{23}\wp_{13}\lambda_{3}
- 2\wp_{13}\lambda_{2} + \wp_{23}\lambda_{1} + \frac{2}{3}\wp_{22}Q_{1333}
\\
&\bm{(-13)} &
\wp_{133}\wp_{222} &= \textstyle  - 2\wp_{12}\wp_{23}^2 + 4\wp_{23}\wp_{13}\wp_{22}
- \wp_{23}\wp_{12}\lambda_{3} + 2\wp_{22}\wp_{13}\lambda_{3}  \\
& & &\quad \textstyle + 2\wp_{33}\wp_{12}\wp_{22} + 4\wp_{33}^2\lambda_{1}
- \frac{2}{3}\wp_{33}\lambda_{3}Q_{1333}
\\
&\bm{(-13)} &
\wp_{123}\wp_{223} &= \textstyle 2\wp_{22}\lambda_{1} + 2\wp_{22}\wp_{13}\lambda_{3}
- 2\wp_{33}^2\lambda_{1} + 2\wp_{12}\wp_{23}^2 + 2\wp_{23}\wp_{13}\wp_{22} \\
& & &\quad \textstyle - 2\wp_{11}\wp_{23} + 2\wp_{23}\wp_{12}\lambda_{3}
+ \frac{1}{3}\wp_{33}\lambda_{3}Q_{1333}
\\
&\bm{(-13)} &
\wp_{122}\wp_{233} &= \textstyle  - \wp_{22}\wp_{13}\lambda_{3} + 4\wp_{12}\wp_{13}
+ 2\wp_{12}\wp_{23}^2 + 2\wp_{12}\lambda_{2} - \wp_{22}\lambda_{1} \\
& & &\quad \textstyle + 2\wp_{23}\wp_{12}\lambda_{3} + 4\wp_{33}^2\lambda_{1}
+ 2\wp_{33}\wp_{12}\wp_{22} - \frac{2}{3}\wp_{33}\lambda_{3}Q_{1333}
\\
&\bm{(-14)} &
\wp_{133}^2 &= \textstyle \wp_{12}^2 + 4\wp_{33}\lambda_{0} + 4\wp_{33}\wp_{13}^2
+ \frac{4}{3}\wp_{13}Q_{1333}
\\
\end{align*}
\begin{align*}
&\bm{(-14)} &
\wp_{123}\wp_{222} &= \textstyle 2\wp_{22}\wp_{23}\wp_{12} - 2\wp_{12}^2
+ 2\wp_{13}\wp_{22}^2 - 2\wp_{23}\wp_{33}\lambda_{1} - 2\lambda_{3}^2\wp_{13}\wp_{33} \\
& & &\quad \textstyle + 2\wp_{22}\wp_{12}\lambda_{3} + 8\wp_{33}\wp_{13}\lambda_{2}
- \frac{8}{3}\wp_{13}Q_{1333} + \frac{1}{3}\wp_{23}\lambda_{3}Q_{1333}
\\
&\bm{(-14)} &
\wp_{122}\wp_{223} &= \textstyle 4\wp_{22}\wp_{23}\wp_{12} - 2\wp_{11}\wp_{22}
+ 2\wp_{12}^2 + 2\wp_{22}\wp_{12}\lambda_{3} + \lambda_{3}^2\wp_{13}\wp_{33} \\
& & &\quad \textstyle + 4\wp_{23}\wp_{33}\lambda_{1} + \wp_{33}\lambda_{3}\lambda_{1}
+ 8\wp_{33}\lambda_{0} + \frac{4}{3}\wp_{13}Q_{1333} \\
& & &\quad \textstyle - \frac{2}{3}\wp_{23}\lambda_{3}Q_{1333}
- 4\wp_{33}\wp_{13}\lambda_{2} - \frac{2}{3}\lambda_{2}Q_{1333}
\\
&\bm{(-14)} &
\wp_{113}\wp_{333} &= \textstyle 2\wp_{33}\wp_{13}^2 + 2\wp_{33}^2\wp_{11}
- 2\wp_{12}^2 + 2\wp_{33}\wp_{13}\lambda_{2} - \wp_{23}\wp_{33}\lambda_{1}
\\& & &\quad \textstyle
- 2\wp_{33}\lambda_{0} - \frac{2}{3}\wp_{13}Q_{1333}
\\
&\bm{(-15)} &
\wp_{123}\wp_{133} &= \textstyle 2\wp_{23}\wp_{13}^2 + 2\wp_{13}\lambda_{1}
+ 2\wp_{13}^2\lambda_{3} - 2\wp_{23}\lambda_{0} + 2\wp_{33}\wp_{12}\wp_{13}
+ \frac{1}{3}\wp_{12}Q_{1333}
\\
&\bm{(-15)} &
\wp_{122}\wp_{222} &= \textstyle 6\wp_{33}\wp_{22}\lambda_{1} - \wp_{23}\lambda_{3}\lambda_{1}
- 2\wp_{33}\wp_{12}\lambda_{3}^2 - 2\wp_{23}^2\lambda_{1} + 2\lambda_{3}\wp_{13}\lambda_{2} \\
& & &\quad \textstyle + 2\lambda_{3}\wp_{11}\wp_{33} + \wp_{23}\wp_{13}\lambda_{3}^2
- 2\wp_{13}^2\lambda_{3}  - \frac{2}{3}\wp_{22}\lambda_{3}Q_{1333} \\
& & &\quad \textstyle + 4\wp_{33}\wp_{12}\lambda_{2} - 2\lambda_{3}\lambda_{0}
+ 4\wp_{22}^2\wp_{12} - \frac{4}{3}\wp_{12}Q_{1333}  - 8\wp_{13}\lambda_{1}
\\
&\bm{(-15)} &
\wp_{113}\wp_{233} &= \textstyle 2\wp_{13}^2\lambda_{3} + 2\wp_{23}\wp_{13}^2
+ 2\wp_{33}\wp_{12}\lambda_{2} + 2\wp_{13}\lambda_{1} + 2\wp_{33}\wp_{23}\wp_{11} \\
& & &\quad \textstyle - 2\wp_{23}\lambda_{0} - \wp_{33}\wp_{22}\lambda_{1}
- \frac{2}{3}\wp_{12}Q_{1333}
\\
&\bm{(-15)} &
\wp_{112}\wp_{333} &= \textstyle 2\wp_{23}\wp_{13}\lambda_{2} - \wp_{23}^2\lambda_{1}
- 4\wp_{13}^2\lambda_{3} + 4\wp_{33}\wp_{12}\wp_{13} - 2\wp_{23}\lambda_{0} \\
& & &\quad \textstyle + 2\wp_{33}\wp_{23}\wp_{11} - 2\wp_{23}\wp_{13}^2 + \frac{4}{3}\wp_{12}Q_{1333}
\\
&\bm{(-16)} &
\wp_{123}^2 &= \textstyle \frac{1}{3}\wp_{23}\wp_{12}\lambda_{2}
- \frac{8}{3}\wp_{11}\wp_{13} + \frac{10}{3}\wp_{13}\wp_{23}\wp_{12}
- \frac{1}{3}\wp_{33}\wp_{11}\wp_{22}  + \frac{2}{9}\wp_{33}\lambda_{2}Q_{1333} \\
& & &\quad \textstyle - \frac{1}{3}\wp_{23}\wp_{22}\lambda_{1}
+ \frac{1}{3}\wp_{22}\wp_{13}\lambda_{2} + 3\wp_{22}\lambda_{0} + \frac{1}{3}\wp_{12}\lambda_{1}
+ \frac{1}{3}\wp_{23}^2\wp_{11} \\
& & &\quad \textstyle + \frac{1}{3}\wp_{13}^2\wp_{22} - \frac{8}{3}\wp_{33}^2\lambda_{0}
+ \frac{1}{3}\wp_{33}\wp_{12}^2 + 3\wp_{12}\wp_{13}\lambda_{3}
- \frac{1}{3}\wp_{33}^2\lambda_{3}\lambda_{1}
\\
&\bm{(-16)} &
\wp_{122}\wp_{133} &= \textstyle  - \frac{8}{3}\wp_{11}\wp_{13}
+ \frac{4}{3}\wp_{13}\wp_{23}\wp_{12} + \frac{2}{3}\wp_{33}\wp_{11}\wp_{22}
- \frac{2}{3}\wp_{23}^2\wp_{11} \\
& & &\quad \textstyle + \frac{4}{3}\wp_{13}^2\wp_{22} + \frac{4}{3}\wp_{33}\wp_{12}^2
+ \frac{2}{3}\wp_{23}\wp_{22}\lambda_{1} - \frac{2}{3}\wp_{22}\wp_{13}\lambda_{2}
- \frac{2}{3}\wp_{23}\wp_{12}\lambda_{2} \\
& & &\quad \textstyle + 3\wp_{12}\wp_{13}\lambda_{3} + \frac{2}{3}\wp_{33}^2\lambda_{3}\lambda_{1}
+ \frac{1}{3}\wp_{12}\lambda_{1} + \frac{16}{3}\wp_{33}^2\lambda_{0}
- \frac{4}{9}\wp_{33}\lambda_{2}Q_{1333}
\\
&\bm{(-16)} &
\wp_{113}\wp_{223} &= \textstyle \frac{4}{3}\wp_{11}\wp_{13} + \frac{4}{3}\wp_{13}\wp_{23}\wp_{12}
+ \frac{2}{3}\wp_{33}\wp_{11}\wp_{22} + \frac{4}{3}\wp_{23}^2\wp_{11}
+ \frac{4}{3}\wp_{23}\wp_{12}\lambda_{2} \\
& & &\quad \textstyle + \frac{4}{3}\wp_{13}^2\wp_{22} - \frac{2}{3}\wp_{33}\wp_{12}^2
- \frac{4}{3}\wp_{23}\wp_{22}\lambda_{1} + \frac{4}{3}\wp_{22}\wp_{13}\lambda_{2}  \\
& & &\quad \textstyle - \frac{1}{3}\wp_{33}^2\lambda_{3}\lambda_{1}
+ \frac{4}{3}\wp_{12}\lambda_{1} - \frac{8}{3}\wp_{33}^2\lambda_{0}
+ \frac{2}{9}\wp_{33}\lambda_{2}Q_{1333}
\\
&\bm{(-16)} &
\wp_{112}\wp_{233} &= \textstyle \frac{16}{3}\wp_{11}\wp_{13}
+ \frac{4}{3}\wp_{13}\wp_{23}\wp_{12} + \frac{2}{3}\wp_{33}\wp_{11}\wp_{22}
+ \frac{4}{3}\wp_{23}^2\wp_{11} - \frac{4}{9}\wp_{33}\lambda_{2}Q_{1333}  \\
& & &\quad \textstyle - \frac{2}{3}\wp_{13}^2\wp_{22} + \frac{4}{3}\wp_{33}\wp_{12}^2
- \frac{1}{3}\wp_{23}\wp_{22}\lambda_{1}
- \frac{2}{3}\wp_{22}\wp_{13}\lambda_{2} - 6\wp_{22}\lambda_{0}  \\
& & &\quad \textstyle + \frac{4}{3}\wp_{23}\wp_{12}\lambda_{2}
+ \frac{2}{3}\wp_{33}^2\lambda_{3}\lambda_{1} + \frac{10}{3}\wp_{12}\lambda_{1}
+ \frac{16}{3}\wp_{33}^2\lambda_{0}
\\
&\bm{(-17)} &
\wp_{122}\wp_{123} &= \textstyle 2\wp_{33}\lambda_{3}\lambda_{0}
+ 2\wp_{12}\wp_{13}\wp_{22} - 2\wp_{11}\wp_{12} + 2\lambda_{1}\wp_{13}\wp_{33}
+ 2\lambda_{3}\wp_{12}^2 \\
& & &\quad \textstyle + 2\wp_{23}\wp_{12}^2 - \frac{1}{3}\lambda_{1}Q_{1333}
- \frac{1}{3}\wp_{13}\lambda_{3}Q_{1333}
\\
&\bm{(-17)} &
\wp_{113}\wp_{222} &= \textstyle  - 8\wp_{23}\wp_{33}\lambda_{0}
- 4\wp_{33}\lambda_{3}\lambda_{0} + 8\lambda_{1}\wp_{13}\wp_{33} + 4\wp_{22}\wp_{12}\lambda_{2}
+ \frac{2}{3}\lambda_{1}Q_{1333}  \\
& & &\quad \textstyle + 4\wp_{12}\wp_{13}\wp_{22} + 4\wp_{11}\wp_{12}
- \lambda_{3}\wp_{23}\wp_{33}\lambda_{1} - 2\wp_{23}\wp_{12}^2 + 2\wp_{22}\wp_{23}\wp_{11} \\
& & &\quad \textstyle - 4\lambda_{3}\wp_{12}^2 - 2\wp_{22}^2\lambda_{1}
+ \frac{2}{3}\wp_{23}\lambda_{2}Q_{1333}
- \frac{4}{3}\wp_{13}\lambda_{3}Q_{1333}
\\
&\bm{(-17)} &
\wp_{112}\wp_{223} &= \textstyle 8\wp_{33}\lambda_{3}\lambda_{0}
+ \lambda_{3}\wp_{23}\wp_{33}\lambda_{1} - 4\lambda_{1}\wp_{13}\wp_{33}
+ 2\wp_{23}\wp_{12}^2  + \frac{2}{3}\wp_{13}\lambda_{3}Q_{1333} \\
& & &\quad \textstyle + 2\wp_{22}\wp_{23}\wp_{11} + 8\wp_{23}\wp_{33}\lambda_{0} + 2\lambda_{3}\wp_{12}^2
- \frac{2}{3}\wp_{23}\lambda_{2}Q_{1333} - \frac{4}{3}\lambda_{1}Q_{1333}
\\
&\bm{(-18)} &
\wp_{122}^2 &= \textstyle  - 4\lambda_{2}\wp_{13}^2 - 4\lambda_{2}\lambda_{0}
+ \lambda_{1}^2 - 16\lambda_{0}\wp_{13} - 8\lambda_{0}\wp_{23}^2 + 4\wp_{22}\wp_{12}^2 \\
& & &\quad \textstyle + 2\lambda_{3}\wp_{13}\lambda_{1} + \wp_{13}^2\lambda_{3}^2
+ 4\lambda_{1}\wp_{13}\wp_{23} + 4\lambda_{1}\wp_{12}\wp_{33}  \\
& & &\quad \textstyle + 8\wp_{33}\wp_{22}\lambda_{0} - \frac{4}{3}\wp_{12}\lambda_{3}Q_{1333}
+ \frac{4}{3}\wp_{11}Q_{1333} - 4\lambda_{3}\wp_{23}\lambda_{0}
\\
&\bm{(-18)} &
\wp_{111}\wp_{333} &= \textstyle  - 9\lambda_{3}\wp_{23}\lambda_{0}
- 2\wp_{13}\lambda_{2}^2 + \lambda_{2}\wp_{23}\lambda_{1} - 2\lambda_{2}\wp_{11}\wp_{33}
+ 2\wp_{11}Q_{1333} \\
& & &\quad \textstyle - 10\lambda_{0}\wp_{23}^2 + 6\wp_{33}\wp_{13}\wp_{11}
- 10\lambda_{0}\wp_{13} - 2\wp_{13}^3 + 7\lambda_{1}\wp_{13}\wp_{23} \\
& & &\quad \textstyle + 2\lambda_{1}\wp_{12}\wp_{33} - 2\wp_{33}\wp_{22}\lambda_{0}
+ 6\lambda_{3}\wp_{13}\lambda_{1} + 2\lambda_{1}^2 - 6\lambda_{2}\lambda_{0} - 4\lambda_{2}\wp_{13}^2
\\
\end{align*}
\begin{align*}
&\bm{(-18)} &
\wp_{113}\wp_{133} &= \textstyle 2\wp_{33}\wp_{13}\wp_{11} + 2\lambda_{2}\wp_{13}^2
+ 2\lambda_{0}\wp_{13} + 2\lambda_{0}\wp_{23}^2 + 2\wp_{13}^3  \\
& & &\quad \textstyle + \lambda_{1}\wp_{12}\wp_{33} - 2\wp_{33}\wp_{22}\lambda_{0}
- 2\lambda_{1}\wp_{13}\wp_{23}
\\
&\bm{(-18)} &
\wp_{112}\wp_{222} &= \textstyle 4\lambda_{2}\wp_{13}^2 + 2\wp_{23}\lambda_{3}\wp_{13}\lambda_{2}
- 2\lambda_{2}\wp_{23}\lambda_{1} + 4\lambda_{2}\wp_{11}\wp_{33} - \wp_{23}^2\lambda_{3}\lambda_{1} \\
& & &\quad \textstyle + 4\lambda_{2}\lambda_{0} - 2\lambda_{1}^2 + 8\lambda_{0}\wp_{23}^2
+ 2\wp_{11}\wp_{22}^2 + 2\wp_{22}\wp_{12}^2 + \frac{2}{3}\wp_{12}\lambda_{3}Q_{1333} \\
& & &\quad \textstyle - 2\lambda_{3}\wp_{33}\wp_{12}\lambda_{2}
+ 2\lambda_{3}\wp_{33}\wp_{22}\lambda_{1} + 4\wp_{13}\lambda_{2}^2
- 8\lambda_{3}\wp_{13}\lambda_{1} - \frac{8}{3}\wp_{11}Q_{1333}  \\
& & &\quad \textstyle - 8\lambda_{1}\wp_{13}\wp_{23} + 4\lambda_{1}\wp_{12}\wp_{33}
+ 8\lambda_{3}\wp_{23}\lambda_{0} - \frac{2}{3}\lambda_{2}\wp_{22}Q_{1333}
- 2\wp_{13}^2\lambda_{3}^2
\\
&\bm{(-19)} &
\wp_{113}\wp_{123} &= \textstyle  - 2\wp_{33}^2\lambda_{3}\lambda_{0}
+ 2\lambda_{2}\wp_{12}\wp_{13} + 2\wp_{12}\wp_{13}^2 + 2\wp_{12}\lambda_{0}
- 2\wp_{23}\wp_{22}\lambda_{0} \\
& & &\quad \textstyle + 2\wp_{13}\wp_{23}\wp_{11} + \frac{1}{3}\wp_{33}\lambda_{1}Q_{1333}
\\
&\bm{(-19)} &
\wp_{111}\wp_{233} &= \textstyle 2\wp_{12}\wp_{33}\wp_{11} - 2\wp_{12}\wp_{13}^2
+ 6\lambda_{1}\wp_{12}\lambda_{3} + 6\lambda_{1}\wp_{12}\wp_{23}  \\
& & &\quad \textstyle + 4\wp_{13}\wp_{23}\wp_{11} - 2\wp_{12}\lambda_{2}^2 - 2\wp_{11}\lambda_{1}
- 4\lambda_{2}\wp_{12}\wp_{13} + \lambda_{2}\wp_{22}\lambda_{1} \\
& & &\quad \textstyle - 2\lambda_{2}\wp_{23}\wp_{11} - 8\wp_{23}\wp_{22}\lambda_{0}
+ \wp_{13}\wp_{22}\lambda_{1} + 4\wp_{33}^2\lambda_{3}\lambda_{0} \\
& & &\quad \textstyle + 6\wp_{11}\wp_{13}\lambda_{3}  - 9\lambda_{3}\wp_{22}\lambda_{0}
+ 2\wp_{12}\lambda_{0} - \frac{2}{3}\wp_{33}\lambda_{1}Q_{1333}
\\
&\bm{(-19)} &
\wp_{112}\wp_{133} &= \textstyle 4\wp_{33}^2\lambda_{3}\lambda_{0} + 2\wp_{12}\wp_{33}\wp_{11}
+ 2\lambda_{2}\wp_{12}\wp_{13}  - 2\wp_{13}\wp_{22}\lambda_{1} + 2\wp_{12}\lambda_{0} \\
& & &\quad \textstyle + 2\wp_{12}\wp_{13}^2  + 4\wp_{23}\wp_{22}\lambda_{0}
- \frac{2}{3}\wp_{33}\lambda_{1}Q_{1333} - \lambda_{1}\wp_{12}\wp_{23}
\\
&\bm{(-20)} &
\wp_{113}\wp_{122} &= \textstyle 2\wp_{22}\wp_{13}\wp_{11} - 4\wp_{33}\lambda_{2}\lambda_{0}
+ 2\wp_{12}^2\wp_{13} - 4\wp_{22}^2\lambda_{0} + \wp_{33}\lambda_{1}^2 \\
& & &\quad \textstyle + 2\wp_{12}\wp_{22}\lambda_{1} - 4\wp_{23}\wp_{33}\lambda_{3}\lambda_{0}
+ \lambda_{1}\lambda_{3}\wp_{13}\wp_{33} - \frac{2}{3}\wp_{13}\lambda_{2}Q_{1333} \\
& & &\quad \textstyle + 8\wp_{13}\wp_{33}\lambda_{0} + \frac{4}{3}\lambda_{0}Q_{1333}
+ \frac{2}{3}\wp_{23}\lambda_{1}Q_{1333}
\\
&\bm{(-20)} &
\wp_{112}\wp_{123} &= \textstyle 8\wp_{33}\lambda_{2}\lambda_{0} + 2\wp_{12}^2\wp_{13}
+ 2\wp_{12}\wp_{23}\wp_{11} + 2\wp_{22}^2\lambda_{0} - 2\wp_{33}\lambda_{1}^2 \\
& & &\quad \textstyle + 2\lambda_{2}\wp_{12}^2 - 2\wp_{12}\wp_{22}\lambda_{1}
+ 2\wp_{23}\wp_{33}\lambda_{3}\lambda_{0} - \frac{8}{3}\lambda_{0}Q_{1333}
- \frac{1}{3}\wp_{23}\lambda_{1}Q_{1333}
\\
&\bm{(-20)} &
\wp_{111}\wp_{223} &= \textstyle  - 12\wp_{33}\lambda_{2}\lambda_{0} + 2\wp_{22}\wp_{13}\wp_{11}
+ 6\wp_{11}\lambda_{3}\wp_{12} - 4\wp_{11}^2 + 4\wp_{33}\lambda_{1}^2 \\
& & &\quad \textstyle - 4\wp_{22}^2\lambda_{0} - 2\lambda_{2}\wp_{11}\wp_{22}
- 8\wp_{13}\wp_{33}\lambda_{0}
- \frac{4}{3}\wp_{23}\lambda_{1}Q_{1333} \\
& & &\quad \textstyle - \lambda_{2}\wp_{33}\lambda_{3}\lambda_{1}
+ 8\wp_{23}\wp_{33}\lambda_{3}\lambda_{0} + \frac{2}{3}\wp_{13}\lambda_{2}Q_{1333}
+ \frac{2}{3}\lambda_{2}^2Q_{1333}   \\
& & &\quad \textstyle - 2\lambda_{2}\wp_{12}^2 - 2\wp_{12}^2\wp_{13} + 4\wp_{12}\wp_{23}\wp_{11}
- \lambda_{1}\lambda_{3}\wp_{13}\wp_{33}  \\
& & &\quad \textstyle - 2\lambda_{3}\lambda_{1}Q_{1333}
+ 4\wp_{12}\wp_{22}\lambda_{1} + \frac{4}{3}\lambda_{0}Q_{1333} + 9\wp_{33}\lambda_{3}^2\lambda_{0}
\\
&\bm{(-21)} &
\wp_{111}\wp_{222} &= \textstyle 6\lambda_{3}\lambda_{2}\lambda_{0}
- 6\lambda_{1}\lambda_{0} - 2\lambda_{3}\lambda_{1}^2 - 2\wp_{22}\lambda_{1}Q_{1333}
- 6\lambda_{3}^2\wp_{13}\lambda_{1} - 8\lambda_{3}\lambda_{0}\wp_{13} \\
& & &\quad \textstyle - 2\wp_{12}^3  + 2\wp_{12}\lambda_{2}Q_{1333}
+ 10\lambda_{3}\lambda_{0}\wp_{23}^2 + 4\wp_{23}\lambda_{2}^2\wp_{13}
- 2\lambda_{3}\wp_{11}Q_{1333}  \\
& & &\quad \textstyle + 6\wp_{11}\wp_{22}\wp_{12}  - 4\wp_{33}\wp_{12}\lambda_{2}^2
- 16\lambda_{0}\wp_{13}\wp_{23} + 9\lambda_{3}^2\wp_{23}\lambda_{0}
- 4\wp_{23}\lambda_{1}^2 \\
& & &\quad \textstyle - 2\lambda_{3}\lambda_{2}\wp_{13}^2 + 2\lambda_{3}\wp_{13}\lambda_{2}^2
+ 6\lambda_{2}\wp_{13}\lambda_{1} - 8\lambda_{0}\wp_{12}\wp_{33}
+ 10\lambda_{1}\wp_{13}^2 \\
& & &\quad \textstyle + 2\lambda_{3}\lambda_{2}\wp_{11}\wp_{33} - 2\wp_{23}^2\lambda_{2}\lambda_{1}
+ 4\lambda_{3}\lambda_{1}\wp_{12}\wp_{33} + 2\lambda_{2}\wp_{33}\wp_{22}\lambda_{1} \\
& & &\quad \textstyle - \lambda_{3}\lambda_{2}\wp_{23}\lambda_{1}  + 2\wp_{11}\wp_{33}\lambda_{1}
- 7\lambda_{3}\lambda_{1}\wp_{13}\wp_{23} + 2\lambda_{3}\lambda_{0}\wp_{22}\wp_{33}
\\
&\bm{(-21)} &
\wp_{112}\wp_{122} &= \textstyle 2\wp_{11}\wp_{22}\wp_{12} + 2\wp_{11}\wp_{33}\lambda_{1}
- 2\lambda_{1}\lambda_{0} - 8\lambda_{3}\lambda_{0}\wp_{13} + 2\wp_{12}^3 \\
& & &\quad \textstyle - 2\lambda_{3}\lambda_{0}\wp_{23}^2 + 8\lambda_{0}\wp_{12}\wp_{33}
+ 2\lambda_{3}\lambda_{0}\wp_{22}\wp_{33} + 2\lambda_{2}\wp_{13}\lambda_{1} \\
& & &\quad \textstyle + \lambda_{3}\lambda_{1}\wp_{13}\wp_{23} - \wp_{23}\lambda_{1}^2
- 2\lambda_{1}\wp_{13}^2 - \frac{2}{3}\wp_{12}\lambda_{2}Q_{1333}
\\
&\bm{(-22)} &
\wp_{113}^2 &= \textstyle 4\wp_{13}^2\wp_{11} - 4\lambda_{0}\wp_{13}\wp_{22}
- 4\lambda_{0}\wp_{12}\wp_{23} - 4\wp_{33}^2\lambda_{2}\lambda_{0} + \wp_{33}^2\lambda_{1}^2 \\
& & &\quad \textstyle + 4\lambda_{1}\wp_{12}\wp_{13} + \frac{4}{3}\wp_{33}\lambda_{0}Q_{1333}
\\
&\bm{(-22)} &
\wp_{111}\wp_{133} &= \textstyle 8\lambda_{0}\wp_{12}\wp_{23} - \lambda_{1}\wp_{12}\lambda_{2}
- 6\lambda_{2}\wp_{22}\lambda_{0} + 2\lambda_{2}\wp_{11}\wp_{13} + 2\wp_{22}\lambda_{1}^2 \\
& & &\quad \textstyle - 2\wp_{33}^2\lambda_{1}^2 - 6\lambda_{0}\wp_{11} + 2\wp_{33}\wp_{11}^2
+ 9\lambda_{0}\wp_{12}\lambda_{3} - \frac{8}{3}\wp_{33}\lambda_{0}Q_{1333} \\
& & &\quad \textstyle - \lambda_{1}\wp_{12}\wp_{13} - 2\lambda_{1}\wp_{11}\wp_{23}
+ 2\wp_{13}^2\wp_{11} - 4\lambda_{0}\wp_{13}\wp_{22}
+ 8\wp_{33}^2\lambda_{2}\lambda_{0}
\\
&\bm{(-23)} &
\wp_{112}\wp_{113} &= \textstyle - 4\wp_{23}\wp_{33}\lambda_{2}\lambda_{0} + 2\lambda_{1}\wp_{12}^2
+ \wp_{23}\wp_{33}\lambda_{1}^2 - 4\lambda_{0}\wp_{12}\wp_{22}  \\
& & &\quad \textstyle 4\wp_{12}\wp_{13}\wp_{11}  + \frac{4}{3}\wp_{23}\lambda_{0}Q_{1333}
- \frac{2}{3}\wp_{13}\lambda_{1}Q_{1333} + 4\lambda_{0}\lambda_{3}\wp_{13}\wp_{33}
\\
\end{align*}
\begin{align*}
&\bm{(-23)} &
\wp_{111}\wp_{123} &= \textstyle 2\wp_{12}\wp_{13}\wp_{11} - 2\wp_{33}\lambda_{3}\lambda_{1}^2
+ 2\lambda_{1}\wp_{33}\lambda_{0} - 2\lambda_{0}\lambda_{3}\wp_{13}\wp_{33} \\
& & &\quad \textstyle + 2\wp_{11}\wp_{12}\lambda_{2} - 2\wp_{11}\wp_{22}\lambda_{1}
+ 6\lambda_{2}\wp_{33}\lambda_{3}\lambda_{0} + 8\wp_{23}\wp_{33}\lambda_{2}\lambda_{0} \\
& & &\quad \textstyle + 2\lambda_{0}\wp_{12}\wp_{22} - 2\wp_{23}\wp_{33}\lambda_{1}^2
+ \frac{1}{3}\lambda_{2}\lambda_{1}Q_{1333} - 3\lambda_{3}\lambda_{0}Q_{1333} \\
& & &\quad \textstyle - \frac{8}{3}\wp_{23}\lambda_{0}Q_{1333} + \frac{1}{3}\wp_{13}\lambda_{1}Q_{1333}
+ 2\wp_{23}\wp_{11}^2
\\
&\bm{(-24)} &
\wp_{111}\wp_{122} &= \textstyle  - 3\lambda_{3}\lambda_{1}\lambda_{0} - \lambda_{2}\lambda_{1}^2
+ 4\lambda_{2}^2\lambda_{0} - 8\lambda_{0}^2 - \lambda_{1}\wp_{13}^2\lambda_{3}
- 8\lambda_{1}\lambda_{0}\wp_{23} - 2\wp_{23}^2\lambda_{1}^2  \\
& & &\quad \textstyle + 4\lambda_{2}\lambda_{0}\wp_{23}^2 - \frac{8}{3}\wp_{22}\lambda_{0}Q_{1333}
+ \frac{2}{3}\wp_{12}\lambda_{1}Q_{1333} - \frac{2}{3}\lambda_{2}\wp_{11}Q_{1333}
+ \lambda_{1}^2\wp_{13} \\
& & &\quad \textstyle + 8\lambda_{2}\lambda_{0}\wp_{13} - 9\lambda_{3}^2\lambda_{0}\wp_{13}
- 2\lambda_{3}\lambda_{1}^2\wp_{23} + 2\wp_{11}^2\wp_{22}
- 2\lambda_{2}\lambda_{1}\wp_{12}\wp_{33}  \\
& & &\quad \textstyle + 2\wp_{11}\wp_{12}^2 + 8\lambda_{0}\wp_{13}^2
+ \lambda_{2}\lambda_{3}\wp_{13}\lambda_{1} + 4\lambda_{3}\lambda_{0}\wp_{12}\wp_{33}
+ 2\lambda_{2}\lambda_{1}\wp_{13}\wp_{23} \\
& & &\quad \textstyle - 8\lambda_{3}\lambda_{0}\wp_{13}\wp_{23}
+ 4\lambda_{2}\wp_{33}\wp_{22}\lambda_{0} + 2\lambda_{3}\lambda_{1}\wp_{11}\wp_{33}
+ 6\lambda_{2}\lambda_{3}\wp_{23}\lambda_{0}
\\
&\bm{(-24)} &
\wp_{112}^2 &= \textstyle  - 8\lambda_{0}\wp_{13}^2 + 8\lambda_{0}\wp_{11}\wp_{33}
- 4\lambda_{1}\lambda_{0}\wp_{23} - 4\lambda_{2}\lambda_{0}\wp_{23}^2 - 8\lambda_{2}\lambda_{0}\wp_{13} \\
& & &\quad \textstyle + 4\wp_{11}\wp_{12}^2 - 8\lambda_{0}^2 + 4\lambda_{3}\lambda_{0}\wp_{13}\wp_{23}
+ 4\lambda_{3}\lambda_{0}\wp_{12}\wp_{33} + \wp_{23}^2\lambda_{1}^2 \\
& & &\quad \textstyle + 4\lambda_{1}^2\wp_{13}
+ \frac{4}{3}\wp_{22}\lambda_{0}Q_{1333} - \frac{4}{3}\wp_{12}\lambda_{1}Q_{1333}
\\
&\bm{(-26)} &
\wp_{111}\wp_{113} &= \textstyle 4\lambda_{0}\wp_{33}\wp_{13}\lambda_{2} + 4\wp_{11}^2\wp_{13}
- \wp_{33}\lambda_{2}\lambda_{1}^2 + 4\wp_{33}\lambda_{2}^2\lambda_{0} - \lambda_{1}^2\wp_{13}\wp_{33} \\
& & &\quad \textstyle - 3\wp_{33}\lambda_{3}\lambda_{1}\lambda_{0} + 2\wp_{12}\wp_{11}\lambda_{1}
- 6\wp_{22}\lambda_{0}\wp_{11} + 2\lambda_{0}\wp_{12}^2  \\
& & &\quad \textstyle - 2\lambda_{2}\lambda_{0}Q_{1333} - \frac{4}{3}\lambda_{0}\wp_{13}Q_{1333}
+ \frac{2}{3}\lambda_{1}^2Q_{1333} + 8\wp_{33}\lambda_{0}^2
\\
&\bm{(-27)} &
\wp_{111}\wp_{112} &= \textstyle  - \frac{4}{3}\wp_{12}\lambda_{0}Q_{1333}
+ 2\lambda_{3}\wp_{13}\lambda_{1}^2 + 8\lambda_{2}\lambda_{1}\lambda_{0}
+ \lambda_{1}^2\wp_{13}\wp_{23} - \frac{2}{3}\wp_{11}\lambda_{1}Q_{1333} \\
& & &\quad \textstyle + 6\lambda_{3}\wp_{33}\lambda_{0}\wp_{11} + 4\wp_{11}^2\wp_{12}
+ 4\wp_{23}\lambda_{2}^2\lambda_{0} - 18\lambda_{3}\lambda_{0}^2 - 2\lambda_{1}^3
- 16\wp_{23}\lambda_{0}^2 \\
& & &\quad \textstyle - 2\lambda_{3}\wp_{13}^2\lambda_{0}
- 6\lambda_{3}\wp_{13}\lambda_{2}\lambda_{0} - 3\lambda_{1}\lambda_{3}\wp_{23}\lambda_{0}
+ 2\lambda_{1}\wp_{33}\wp_{22}\lambda_{0} \\
& & &\quad \textstyle - \lambda_{2}\wp_{23}\lambda_{1}^2 - 2\lambda_{1}^2\wp_{12}\wp_{33}
- 2\lambda_{1}\lambda_{0}\wp_{23}^2 + 8\lambda_{1}\lambda_{0}\wp_{13}
+ 4\lambda_{0}\wp_{33}\wp_{12}\lambda_{2}
\\
&\bm{(-30)} &
\wp_{111}^2 &= \textstyle  - 4\wp_{23}\lambda_{1}^3 - 8\wp_{23}^2\lambda_{0}^2
- 8\wp_{13}\lambda_{2}^2\lambda_{0} - 4\lambda_{2}\lambda_{0}\wp_{13}^2
+ 16\wp_{23}\lambda_{2}\lambda_{1}\lambda_{0} \\
& & &\quad \textstyle + 4\wp_{23}\lambda_{1}\lambda_{0}\wp_{13}
+ 6\lambda_{3}\lambda_{1}\lambda_{0}\wp_{13} - 4\lambda_{0}\wp_{11}Q_{1333}
- 36\lambda_{3}\wp_{23}\lambda_{0}^2 \\
& & &\quad \textstyle + 18\lambda_{3}\lambda_{2}\lambda_{1}\lambda_{0} + \lambda_{2}^2\lambda_{1}^2
- 4\lambda_{3}\lambda_{1}^3 - 4\lambda_{2}^3\lambda_{0} - 4\lambda_{1}^2\lambda_{0}
- 27\lambda_{3}^2\lambda_{0}^2 \\
& & &\quad \textstyle + 12\lambda_{2}\lambda_{0}^2 + 2\wp_{13}\lambda_{2}\lambda_{1}^2
+ 16\wp_{13}\lambda_{0}^2 + \lambda_{1}^2\wp_{13}^2 - 4\wp_{12}\lambda_{1}\wp_{33}\lambda_{0} \\
& & &\quad \textstyle - 4\wp_{11}\wp_{33}\lambda_{1}^2 + 16\wp_{11}\wp_{33}\lambda_{2}\lambda_{0}
+ 4\wp_{11}^3 + 8\wp_{33}\wp_{22}\lambda_{0}^2
\end{align*}

\section{The three-term three-variable addition formula} \label{APP_34add}

The addition formula in Theorem \ref{thm:34_3t3v} is constructed using the following polynomials.

\begin{align*}
&P_{30}(\bu,\bv,\bw) = \textstyle
\frac{1}{8}\wp^{[2 2]}(\bu)\wp^{[1 1]}(\bv)\wp^{[2 2]}(\bw)
- \frac{1}{8}\wp_{112}(\bu)\wp^{[1 2]}(\bv)\partial_{3}Q_{1333}(\bw) \\
&\quad \textstyle - \frac{1}{48}\partial_{3}Q_{1333}(\bu)\wp^{[2 2]}(\bv)\partial_{3}Q_{1333}(\bw)
+ \frac{1}{8}\partial_{1}Q_{1333}(\bu)\wp^{[1 3]}(\bw)\wp_{133}(\bv) \\
&\quad \textstyle - \frac{3}{32}\wp_{112}(\bu)\wp_{112}(\bv)\wp^{[1 1]}(\bw)
- \frac{1}{32}\wp_{22}(\bv)\partial_{1}Q_{1333}(\bu)\partial_{1}Q_{1333}(\bw) \\
&\quad \textstyle + \frac{1}{24}\partial_{3}Q_{1333}(\bu)\wp^{[1 1]}(\bv)\wp_{111}(\bw)
+ \frac{3}{16}\wp_{112}(\bu)\wp^{[2 2]}(\bv)\wp_{222}(\bw)
+ \frac{3}{8}\wp^{[2 3]}(\bu)\wp_{22}(\bv)\wp^{[2 3]}(\bw) \\
&\quad \textstyle + \frac{1}{4}\wp_{333}(\bu)\wp_{111}(\bw)\wp^{[2 2]}(\bv)
+ \frac{1}{8}Q_{1333}(\bu)\wp^{[3 3]}(\bv)Q_{1333}(\bw)
+ \frac{3}{8}\wp_{112}(\bu)\wp_{13}(\bv)\wp_{112}(\bw) \\
&\quad \textstyle + \frac{1}{96}\partial_{2}Q_{1333}(\bu)\partial_{2}Q_{1333}(\bv)\wp^{[1 3]}(\bw)
- \frac{1}{8}\wp_{222}(\bu)\wp_{111}(\bv)\wp^{[1 2]}(\bw)\\
&\quad \textstyle + \frac{1}{16}\partial_{1}Q_{1333}(\bu)\wp_{12}(\bv)\partial_{2}Q_{1333}(\bw)
+ \frac{1}{8}\wp^{[1 3]}(\bu)\wp^{[1 3]}(\bv)\wp^{[1 3]}(\bw) \\
&\quad \textstyle + \frac{1}{8}\wp_{111}(\bu)\wp_{122}(\bw)\wp^{[1 1]}(\bv)
- \frac{3}{8}\wp^{[3 3]}(\bu)\wp_{123}(\bv)\wp_{123}(\bw)
+ \frac{3}{8}\wp^{[3 3]}(\bu)\wp_{223}(\bv)\wp_{113}(\bw) \\
&\quad \textstyle - \frac{3}{8}Q_{1333}(\bu)\wp_{113}(\bv)\wp_{113}(\bw)
- \frac{3}{16}\wp_{122}(\bu)\wp^{[2 2]}(\bw)\wp_{122}(\bv)
- \frac{3}{4}\wp^{[2 3]}(\bu)\wp_{12}(\bv)\wp^{[1 3]}(\bw) \\
&\quad \textstyle - \frac{1}{4}\partial_{1}Q_{1333}(\bu)\wp_{11}(\bv)\wp_{133}(\bw)
+ \wp^{[2 2]}(\bu)\wp^{[2 2]}(\bw)\wp_{13}(\bv)
+ \frac{3}{4}\wp^{[1 3]}(\bu)\wp_{11}(\bw)\wp^{[1 3]}(\bv) \\
\end{align*}
\begin{align*}
&\quad \textstyle + \frac{3}{8}\wp^{[3 3]}(\bu)\wp_{33}(\bw)\wp^{[3 3]}(\bv)
- \frac{1}{48}\partial_{2}Q_{1333}(\bu)\partial_{2}Q_{1333}(\bw)\wp_{11}(\bv) \\
&\quad \textstyle - \frac{1}{4}\wp^{[1 2]}(\bu)\wp^{[1 2]}(\bv)\wp^{[2 2]}(\bw)
- \frac{1}{6}\wp_{111}(\bv)\partial_{3}Q_{1333}(\bw)\wp_{13}(\bu)
+ \frac{1}{2}\wp_{11}(\bu)\wp_{11}(\bv)\wp_{11}(\bw) \\
&\quad \textstyle + \frac{1}{8}\wp^{[2 3]}(\bu)\partial_{2}Q_{1333}(\bv)\wp_{133}(\bw)
+ \frac{1}{4}\wp_{111}(\bu)\wp_{122}(\bw)\wp_{13}(\bv) - \frac{3}{8}\wp_{23}(\bu)\wp_{111}(\bw)\wp_{112}(\bv) \\
&\quad \textstyle - \frac{1}{8}\wp_{233}(\bu)\wp^{[2 3]}(\bv)\partial_{1}Q_{1333}(\bw)
+ \frac{3}{2}\wp^{[2 3]}(\bu)\wp_{11}(\bv)\wp_{12}(\bw) - \frac{1}{8}\wp_{111}(\bu)\wp_{111}(\bv)
\\ \\
&P_{27}(\bu,\bv,\bw) = \textstyle \Big[
\frac{3}{16}\wp_{12}(\bv)\partial_{1}Q_{1333}(\bu)\wp_{133}(\bw)
+ \frac{1}{8}\wp_{11}(\bw)\wp_{133}(\bv)d_{2}Q_{1333}(\bu) \\
&\quad \textstyle - \frac{1}{16}\wp^{[1 3]}(\bw)\wp_{133}(\bv)\partial_{2}Q_{1333}(\bu)
+ \frac{3}{8}\wp_{113}(\bw)\wp_{123}(\bv)Q_{1333}(\bu)
- \frac{3}{8}\wp^{[2 2]}(\bw)\wp^{[2 2]}(\bv)\wp_{23}(\bu) \\
&\quad \textstyle + \frac{1}{96}\partial_{3}Q_{1333}(\bw)\partial_{3}Q_{1333}(\bv)\wp^{[1 2]}(\bu)
- \frac{3}{16}\wp^{[2 2]}(\bv)\wp_{112}(\bw)\wp_{333}(\bu)
- \frac{3}{4}\wp^{[2 2]}(\bv)\wp^{[1 2]}(\bw)\wp_{13}(\bu) \\
&\quad \textstyle + \frac{9}{32}\wp_{23}(\bw)\wp_{112}(\bv)\wp_{112}(\bu)
+ \frac{1}{8}\wp^{[1 2]}(\bw)\wp^{[1 2]}(\bv)\wp^{[1 2]}(\bu)
+ \frac{3}{32}\wp_{122}(\bw)\wp_{122}(\bv)\wp^{[1 2]}(\bu)  \\
&\quad \textstyle - \frac{9}{16}\wp_{122}(\bv)\wp_{13}(\bw)\wp_{112}(\bu)
+ \frac{1}{16}\partial_{3}Q_{1333}(\bv)\wp^{[1 2]}(\bw)\wp_{122}(\bu)
- \frac{9}{8}\wp_{12}(\bu)\wp^{[2 3]}(\bw)\wp_{12}(\bv) \\
&\quad \textstyle - \frac{1}{8}\wp_{111}(\bv)\wp_{333}(\bw)\wp^{[1 2]}(\bu)
+ \frac{3}{16}\wp_{112}(\bv)\partial_{3}Q_{1333}(\bw)\wp_{13}(\bu) \Big]\lambda_{3} \\
\\
&P_{24}(\bu,\bv,\bw) = \textstyle \Big[
\frac{1}{4}\wp_{13}(\bu)\wp^{[1 1]}(\bv)\wp^{[2 2]}(\bw)
- \frac{1}{4}8\wp_{222}(\bu)\wp^{[1 2]}(\bv)\partial_{3}Q_{1333}(\bw) \\
&\quad \textstyle - \frac{1}{8}\partial_{1}Q_{1333}(\bu)\wp_{12}(\bv)\wp_{233}(\bw)
- \frac{3}{16}\wp_{122}(\bu)\wp_{112}(\bv)\wp_{23}(\bw) + \frac{3}{8}\wp^{[1 2]}(\bu)\wp_{333}(\bw)\wp_{112}(\bv) \\
&\quad \textstyle + \frac{3}{8}\wp_{12}(\bu)\wp_{12}(\bv)\wp^{[1 3]}(\bw)
+ \frac{3}{16}\wp_{222}(\bu)\wp_{112}(\bw)\wp_{13}(\bv)
+ \frac{3}{4}\wp^{[1 2]}(\bv)\wp_{23}(\bu)\wp^{[2 2]}(\bw)  \\
&\quad \textstyle - \frac{1}{6}\partial_{3}Q_{1333}(\bu)\wp_{122}(\bv)\wp_{13}(\bw)
- \frac{3}{8}\wp_{133}(\bu)\wp_{133}(\bv)\wp^{[1 3]}(\bw) + \frac{5}{16}\wp_{122}(\bu)\wp_{122}(\bv)\wp_{13}(\bw) \\
&\quad \textstyle - \frac{1}{16}\partial_{3}Q_{1333}(\bu)\wp_{112}(\bv)\wp_{23}(\bw)
+ \frac{3}{4}\wp_{133}(\bu)\wp_{133}(\bw)\wp_{11}(\bv) - \frac{1}{4}Q_{1333}(\bu)\wp_{123}(\bv)\wp_{123}(\bw) \\
&\quad \textstyle - \frac{1}{48}\wp^{[1 1]}(\bv)\partial_{3}Q_{1333}(\bw)\wp_{122}(\bu)
+ \frac{9}{8}\wp_{113}(\bu)\wp_{33}(\bw)\wp_{113}(\bv) + \frac{3}{4}\wp_{12}(\bv)\wp^{[2 3]}(\bw)\wp_{22}(\bu) \\
&\quad \textstyle - \frac{1}{4}\wp_{133}(\bu)\wp_{12}(\bv)\partial_{2}Q_{1333}(\bw)
- \frac{1}{16}\wp_{122}(\bu)\wp^{[1 2]}(\bv)\wp_{222}(\bw)
+ \frac{1}{4}\wp_{13}(\bv)\wp^{[1 2]}(\bu)\wp^{[1 2]}(\bw) \\
&\quad \textstyle - \frac{1}{4}\wp_{33}(\bu)\wp^{[3 3]}(\bw)Q_{1333}(\bv)
- \frac{1}{8}\wp_{223}(\bu)\wp_{113}(\bw)Q_{1333}(\bv)
- \frac{1}{32}\wp_{122}(\bu)\wp_{122}(\bv)\wp^{[1 1]}(\bw) \\
&\quad \textstyle - \frac{3}{4}\wp_{12}(\bu)\wp_{12}(\bv)\wp_{11}(\bw)
+ \frac{1}{4}\wp_{333}(\bu)\wp_{111}(\bv)\wp_{13}(\bw) - \frac{1}{8}\wp^{[1 1]}(\bv)\wp^{[1 2]}(\bu)\wp^{[1 2]}(\bw)  \\
&\quad \textstyle - \frac{1}{444}\wp_{13}(\bu)\partial_{3}Q_{1333}(\bv)\partial_{3}Q_{1333}(\bw)
- \frac{1}{288}\wp^{[1 1]}(\bu)\partial_{3}Q_{1333}(\bv)\partial_{3}Q_{1333}(\bw) \\
&\quad \textstyle + \frac{1}{24}\partial_{3}Q_{1333}(\bu)\wp_{333}(\bv)\wp^{[2 2]}(\bw)
- \frac{5}{2}\wp_{13}(\bu)\wp_{13}(\bv)\wp^{[2 2]}(\bw)
+ \frac{1}{8}\wp_{333}(\bu)\wp^{[2 2]}(\bv)\wp_{122}(\bw) \\
&\quad \textstyle
- \frac{1}{4}\wp^{[2 2]}(\bu)\wp^{[2 2]}(\bv)
+ \frac{3}{16}\wp_{112}(\bu)\wp_{112}(\bv)
\Big] \lambda_2
+  \Big[
- \frac{9}{32}\wp_{113}(\bv)\wp_{33}(\bu)\wp_{113}(\bw) \\
&\quad \textstyle
- \frac{3}{16}\wp_{133}(\bu)\wp_{133}(\bw)\wp_{11}(\bv)
+ \frac{3}{32}\wp_{133}(\bu)\wp_{133}(\bv)\wp^{[1 3]}(\bw)
+ \frac{9}{8}\wp_{13}(\bu)\wp_{13}(\bv)\wp^{[2 2]}(\bw)
\Big] \lambda_3^2 \\
\\
&P_{21}(\bu,\bv,\bw) = \textstyle
\Big[
\frac{3}{8}\wp_{233}(\bu)\wp^{[1 3]}(\bw)\wp_{133}(\bv)
- \frac{9}{4}\wp_{113}(\bu)\wp_{123}(\bw)\wp_{33}(\bv)
+ \frac{9}{8}\wp_{12}(\bu)\wp_{12}(\bv)\wp_{12}(\bw) \\
&\quad \textstyle + \frac{11}{4}\wp^{[2 2]}(\bu)\wp_{23}(\bv)\wp_{13}(\bw)
- \frac{1}{4}\wp^{[1 1]}(\bu)\wp_{13}(\bw)\wp^{[1 2]}(\bv)
- \frac{1}{2}\wp^{[2 2]}(\bu)\wp_{23}(\bv)\wp^{[1 1]}(\bw) \\
&\quad \textstyle - \frac{3}{8}\wp^{[2 3]}(\bu)\wp_{22}(\bv)\wp_{22}(\bw)
+ \frac{3}{32}\wp_{122}(\bu)\wp_{122}(\bw)\wp_{23}(\bv)
- \frac{3}{4}\wp^{[1 3]}(\bu)\wp_{22}(\bv)\wp_{12}(\bw) \\
&\quad \textstyle + \frac{1}{48}\partial_{3}Q_{1333}(\bu)\wp^{[1 1]}(\bw)\wp_{222}(\bv)
+ \frac{3}{16}\partial_{3}Q_{1333}(\bu)\wp_{23}(\bv)\wp_{122}(\bw) \\
&\quad \textstyle + \frac{1}{96}\partial_{3}Q_{1333}(\bu)\wp_{23}(\bw)\partial_{3}Q_{1333}(\bv)
- \frac{9}{16}\wp_{112}(\bv)\wp_{122}(\bw)
- \frac{3}{16}\wp^{[1 1]}(\bu)\wp_{112}(\bw)\wp_{333}(\bv) \\
&\quad \textstyle + \frac{1}{8}\wp_{22}(\bu)\partial_{2}Q_{1333}(\bw)\wp_{133}(\bv)
+ \frac{1}{8}\wp^{[1 1]}(\bu)\wp^{[1 1]}(\bw)\wp^{[1 2]}(\bv) \\
&\quad \textstyle - \frac{3}{4}\wp_{11}(\bu)\wp_{233}(\bw)\wp_{133}(\bv)
- \frac{1}{8}\wp_{333}(\bu)\wp_{111}(\bv)\wp_{23}(\bw)
- \frac{1}{16}\wp_{112}(\bu)\partial_{3}Q_{1333}(\bv) \\
&\quad \textstyle - \frac{5}{8}\wp_{23}(\bu)\wp^{[1 2]}(\bv)\wp^{[1 2]}(\bw)
+ \frac{1}{16}\wp_{233}(\bu)\partial_{1}Q_{1333}(\bw)\wp_{22}(\bv) \\
&\quad \textstyle - \frac{3}{8}\wp_{333}(\bu)\wp_{122}(\bv)\wp^{[1 2]}(\bw)
+ \frac{3}{16}\wp_{233}(\bu)\partial_{2}Q_{1333}(\bw)\wp_{12}(\bv)
+ \frac{5}{4}\wp^{[2 2]}(\bv)\wp^{[1 2]}(\bw) \\
&\quad \textstyle - \frac{3}{8}\wp_{112}(\bv)\wp_{13}(\bw)\wp_{333}(\bu)
- \frac{7}{16}\wp_{222}(\bu)\wp_{13}(\bw)\wp_{122}(\bv)
+ \frac{7}{2}\wp^{[1 2]}(\bu)\wp_{13}(\bv)\wp_{13}(\bw) \\
\end{align*}
\begin{align*}
&\quad \textstyle + \frac{3}{2}\wp_{22}(\bv)\wp_{11}(\bu)\wp_{12}(\bw)
+ \frac{1}{16}\wp_{222}(\bu)\wp^{[1 1]}(\bw)\wp_{122}(\bv)
- \frac{9}{8}\wp_{12}(\bu)\wp_{133}(\bv)\wp_{133}(\bw) \\
&\quad \textstyle - \frac{1}{16}\wp^{[2 2]}(\bu)\wp_{333}(\bv)\wp_{222}(\bw)
+ \frac{5}{48}\wp_{222}(\bu)\wp_{13}(\bw)\partial_{3}Q_{1333}(\bv) \\
&\quad \textstyle + \frac{3}{8}Q_{1333}(\bu)\wp_{123}(\bv)\wp_{223}(\bw)
+ \frac{1}{32}\wp_{222}(\bu)\wp_{222}(\bw)\wp^{[1 2]}(\bv)
\Big] \lambda_{1} \\
&\quad \textstyle -  \Big[ \frac{3}{4}\wp^{[2 2]}(\bu)\wp_{23}(\bv)\wp_{13}(\bw)
+ \frac{3}{4}\wp_{13}(\bu)\wp_{13}(\bv)\wp^{[1 2]}(\bw)
+ \frac{1}{16}\wp^{[1 2]}(\bu)\wp_{122}(\bw)\wp_{333}(\bv) \\
&\quad \textstyle + \frac{3}{16}\wp_{112}(\bv)\wp_{13}(\bu)\wp_{333}(\bw)
+ \frac{1}{48}\partial_{3}Q_{1333}(\bu)\wp_{333}(\bv)\wp^{[1 2]}(\bw)
\Big] \lambda_2\lambda_3 \\
\\
&P_{18}(\bu,\bv,\bw) = \textstyle
\Big[
- \frac{3}{4}\wp^{[1 2]}(\bv)\wp^{[1 2]}(\bw)
- \frac{9}{4}\wp^{[2 2]}(\bu)\wp^{[1 1]}(\bv)
+ \frac{9}{16}\wp_{112}(\bv)\wp_{222}(\bu) \\
&\quad \textstyle - 6\wp_{13}(\bu)\wp_{23}(\bv)\wp^{[1 2]}(\bw)
+ \wp_{333}(\bu)\wp_{13}(\bv)\wp_{122}(\bw)
+ 2\wp_{13}(\bu)\wp_{13}(\bv)\wp_{13}(\bw) \\
&\quad \textstyle - \frac{27}{8}\wp_{22}(\bu)\wp_{12}(\bv)\wp_{12}(\bw)
+ \frac{1}{3}\wp_{13}(\bu)\partial_{3}Q_{1333}(\bv)\wp_{333}(\bw)
+ \frac{9}{16}\wp_{122}(\bu)\wp_{122}(\bw) \\
&\quad \textstyle - \frac{1}{16}\partial_{3}Q_{1333}(\bu)\partial_{3}Q_{1333}(\bw)
+ \frac{9}{4}\wp_{233}(\bu)\wp_{133}(\bw)\wp_{12}(\bv)
- \frac{9}{4}\wp^{[2 2]}(\bu)\wp_{23}(\bv)\wp_{23}(\bw) \\
&\quad \textstyle - \frac{3}{8}\wp_{233}(\bu)\wp_{233}(\bw)\wp^{[1 3]}(\bv)
+ \frac{1}{4}\wp^{[2 2]}(\bu)\wp_{333}(\bv)\wp_{333}(\bw)
+ \frac{3}{4}\wp_{112}(\bu)\wp_{23}(\bw)\wp_{333}(\bv)\\
&\quad \textstyle - \frac{3}{2}\wp_{33}(\bu)\wp^{[3 3]}(\bv)\wp_{33}(\bw)
+ \frac{3}{4}\wp_{233}(\bu)\wp_{233}(\bv)\wp_{11}(\bw)
- \frac{3}{8}\wp_{223}(\bu)\wp_{223}(\bv)Q_{1333}(\bw) \\
&\quad \textstyle - \frac{3}{32}\wp^{[1 1]}(\bu)\wp_{222}(\bv)\wp_{222}(\bw)
+ \frac{3}{8}\wp_{13}(\bu)\wp_{222}(\bv)\wp_{222}(\bw)
+ 3\wp_{123}(\bu)\wp_{123}(\bv)\wp_{33}(\bw) \\
&\quad \textstyle + \frac{9}{8}\wp^{[1 3]}(\bu)\wp_{22}(\bv)\wp_{22}(\bw)
- \frac{3}{8}\wp_{222}(\bu)\wp_{23}(\bv)\partial_{3}Q_{1333}(\bw)
- \frac{9}{4}\wp_{11}(\bu)\wp_{22}(\bv)\wp_{22}(\bw) \\
&\quad \textstyle - 3\wp^{[1 1]}(\bu)\wp_{13}(\bv)\wp_{13}(\bw)
+ \frac{1}{2}\wp_{122}(\bu)\wp_{333}(\bv)\wp^{[1 1]}(\bw)
- \frac{1}{8}\wp^{[1 1]}(\bu)\wp^{[1 1]}(\bv)\wp^{[1 1]}(\bw) \\
&\quad \textstyle - \frac{1}{12}\wp^{[1 1]}(\bu)\wp_{333}(\bv)\partial_{3}Q_{1333}(\bw)
+ \frac{3}{2}\wp_{223}(\bu)\wp_{113}(\bv)\wp_{33}(\bw)  \\
&\quad \textstyle - \frac{3}{8}\wp_{233}(\bu)\wp_{22}(\bw)\partial_{2}Q_{1333}(\bv)
+ \frac{1}{2}Q_{1333}(\bu)Q_{1333}(\bw)\wp_{33}(\bv)
- \frac{1}{4}\wp_{111}(\bu)\wp_{333}(\bw) \\
&\quad \textstyle + \frac{9}{8}\wp_{133}(\bu)\wp_{133}(\bv)\wp_{22}(\bw)
+ \frac{1}{4}\wp^{[1 2]}(\bu)\wp_{333}(\bv)\wp_{222}(\bw)
+ \frac{3}{2}\wp^{[1 1]}(\bu)\wp_{23}(\bw)\wp^{[1 2]}(\bv)
\Big] \lambda_{0} \\
&\quad \textstyle +  \Big[
\frac{3}{8}\wp_{22}(\bu)\wp_{12}(\bv)\wp_{12}(\bw)
+ \wp_{13}(\bu)\wp_{13}(\bv)\wp_{13}(\bw)
- \frac{1}{24}\wp_{122}(\bu)\partial_{3}Q_{1333}(\bw) \\
&\quad \textstyle+ \frac{1}{24}Q_{1333}(\bu)Q_{1333}(\bv)\wp_{33}(\bw)
+ \frac{3}{4}\wp_{13}(\bu)\wp_{23}(\bv)\wp^{[1 2]}(\bw)
+ \frac{1}{4}\wp^{[1 2]}(\bu)\wp^{[1 2]}(\bw) \\
&\quad \textstyle+ \frac{1}{8}\wp_{333}(\bu)\wp_{13}(\bv)\wp_{122}(\bw)
+ \frac{1}{8}\wp_{13}(\bu)\wp^{[1 1]}(\bw)\wp_{13}(\bv)
+ \frac{1}{24}\wp_{13}(\bu)\partial_{3}Q_{1333}(\bv)\wp_{333}(\bw) \\
&\quad \textstyle- \frac{1}{2}\wp^{[2 2]}(\bu)\wp_{13}(\bv)
- \frac{1}{444}\partial_{3}Q_{1333}(\bu)\partial_{3}Q_{1333}(\bw)
- \frac{1}{16}\wp_{122}(\bu)\wp_{122}(\bw)
\Big] \lambda_2^2 \\
&\quad \textstyle +  \Big[
\frac{3}{4}\wp^{[2 2]}(\bu)\wp_{13}(\bv)
+ \frac{3}{16}\wp_{122}(\bu)\wp_{122}(\bv)
- \frac{1}{16}\wp_{333}(\bu)\wp^{[2 2]}(\bv)\wp_{333}(\bw)  \\
&\quad \textstyle- \frac{3}{4}\wp_{13}(\bu)\wp_{23}(\bv)\wp^{[1 2]}(\bw)
+ \frac{1}{8}\wp_{13}(\bu)\wp_{333}(\bw)\wp_{122}(\bv)
+ \frac{1}{8}\partial_{3}Q_{1333}(\bu)\wp_{122}(\bw) \\
&\quad \textstyle+ \frac{1}{48}\partial_{3}Q_{1333}(\bv)\partial_{3}Q_{1333}(\bw)
- \frac{1}{12}\partial_{3}Q_{1333}(\bu)\wp_{333}(\bw)\wp_{13}(\bv)
- \frac{3}{4}\wp^{[1 2]}(\bv)\wp^{[1 2]}(\bw) \\
&\quad \textstyle+ \frac{3}{8}\wp_{123}(\bv)\wp_{123}(\bw)\wp_{33}(\bu)
+ \frac{1}{16}\wp_{333}(\bu)\wp^{[1 2]}(\bw)\wp_{222}(\bv) \\
&\quad \textstyle- \frac{9}{4}\wp_{13}(\bu)\wp_{13}(\bv)\wp_{13}(\bw)
+ \frac{3}{16}\wp_{113}(\bu)\wp_{33}(\bv)\wp_{223}(\bw)
+ \frac{3}{8}\wp_{23}(\bu)\wp^{[2 2]}(\bw)\wp_{23}(\bv) \\
&\quad \textstyle+ \frac{3}{16}\wp_{23}(\bu)\wp_{112}(\bv)\wp_{333}(\bw)
- \frac{1}{8}\wp_{33}(\bu)Q_{1333}(\bv)Q_{1333}(\bw) \\
&\quad \textstyle+ \frac{9}{16}\wp_{233}(\bu)\wp_{133}(\bw)\wp_{12}(\bv)
- \frac{9}{8}\wp_{12}(\bu)\wp_{12}(\bv)\wp_{22}(\bw)
+ \frac{3}{8}\wp_{33}(\bu)\wp^{[3 3]}(\bw)\wp_{33}(\bv) \\
&\quad \textstyle+ \frac{1}{16}\wp_{333}(\bu)\wp_{122}(\bv)\wp^{[1 1]}(\bw)
+ \frac{1}{4}8\partial_{3}Q_{1333}(\bu)\wp_{333}(\bw)\wp^{[1 1]}(\bv)
\Big] \lambda_3\lambda_1 \\
\\
&P_{15}(\bu,\bv,\bw) = \textstyle
\Big[
- \frac{9}{4}\wp_{13}(\bu)\wp^{[1 2]}(\bw)
- \frac{9}{16}\wp_{222}(\bu)\wp_{122}(\bv)
- \frac{9}{8}\wp_{133}(\bu)\wp_{22}(\bw)\wp_{233}(\bv) \\
&\quad \textstyle + \frac{9}{8}\wp^{[1 2]}(\bu)\wp_{23}(\bv)\wp_{23}(\bw)
- \frac{1}{8}\wp_{333}(\bu)\wp_{333}(\bv)\wp^{[1 2]}(\bw)
+ \frac{9}{16}\wp_{333}(\bu)\wp_{112}(\bv) \\
&\quad \textstyle + \frac{9}{4}\wp^{[1 2]}(\bu)\wp^{[1 1]}(\bw)
- \frac{9}{8}\wp_{333}(\bu)\wp_{23}(\bw)\wp_{122}(\bv)
+ \frac{27}{8}\wp_{22}(\bu)\wp_{22}(\bv)\wp_{12}(\bw) \\
&\quad \textstyle- \frac{3}{16}\wp_{222}(\bu)\partial_{3}Q_{1333}(\bw)
+ \frac{9}{2}\wp_{13}(\bu)\wp_{13}(\bv)\wp_{23}(\bw)
- \frac{3}{8}\wp_{333}(\bu)\wp_{13}(\bw)\wp_{222}(\bv) \\
&\quad \textstyle- \frac{9}{8}\wp_{233}(\bu)\wp_{233}(\bv)\wp_{12}(\bw)
- \frac{9}{4}\wp_{33}(\bu)\wp_{223}(\bw)\wp_{123}(\bv) \\
\end{align*}
\begin{align*}
&\quad \textstyle- \frac{3}{16}\wp_{222}(\bu)\wp^{[1 1]}(\bv)\wp_{333}(\bw)
+ \frac{9}{4}\wp_{23}(\bu)\wp^{[1 1]}(\bw)\wp_{13}(\bv)
\Big] \lambda_3\lambda_0 \\
&\quad \textstyle +  \Big[
\frac{3}{4}\wp_{13}(\bu)\wp^{[1 2]}(\bw)
- \frac{1}{2}\wp^{[1 1]}(\bu)\wp_{13}(\bw)\wp_{23}(\bv)
- \frac{1}{16}\wp_{122}(\bu)\wp_{23}(\bw)\wp_{333}(\bv) \\
&\quad \textstyle + \frac{1}{4}\wp^{[2 2]}(\bu)\wp_{23}(\bv)
- \frac{1}{48}\wp_{23}(\bv)\wp_{333}(\bw)\partial_{3}Q_{1333}(\bu)
+ \frac{1}{16}\wp_{222}(\bu)\wp_{122}(\bv) \\
&\quad \textstyle- \frac{3}{8}\wp_{23}(\bu)\wp_{23}(\bv)\wp^{[1 2]}(\bw)
- \frac{7}{4}\wp_{13}(\bu)\wp_{13}(\bv)\wp_{23}(\bw)
- \frac{1}{16}\wp_{13}(\bv)\wp_{222}(\bw)\wp_{333}(\bu) \\
&\quad \textstyle - \frac{1}{4}\wp^{[1 2]}(\bu)\wp^{[1 1]}(\bv)
+ \frac{1}{48}\wp_{222}(\bv)\partial_{3}Q_{1333}(\bw)
- \frac{3}{8}\wp_{22}(\bu)\wp_{22}(\bv)\wp_{12}(\bw)
\Big] \lambda_2\lambda_1\\
&\quad \textstyle+  \Big[
\frac{9}{8}\wp_{13}(\bu)\wp_{13}(\bv)\wp_{23}(\bw)
+ \frac{1}{3}2\wp_{333}(\bu)\wp_{333}(\bw)\wp^{[1 2]}(\bv)
\Big] \lambda_3^2\lambda_1 \\
\\
&P_{12}(\bu,\bv,\bw) = \textstyle
\Big[
\frac{9}{2}\wp_{13}(\bv)\wp_{13}(\bw)
- \frac{3}{4}\wp^{[1 1]}(\bu)\wp^{[1 1]}(\bv)
- \frac{5}{4}\wp_{333}(\bu)\wp_{13}(\bw)\wp_{333}(\bv) \\
&\quad \textstyle + \frac{3}{16}\wp_{222}(\bu)\wp_{222}(\bv)
+ \frac{9}{4}\wp_{23}(\bu)\wp_{13}(\bw)\wp_{23}(\bv)
- \frac{9}{8}\wp_{22}(\bu)\wp_{22}(\bv)\wp_{22}(\bw) \\
&\quad \textstyle+ \frac{3}{8}\wp_{333}(\bu)\wp^{[1 1]}(\bw)\wp_{333}(\bv)
+ \frac{3}{4}\wp^{[1 1]}(\bw)\wp_{13}(\bu)
+ \frac{5}{8}\partial_{3}Q_{1333}(\bv)\wp_{333}(\bw) \\
&\quad \textstyle
+ \frac{3}{8}\wp_{122}(\bv)\wp_{333}(\bw)
- \frac{9}{8}\wp_{23}(\bu)\wp^{[1 1]}(\bv)\wp_{23}(\bw)
+ \frac{9}{8}\wp_{233}(\bu)\wp_{22}(\bw)\wp_{233}(\bv) \\
&\quad \textstyle- \frac{3}{4}\wp_{23}(\bu)\wp^{[1 2]}(\bw)
+ \frac{9}{8}\wp_{333}(\bu)\wp_{222}(\bv)\wp_{23}(\bw)
+ \frac{9}{8}\wp_{223}(\bu)\wp_{33}(\bw)\wp_{223}(\bv) \\
&\quad \textstyle
- 4\wp_{33}(\bu)\wp_{33}(\bv)Q_{1333}(\bw)
\Big] \lambda_2\lambda_0 +  \Big[
\frac{9}{8}\wp_{33}(\bu)\wp_{33}(\bw)Q_{1333}(\bv)
- \frac{9}{16}\wp_{122}(\bu)\wp_{333}(\bw) \\
&\quad \textstyle - \frac{3}{16}\partial_{3}Q_{1333}(\bv)\wp_{333}(\bw)
- \frac{27}{8}\wp_{23}(\bu)\wp_{13}(\bw)\wp_{23}(\bv) \\
&\quad \textstyle- \frac{3}{32}\wp_{333}(\bu)\wp^{[1 1]}(\bw)\wp_{333}(\bv)
+ \frac{3}{8}\wp_{333}(\bu)\wp_{13}(\bw)\wp_{333}(\bv)
\Big] \lambda_3^2\lambda_0\\
&\quad \textstyle+  \Big[
- \frac{9}{32}\wp_{33}(\bv)\wp_{223}(\bu)\wp_{223}(\bw)
- \frac{1}{16}\wp_{222}(\bu)\wp_{222}(\bw) - \frac{25}{8}\wp_{13}(\bv)\wp_{13}(\bw) \\
&\quad \textstyle+ \frac{5}{8}\wp_{23}(\bu)\wp_{13}(\bv)\wp_{23}(\bw)
- \frac{3}{32}\wp_{333}(\bu)\wp^{[1 1]}(\bv)\wp_{333}(\bw)
+ \frac{3}{8}\wp_{333}(\bu)\wp_{13}(\bv)\wp_{333}(\bw) \\
&\quad \textstyle- \frac{3}{16}\wp_{122}(\bv)\wp_{333}(\bw) + \frac{1}{8}\wp^{[2 2]}(\bw)
- \frac{9}{32}\wp_{22}(\bu)\wp_{233}(\bv)\wp_{233}(\bw) \\
&\quad \textstyle+ \frac{9}{8}\wp_{33}(\bu)\wp_{33}(\bw)Q_{1333}(\bv)
- \frac{3}{16}\partial_{3}Q_{1333}(\bv)\wp_{333}(\bw)
+ \frac{3}{8}\wp_{22}(\bu)\wp_{22}(\bv)\wp_{22}(\bw) \\
&\quad \textstyle+ \frac{1}{2}\wp_{23}(\bv)\wp^{[1 1]}(\bu)\wp_{23}(\bw)
+ \frac{1}{4}\wp^{[1 1]}(\bv)\wp^{[1 1]}(\bw)
- \frac{1}{2}\wp^{[1 1]}(\bv)\wp_{13}(\bw) \\
&\quad \textstyle - \frac{1}{4}\wp_{23}(\bu)\wp_{222}(\bv)\wp_{333}(\bw)
- \frac{1}{4}\wp_{23}(\bu)\wp^{[1 2]}(\bv)
\Big] \lambda_1^2
+  \Big[
\frac{3}{4}\wp_{13}(\bv)\wp_{13}(\bw) \\
&\quad \textstyle+ \frac{1}{48}\wp_{333}(\bu)\partial_{3}Q_{1333}(\bv)
- \frac{1}{8}\wp_{33}(\bu)\wp_{33}(\bw)Q_{1333}(\bv)
+ \frac{1}{16}\wp_{122}(\bu)\wp_{333}(\bw) \\
&\quad \textstyle + \frac{3}{8}\wp_{23}(\bu)\wp_{13}(\bw)\wp_{23}(\bv)
- \frac{1}{16}\wp_{333}(\bu)\wp_{13}(\bw)\wp_{333}(\bv)
\Big] \lambda_3\lambda_2\lambda_1
-  \frac{1}{4}\wp_{13}(\bv)\wp_{13}(\bw)\lambda_{2}^3 \\
\\
&P_{9}(\bu,\bv,\bw) = \textstyle
\frac{1}{32}\wp_{23}(\bu)\wp_{333}(\bv)\wp_{333}(\bw)\lambda_{3}\lambda_{1}^2
- \frac{3}{8}\wp_{23}(\bu)\wp_{23}(\bv)\wp_{23}(\bw)\lambda_{3}\lambda_{1}^2\\
&\quad \textstyle - \frac{9}{2}\wp_{13}(\bv)\wp_{23}(\bw)\lambda_{3}\lambda_{2}\lambda_{0}
- \frac{9}{8}\wp_{23}(\bu)\wp_{23}(\bv)\wp_{23}(\bw)\lambda_{1}\lambda_{0}
- \frac{1}{8}\wp_{222}(\bv)\wp_{333}(\bw)\lambda_{3}\lambda_{1}^2 \\
&\quad \textstyle- \frac{1}{8}\wp_{23}(\bu)\wp_{333}(\bv)\wp_{333}(\bw)\lambda_{1}\lambda_{0}
+ \frac{3}{8}\wp_{333}(\bu)\wp_{222}(\bv)\lambda_{3}\lambda_{2}\lambda_{0}
- \frac{3}{4}\wp^{[1 2]}(\bv)\lambda_{1}\lambda_{0} \\
&\quad \textstyle+ \frac{3}{4}\wp^{[1 1]}(\bv)\wp_{23}(\bw)\lambda_{1}\lambda_{0}
+ \frac{9}{8}\wp_{23}(\bu)\wp_{23}(\bv)\wp_{23}(\bw)\lambda_{3}\lambda_{2}\lambda_{0}
+ \frac{3}{4}\wp_{13}(\bw)\wp_{23}(\bu)\lambda_{3}\lambda_{1}^2 \\
&\quad \textstyle + \frac{1}{4}\wp_{13}(\bu)\wp_{23}(\bv)\lambda_{2}^2\lambda_{1}
+ \frac{33}{4}\wp_{13}(\bw)\wp_{23}(\bv)\lambda_{1}\lambda_{0}
+ \frac{3}{16}\wp_{222}(\bu)\wp_{333}(\bv)\lambda_{1}\lambda_{0} \\
\\
&P_{6}(\bu,\bv,\bw) = \textstyle 9\wp_{13}(\bw)\lambda_{0}^2
+ \frac{3}{8}\wp_{33}(\bu)\wp_{33}(\bv)\wp_{33}(\bw)\lambda_{3}^2\lambda_{1}^2
+ \frac{3}{16}\wp_{333}(\bv)\wp_{333}(\bw)\lambda_{3}^2\lambda_{2}\lambda_{0}\\
&\quad \textstyle + \frac{9}{8}\wp^{[1 1]}(\bu)\lambda_{0}^2
+ \frac{9}{2}\wp_{33}(\bu)\wp_{33}(\bv)\wp_{33}(\bw)\lambda_{2}^2\lambda_{0}
+ \frac{9}{16}\wp_{333}(\bv)\wp_{333}(\bw)\lambda_{3}\lambda_{1}\lambda_{0}\\
&\quad \textstyle- \frac{9}{8}\wp_{23}(\bu)\wp_{23}(\bv)\lambda_{3}\lambda_{1}\lambda_{0}
- \frac{27}{4}\wp_{23}(\bv)\wp_{23}(\bw)\lambda_{0}^2
- \frac{3}{2}\wp_{13}(\bw)\lambda_{2}^2\lambda_{0} \\
&\quad \textstyle- \frac{3}{4}\wp_{333}(\bu)\wp_{333}(\bv)\lambda_{2}^2\lambda_{0}
+ \frac{1}{2}\wp_{13}(\bw)\lambda_{2}\lambda_{1}^2
- \frac{5}{4}\wp_{333}(\bv)\wp_{333}(\bw)\lambda_{0}^2 \\
&\quad \textstyle+ 6\wp_{33}(\bu)\wp_{33}(\bv)\wp_{33}(\bw)\lambda_{0}^2
- \frac{9}{8}\wp_{33}(\bu)\wp_{33}(\bv)\wp_{33}(\bw)\lambda_{2}\lambda_{1}^2
+ \frac{9}{4}\wp_{23}(\bv)\wp_{23}(\bw)\lambda_{2}^2\lambda_{0} \\
&\quad \textstyle- \frac{9}{8}\wp_{33}(\bu)\wp_{33}(\bv)\wp_{33}(\bw)\lambda_{3}^2\lambda_{2}\lambda_{0}
- 3\wp_{33}(\bu)\wp_{33}(\bv)\wp_{33}(\bw)\lambda_{3}\lambda_{1}\lambda_{0}\\
&\quad \textstyle - \frac{1}{16}\wp_{333}(\bu)\wp_{333}(\bv)\lambda_{3}^2\lambda_{1}^2
- \frac{5}{8}\wp_{23}(\bu)\wp_{23}(\bv)\lambda_{2}\lambda_{1}^2
+ \frac{3}{16}\wp_{333}(\bu)\wp_{333}(\bv)\lambda_{2}\lambda_{1}^2 \\
\end{align*}

\begin{align*}
&P_{3}(\bu,\bv,\bw) = \textstyle
- \frac{11}{8}\wp_{23}(\bw)\lambda_{1}^3
- \frac{81}{8}\wp_{23}(\bv)\lambda_{3}\lambda_{0}^2
+ \frac{21}{4}\wp_{23}(\bw)\lambda_{2}\lambda_{1}\lambda_{0} \\
\\
&P_{0}(\bu,\bv,\bw) = \textstyle
\frac{9}{4}\lambda_{3}\lambda_{2}\lambda_{1}\lambda_{0}
+ \frac{1}{8}\lambda_{2}^2\lambda_{1}^2
- \frac{1}{2}\lambda_{3}\lambda_{1}^3
- \frac{1}{2}\lambda_{2}^3\lambda_{0}
- \frac{3}{4}\lambda_{1}^2\lambda_{0}
- \frac{27}{8}\lambda_{3}^2\lambda_{0}^2
+ \frac{9}{4}\lambda_{2}\lambda_{0}^2
\end{align*}

\subsection*{Acknowledgments}

It is a pleasure to thank to Chris Athorne, Victor Enolski, John
Gibbons, Shigeki Matsutani and Atsushi Nakayashiki for useful
conversations.  Part of the work was done while Y.O. was visiting Heriot-Watt University, under the financial support of the Edinburgh Mathematical Society and JSPS grant-in-aid for scientific research no. 19540002.

\bibliography{Genus_Three}{}
\bibliographystyle{plain}

\subsection*{Authors}

\noindent J.~C.~Eilbeck \\
Department of Mathematics and the Maxwell Institute for Mathematical Sciences \\
Heriot-Watt University, Edinburgh \\ EH14 4AS, UK \\
J.C.Eilbeck@hw.ac.uk

\bigskip

\noindent M.~England [\textbf{Corresponding Author}] \\
School of Mathematics and Statistics, \\ University of Glasgow, \\
Glasgow G12 8QQ, UK. \\
Matthew.England@glasgow.ac.uk

\bigskip

\noindent Y.~\^Onishi \\
4-3-11, Takeda, Kofu, 400-8511, \\
Faculty of Education and Human Sciences, \\
University of Yamanashi, Japan. \\
yonishi@yamanashi.ac.jp

\end{document}